\documentclass[a4paper]{article}
\usepackage[utf8]{inputenc}
\usepackage[T1]{fontenc}
\usepackage[english]{babel}
\usepackage{amsmath}
\usepackage{amssymb}
\usepackage{amsthm}
\usepackage{graphicx}
\usepackage{xcolor}
\usepackage{lmodern}
\usepackage{caption}
\usepackage{subcaption}

\usepackage[a4paper]{geometry}
\theoremstyle{plain}

\newtheorem{lem}{Lemma}
\newtheorem{cor}{Corollary}
\newtheorem{prop}{Proposition}

\theoremstyle{definition}
\newtheorem{exm}{Example}
\newtheorem{dfn}{Definition}

\theoremstyle{remark}
\newtheorem{rmk}{Remark}
\newtheorem{conv}{Convention}

\renewcommand{\vec}{\mathbf}
\newcommand{\bs}{\boldsymbol}
\newcommand{\ds}{\displaystyle}
\newcommand{\IR}{\mathbb{R}}

\newcommand{\bmat}[1]{\begin{bmatrix} #1 \end{bmatrix}}

\renewcommand{\d}{\mathrm{d}}

\title{Weak Form of Stokes-Dirac Structures and Geometric Discretization of Port-Hamiltonian Systems\footnote{Accepted version of P. Kotyczka et al., Weak form of Stokes–Dirac structures and geometric discretization of port-Hamiltonian systems, J. Comput. Phys. 361 (2018) 442--476, https://doi.org/10.1016/j.jcp.2018.02.006.\newline  \copyright\; 2018. This manuscript version is made available under the CC-BY-NC-ND 4.0 license \newline http://creativecommons.org/licenses/by-nc-nd/4.0/.}}

\author{Paul Kotyczka\footnote{Univ Lyon, Université Claude Bernard Lyon 1, CNRS, LAGEP UMR 5007, 43 Boulevard du 11 Novembre 1918, 69622 Villeurbanne Cedex, France (until August 2017) / Technical University of Munich, Department of Mechanical Engineering, Chair of Automatic Control, Boltzmannstra{\ss}e 15, 85748 Garching, Germany. \texttt{kotyczka@tum.de}.},\; Bernhard Maschke\footnote{Univ Lyon, Université Claude Bernard Lyon 1, CNRS, LAGEP UMR 5007, 43 Boulevard du 11 Novembre 1918, 69622 Villeurbanne Cedex, France. \texttt{bernhard.maschke@univ-lyon1.fr}.} \; and Laurent Lefèvre\footnote{Univ. Grenoble Alpes, LCIS, 26902 Valence, France. \texttt{laurent.lefevre@lcis.grenoble-inp.fr}.}}

\begin{document}

\maketitle

\begin{abstract}
We present the mixed Galerkin discretization of distributed parameter port-Hamiltonian systems. On the prototypical example of hyperbolic systems of two conservation laws in arbitrary spatial dimension, we derive the main contributions: (i) A weak formulation of the underlying geometric (Stokes-Dirac) structure with a segmented boundary according to the causality of the boundary ports. (ii) The geometric approximation of the Stokes-Dirac structure by a finite-dimensional Dirac structure is realized using a mixed Galerkin approach and power-preserving linear maps, which define minimal discrete power variables. (iii) With a consistent approximation of the Hamiltonian, we obtain finite-dimensional port-Hamiltonian state space models. By the degrees of freedom in the power-preserving maps, the resulting family of structure-preserving schemes allows for trade-offs between centered approximations and upwinding. We illustrate the method on the example of Whitney finite elements on a 2D simplicial triangulation and compare the eigenvalue approximation in 1D with a related approach.

\end{abstract}

\section{Introduction}
The \emph{port-Hamiltonian} (PH) approach for the modeling, interconnection and control of multi-physics systems underwent an enormous evolution during the past two decades. In this article, we concentrate on distributed parameter PH systems as initially presented in \cite{schaft2002hamiltonian}, and refer the reader to the books \cite{duindam2009modeling}, \cite{jacob2012linear} and \cite{schaft2014port} for a more general overview on theory and applications. The salient feature of a PH system is its representation in terms of (i) a \emph{linear geometric interconnection structure} -- a \emph{Stokes-Dirac structure} -- that describes the power flows inside the system and over its boundary and (ii) an energy functional (or more generally potentials) from which the constitutive or closure relations are derived, and which determines the nature of the system. Completely different systems -- linear/nonlinear or hyperbolic/parabolic \cite{zwart2016building} -- can share the same interconnection structure. PH systems are by definition \emph{open systems}, they interact with their environment through energy flow over \emph{boundary ports}. The \emph{in}- and \emph{outputs} in the sense of systems' theory and control are defined via a duality product whose value equals the exchanged power at the port. The definition of boundary port variables plays a crucial role in showing that a PH system is a well-posed boundary control systems \cite{legorrec2005dirac}. The definition of distributed power variables as in- and outputs is discussed in \cite{nishida2015boundary}.

The simulation and control by numerical methods, of \emph{complex} (complex geometries, nonlinearities, interdomain couplings) distributed parameter PH systems, requires a spatial discretization, which shall retain the underlying geometric properties related to \emph{power continuity}. According to the separation of the interconnection structure from the constitutive equations, a \emph{geometric} or \emph{structure-preserving discretization} consists of two steps:
\begin{itemize}
	\item Finite-dimensional approximation of the underlying Stokes-Dirac structure. The duality between the power variables (their duality product has the interpretation of a power) must be mapped onto the finite-dimensional approximation. This requires a \emph{mixed} approach with different approximation spaces for each group of dual power variables (called flows and efforts). The subspace of the approximated, discrete (in space) power variables on which the preserved power-continuity holds, defines a \emph{Dirac structure} as a finite-dimensional counterpart of the Stokes-Dirac structure.
	\item  Consistent discretization of the constitutive equations in the previously chosen approximation spaces, which gives rise to the definition of a \emph{discrete Hamiltonian}.
\end{itemize} 
A geometric or structure-preserving discretization is, hence, a \emph{compatible} discretization as defined in \cite{bochev2006principles}: ``\emph{Compatible discretizations transform partial differential equations to discrete algebraic problems that mimic fundamental properties of the continuum equations.}'' For PH systems, such a fundamental property is the power balance defined by the Dirac structure with respect to which the PH system is defined. The \emph{open} character of PH systems requires special attention to the treatment of the boundary port variables, in particular the boundary inputs which are \emph{imposed} as boundary conditions. The \emph{simplicial} discretization \cite{seslija2012discrete}, \cite{seslija2014explicit} of PH systems based on \emph{discrete exterior calculus} (see e.\,g. \cite{desbrun2008discrete}) can be considered a \emph{direct discrete} formulation of the conservation laws, which, in conjunction with the consistent approximation of the closure equations, is such a compatible discretization. \cite{kotyczka2017discrete} addresses a generalized distribution of boundary inputs on dual meshes, a revised interpretation of the resulting state space models, and the consistent numerical approximation of nonlinear closure equations. Recently, the very related discretization on staggered grids has been reported using finite volumes \cite{kotyczka2016finite} and finite differences \cite{trenchant2017structure}.

The first approach for a \emph{structure-preserving} discretization of PH systems in the spirit of \emph{mixed finite elements} has been proposed in \cite{golo2004hamiltonian},  see  \cite{baaiu2009structure} for its application to a \emph{diffusive} process. There, the Stokes-Dirac structure is, however, discretized in \emph{strong form} which produces restrictive compatibility conditions. In the 1D \emph{pseudo-spectral} method \cite{moulla2012pseudo}, the degeneracy of the discrete duality product is rectified by the definition of \emph{reduced effort variables} (see also the recent paper \cite{vu2017symplectic} for the application to plasma dynamics described by a parabolic PDE). In \cite{farle2013porta}, the Stokes-Dirac structure is reformulated, changing the role of state and co-state variable in one conservation law. The discrete power variables are immediately connected with a non-degenerate duality pairing, at the price of a metric-dependent interconnection structure. 

The \emph{weak} formulation as the basis for Galerkin numerical approximations, including the different variations of the finite element method (see \cite{quarteroni1994numerical}, to cite only one textbook), has been only rarely used for modeling and discretization of PH systems: In \cite{farle2013porta}, one of the two conservation laws is written in weak form. \cite{altmann2017port} presents the PH model of the reactive 1D Navier-Stokes equations in weak form. In \cite{cardoso2017port}, the inclusion of a piezo patch on a flexible beam in the PH model, and the structure-preserving discretization are performed via the weak form.

In this article, we present the geometric discretization of distributed parameter PH systems based on the \emph{weak formulation} of the underlying Stokes-Dirac structure. Doing so, some limitations and restrictions of current approaches for PH systems can be overcome.
\begin{itemize}
    \item The strict separation of \emph{metric-independent structure} and \emph{constitutive equations} is maintained in our approach.
    \item Our formulation is valid for systems on spatial domains with \emph{arbitrary dimension}.
    \item Boundary inputs\footnote{In-domain inputs can be treated identically.} are imposed weakly, i.\,e. they \emph{appear directly} in the weak formulation of the Stokes-Dirac structure and the finite-dimensional approximation.
    \item The power-preserving maps for the discrete power variables offer design degrees of freedom to parametrize the resulting finite-dimensional PH state space models. They can be used to realize \emph{upwinding}.
    \item Mapping the flow variables instead of the efforts \emph{avoids a structural artificial feedthrough}, which is not desirable for the approximation of hyperbolic systems.
\end{itemize}

We consider as the prototypical example of distributed parameter PH systems, an \emph{open system of two hyperbolic conservation laws in canonical form}, as presented in \cite{schaft2002hamiltonian}. We use the language of differential forms, see e.\,g. \cite{flanders1963differential}, which highlights the geometric nature of each variable and allows for a unifying representation independent from the dimension of the spatial domain. 

An important reason for expressing the spatial discretization of PH systems based on the weak form is to make the link with modern geometric discretization methods. Bossavit's work in computational electromagnetism \cite{bossavit1991differential}, \cite{bossavit1998computational} and Tonti's cell method \cite{tonti2001direct} keep track of the geometric nature of the system variables which allows for a direct interpretation of the discrete variables in terms of \emph{integral} system quantities. This integral point of view is also adopted in \emph{discrete exterior calculus} \cite{desbrun2008discrete}. \emph{Finite element exterior calculus} \cite{arnold2010finite} gives a theoretical frame to describe functional spaces of differential forms and their compatible approximations, which includes the construction of higher order approximation bases that generalize the famous \emph{Whitney} forms \cite{whitney1957geometric}, see also \cite{rapetti2009whitney}. We refer also to the recent article \cite{hiemstra2014high} which proposes \emph{conforming} polynomial approximation bases, in which the conservation laws are \emph{exactly} satisfied, and which gives an excellent introduction to the geometric discretization. Impressing examples for the use of geometric discretization methods can be found in weather prediction \cite{cotter2012mixed} or in the simulation of large-scale fluid flows \cite{cotter2014finite}, where the conservation of potential vorticity plays an important role. Another important aspect of using the weak form as basis for structure-preserving discretization is to make the link with well-known numerical methods and to pave the way for a simulation of PH systems with existing numerical tools like FreeFEM++ \cite{hecht2012new}, GetDP \cite{geuzaine2008getdp} or FEniCS \cite{alnaes2015fenics}. 

The paper is structured as follows. In Section 2, we give a quick introduction to functional spaces of differential forms and we review the definition of distributed parameter PH systems based on the underlying Stokes-Dirac structure. Following the definition of boundary ports with alternating causality, we propose the \emph{weak form of the Stokes-Dirac structure}. Section 3 deals with the \emph{mixed Galerkin approximation} of this Stokes-Dirac structure. Due to the different geometric nature of the power variables and their approximation spaces, the discrete power balance involves \emph{degenerate} duality pairings. We define \emph{minimal} discrete power variables (pairs of bond variables) with \emph{non-degenerate} duality products by \emph{power-preserving mappings}. The so-defined subspace of the bond space is a \emph{Dirac structure} which admits different representations. The \emph{explicit} input-output representation, together with the finite-dimensional approximation of the Hamiltonian, leads to the desired PH approximate models in state space form. Section 4 illustrates the approach using Whitney finite elements on a 2D simplicial grid. We highlight the interpretation of the finite-dimensional state and power variables in terms of integral quantities on the grid and illustrate how the approximation quality can be tuned by the mapping parameters with a 2D simulation study. We compare the 1D eigenvalue approximation with the method of \cite{golo2004hamiltonian}. Certain parameter choices can be interpreted in terms of \emph{upwinding}, which is particularly favorable for hyperbolic systems. Section \ref{sec:conclusions} closes the paper with a summary and an outlook to ongoing and future work.

\section{Weak form for port-Hamiltonian systems of conservation laws}

\subsection{Differential forms and functional spaces}
To make the remainder of the paper self-contained, we give a compact introduction to the calculus with differential forms and their functional spaces. For further reading we refer to \cite{flanders1963differential}, \cite{arnold1989mathematical} and the paper \cite{arnold2010finite} with its numerous references. The calculus with differential forms, or \emph{exterior differential calculus} is widely used in the simulation of Maxwell's equations \cite{bossavit1998computational}, to give one example. \emph{Discrete exterior calculus} \cite{desbrun2008discrete} extends the formalism to discrete geometric objects defined on oriented meshes, and \emph{finite element exterior calculus} \cite{arnold2010finite} sets the framework for numerical approximation using finite element spaces of differential forms \cite{arnold2013spaces}.

\subsubsection{Smooth differential forms} We represent distributed parameter PH systems in the language of \emph{differential forms}, see e.\,g. \cite{flanders1963differential} for a comprehensive introduction to \emph{smooth} differential forms, i.\,e. differential forms with sufficiently differentiable (in the classical sense) coefficient functions. Let $\Omega$ be an open, bounded and connected $n$-dimensional spatial domain with Lipschitz boundary $\partial \Omega$ and denote $\Lambda^k(\Omega)$ the space of smooth differential $k$-forms on $\Omega$. For a smooth $n-1$-form $\omega \in \Lambda^{n-1}(\Omega)$, the continuous extension to the boundary is denoted $\mathrm{tr}\, \omega \in \Lambda^{n-1}(\partial \Omega)$. The symbol $\mathrm{tr}$ stems from the \emph{trace map}, which defines the extension to the boundary for Lebesgue integrable functions (see further below). The \emph{exterior derivative} $\mathrm d: \Lambda^k(\Omega) \rightarrow \Lambda^{k+1}(\Omega)$ represents, depending on the degree $k$, the different differential operators from vector calculus. The wedge product $\wedge: \Lambda^k(\Omega) \times \Lambda^l(\Omega) \rightarrow \Lambda^{k+l}(\Omega)$ is a skew-symmetric exterior product of differential forms. We will make frequent use of the following three formulas\footnote{See e.\,g. \cite{flanders1963differential}, Sections 2.3 and 3.2 for the first formulas. For Stokes' theorem, see e.\,g. \cite{flanders1963differential}, Section 5.8 or \cite{arnold1989mathematical}, Section 36.D, formulated for $\Omega$ a $n$-chain, i.\,e. a formal sum of $n$-simplices on a manifold $M \supset \Omega$.} for $\lambda \in \Lambda^k(\Omega)$, $\mu\in \Lambda^l(\Omega)$, and $\omega\in \Lambda^{n-1}(\Omega)$:
\begin{alignat}{2}
	\label{c04-eq:skew-symmetry-wedge-product}
	\lambda \wedge \mu &= (-1)^{kl} \mu \wedge \lambda 
    && \text{(Skew-symmetry of $\wedge$)}\\
	\label{c04-eq:ext-diff-product}
	\textrm d( \lambda \wedge  \mu ) &=
	\textrm d \lambda \wedge \mu  + (-1)^{k} \lambda \wedge \textrm d \mu \qquad
    &&  \text{(Product rule for $\d$)}\\
	\label{c04-eq:stokes-theorem}
	\int_{\Omega} \mathrm d \omega &= \int_{\partial \Omega} \mathrm{tr}\,\omega
    && \text{(Stokes' theorem)}
\end{alignat}
A natural pairing or \emph{duality product} between two differential forms $\lambda \in \Lambda^k(\Omega)$ and $\mu \in \Lambda^{n-k}(\Omega)$ on $\Omega$  is given by
\begin{equation}
    \label{c04-eq:notation-duality-product}
    \langle \lambda | \mu \rangle_{\Omega} := \int_{\Omega} \lambda \wedge \mu.
\end{equation}
Accordingly for $\partial \Omega$, see \cite{schaft2002hamiltonian}, Eq. (5). The generalized Stokes' theorem \eqref{c04-eq:stokes-theorem}, together with the product rule \eqref{c04-eq:ext-diff-product} and the short notation  \eqref{c04-eq:notation-duality-product}, gives the \emph{integration-by-parts} formula for smooth differential forms $\lambda \in \Lambda^k(\Omega)$ and $\mu \in \Lambda^{n-k-1}(\Omega)$,
\begin{equation}
    \label{c04-eq:integration-by-parts}
    \langle \mathrm d\lambda | \mu \rangle_{\Omega} =
    \langle \mathrm {tr}\, \lambda | \mathrm{tr}\, \mu \rangle_{\partial\Omega}
    - (-1)^k \langle \lambda | \mathrm d \mu \rangle_{\Omega}. 
\end{equation}

\subsubsection{Lebesgue and Sobolev spaces of differential forms} 
\label{subsubsec:lebesgue-sobolev}
We recall some important definitions and facts, which ensure that the formulas from the previous subsection make also sense on functional spaces of differential forms with weaker smoothness conditions. Section 4 of \cite{arnold2010finite} gives a quick and concise introduction into calculus with differential forms whose coefficient functions belong to Lebesgue spaces $L^p(\Omega)$ and Sobolev spaces, in particular $H^m(\Omega) = W^{m,2}(\Omega)$. The space $L^2 \Lambda^k(\Omega)$ of differential forms with square integrable coefficient functions is equipped with the inner product\footnote{To define the inner product, we need a volume form. For $\Omega \subset \IR^n$, we take $d\text{vol} = d^nz$ as in \cite{holm2011geometric}, Definition 3.6.2.}
\begin{equation}
    \langle \alpha, \beta \rangle_{L^2\Lambda^k(\Omega)} := 
    \int_\Omega \sum_{i=1}^n \alpha_i(z) \beta_i(z) \,d\text{vol},
\end{equation}
where $\alpha_i, \beta_i \in L^2(\Omega)$, $i=1,\ldots,n$ are the component functions of $\alpha, \beta \in L^2 \Lambda^k(\Omega)$. The \emph{weak exterior derivative} $\mathrm d \lambda$ of $\lambda \in \Lambda^k(\Omega)$ can be defined via the integration-by-parts formula \eqref{c04-eq:integration-by-parts}, with smooth differential forms $\mu$ that vanish on the boundary (due to their compact support in $\Omega$):
\begin{equation}
    \label{eq:weak-exterior-derivative}
     \langle \mathrm d\lambda | \mu \rangle_{\Omega} =
    - (-1)^k \langle \lambda | \mathrm d \mu \rangle_{\Omega} \qquad \forall \mu \in C_c^\infty \Lambda^{n-k-1}(\Omega).
\end{equation}
We do not introduce a new symbol, as we will understand $\mathrm d$ in this weak sense in the rest of the paper. This allows to apply the exterior derivative to differential forms whose coefficient functions are not differentiable in the classical sense. The Sobolev spaces $H^m \Lambda^k(\Omega)$ contain the differential forms on $\Omega$ with $L^2$ weak derivatives up to order $m$. The corresponding inner product for $m=1$ is defined as
\begin{equation}
    \langle \alpha , \beta \rangle_{H^1 \Lambda^k(\Omega)} :=
    \langle \alpha , \beta \rangle_{L^2 \Lambda^k(\Omega)} +
    \langle \mathrm d \alpha , \mathrm d \beta \rangle_{L^2 \Lambda^{k+1}(\Omega)}.
\end{equation}
As we deal with \emph{boundary control systems}, we are particularly interested in the extension of certain differential forms to the boundary. Fortunately, the \emph{trace theorem} from classical functional analysis\footnote{See e.\,g. \cite{brezis2011functional}, Section 9.8, paragraphs 6 and 7 for fractional Sobolev spaces and the trace theorem.} extends to differential forms as discussed in Section 4 of \cite{arnold2010finite}. We will make heavy use of the implication
\begin{equation}
    \lambda \in H^1 \Lambda^k(\Omega) \quad \Rightarrow \quad \mathrm{tr}\, \lambda \;\in\; H^{1/2} \Lambda^k(\partial \Omega) \;\subset\; L^{2} \Lambda^k(\partial \Omega).
\end{equation}
Where convenient for compactness, we use the common abusive notation $\int_{\partial \Omega} \omega = \int_{\partial\Omega} \mathrm tr\, \omega$ for the extension of $\omega \in H^m \Lambda^{n-1}(\Omega)$, $m\geq1$ to the boundary.

\subsection{Distributed parameter port-Hamiltonian systems}
We consider systems of two conservation laws in a \emph{canonical form}\footnote{Or systems of two conservation laws with \emph{canonical interdomain coupling}.} as introduced in \cite{schaft2002hamiltonian}. These systems share a common linear geometric structure that relates their \emph{power variables}, i.\,e. the pairs of physical quantities that constitute their \emph{power balance} equation.

\subsubsection{The Stokes-Dirac structure}
The canonical structure, defined on the open, connected $n$-dimensional domain $\Omega$ with Lipschitz boundary $\partial \Omega$, is expressed, on the one hand, by
\begin{equation}
	\label{eq:stokes-dirac-interconnection}
	\bmat{f^p\\ f^q}
		=
    \underbrace{
	\bmat{0 & (-1)^r \mathrm{d}\\ \mathrm{d} & 0} 
    }_{\mathcal J}
    \bmat{e^p\\ e^q},
\end{equation}
with the \emph{flow} differential forms $f^p \in L^2 \Lambda^p(\Omega)$, $f^q \in L^2 \Lambda^q(\Omega)$, and the \emph{effort} differential forms $e^p \in H^1 \Lambda^{n-p}(\Omega)$, $e^q \in H^1 \Lambda^{n-q}(\Omega)$. The degrees $p$ and $q$ of the differential forms satisfy $p+q=n+1$ and the exponent $r = pq + 1$ ensures the formal skew-symmetry\footnote{A \emph{formal} differential operator $\mathcal J$ is defined \emph{without} boundary conditions (see e.\,g. \cite{kato1995perturbation}, Sect. III.3). Formal skew-symmetry is verified by $\langle \vec e, \mathcal J \vec e\rangle = - \langle \mathcal J \vec e, \vec e \rangle$ under \emph{zero} boundary conditions, where $\langle \cdot, \cdot \rangle$ is the inner product on the appropriate functional space.} of the matrix-valued differential operator $\mathcal J$ for arbitrary dimension $n$, see \cite{schaft2002hamiltonian}. On the other hand, the extensions of the efforts to the boundary define the \emph{boundary port variables}
\begin{equation}
	\label{eq:stokes-dirac-boundary-ports}
	\bmat{f^\partial\\ e^\partial }
		=
	\bmat{\mathrm{tr} & 0\\ 0 & (-1)^p \mathrm{tr}} 
    \bmat{ e^p\\  e^q},
\end{equation}
$f^\partial \in L^2 \Lambda^{n-p}(\partial \Omega)$, $e^\partial \in L^2 \Lambda^{n-q}(\partial \Omega)$. Note that here we repeat the case of \cite{schaft2002hamiltonian} with a single \emph{causality}. The term causality describes which of the boundary port variables is imposed as an \emph{input} boundary condition in the sense of automatic control. Moreover, the definition of boundary flows and efforts is not unique (see \cite{legorrec2005dirac} for a complete characterization). As shown in \cite{schaft2002hamiltonian}, the pairs of flow and effort variables that satisfy \eqref{eq:stokes-dirac-interconnection}, \eqref{eq:stokes-dirac-boundary-ports}, define a linear subspace of the \emph{bond space}\footnote{As a reference to bond graph modeling of dynamical systems \cite{paynter1961analysis}, see also \cite{duindam2009modeling}, Chapter 1.} $\mathcal F \times \mathcal E$, 
\begin{equation}
	\label{c04-eq:spaces-F-E}
	\begin{split}
		\mathcal F &= L^2\Lambda^p(\Omega) \times L^2\Lambda^{q}(\Omega) \times L^2\Lambda^{n-p}(\partial \Omega),\\
		\mathcal E &= H^1\Lambda^{n-p}(\Omega) \times H^1\Lambda^{n-q}(\Omega) \times L^2\Lambda^{n-q}(\partial \Omega),
	\end{split}
\end{equation}
on which the \emph{power balance} equation 
\begin{equation}
	\label{c04-eq:power-continuity}
    \langle e^p | f^p \rangle_{\Omega} + \langle e^q | f^q \rangle_{\Omega}
    + 
    \langle e^\partial | f^\partial \rangle_{\partial \Omega}
    = 0
\end{equation}
holds. In addition, this subspace is \emph{maximally isotropic} with respect to the symmetrized duality pairing which is represented by the left hand terms of \eqref{c04-eq:power-continuity}. For details on this linear subspace called a \emph{Stokes-Dirac structure}, we refer to \cite{schaft2002hamiltonian}. It essentially generalizes the notion of a \emph{Dirac structure} to the distributed parameter case by exploiting \emph{Stokes'} theorem.

A Dirac structure, whose definition and characterization are summarized below, can be considered as ``\emph{the geometrical notion formalizing general power-conserving interconnections}'' \cite{schaft2002hamiltonian}.

\begin{dfn}[\cite{courant1990dirac}, Definition 1.1.1]
    \label{c04-def:fin-dim-dirac-structure}
	Given the finite-dimensional linear space $V$ over $\IR$ or another field and its dual $V^*$ with respect to the duality pairing $\langle \cdot | \cdot \rangle: V \times V^*\rightarrow \IR$. Define the symmetric bilinear form 
	\begin{equation}
		\langle\langle (\vec f_1, \vec e_1), (\vec f_2, \vec e_2) \rangle \rangle := \frac{1}{2} \left( \langle \vec e_1 | \vec f_2 \rangle + \langle \vec e_2 | \vec f_1 \rangle \right), \qquad (\vec f_i, \vec e_i) \in V \times V^*, \; i=1,2.
	\end{equation}
	A \emph{Dirac structure} is a linear subspace $D \subset  V \times V^*$ which is \emph{maximally isotropic} under $\langle\langle \cdot, \cdot \rangle\rangle$.
\end{dfn}

Equivalently, a Dirac structure can be characterized as the subspace $D \subset  V \times V^*$ which equals its orthogonal complement with respect to $\langle\langle \cdot, \cdot \rangle\rangle$: $D = D^\perp$, see \cite{schaft2002hamiltonian}, Definition 2.1. $D$ is isotropic under $\langle\langle \cdot, \cdot \rangle\rangle$, if $\langle\langle (\vec f_1,\vec e_1), (\vec f_2,\vec e_2) \rangle\rangle = 0$ for all $(\vec f_1, \vec e_1), (\vec f_2, \vec e_2) \in D$, from which $D\subset D^\perp$ follows. If, in addition, for every $(\vec f_1, \vec e_1) \in D$ there exists \emph{no} $(\vec f_3, \vec e_3) \notin D$ such that $\langle\langle (\vec f_1,\vec e_1), (\vec f_3,\vec e_3) \rangle\rangle = 0$, then $D$ is \emph{maximally} isotropic, and also $D^\perp \subset D$ is true, which implies $D = D^\perp$. The isotropy condition implies that 
\begin{equation}
	\langle\langle (\vec f,\vec e), (\vec f,\vec e) \rangle\rangle = \langle \vec e|\vec f \rangle = 0 \qquad \forall\; (\vec f, \vec e) \in V \times V^*.
\end{equation}
If $V$ and $V^*$ are spaces of conjugated power variables, this is indeed a power balance equation. For more details and the different representations of finite-dimensional Dirac structures (in the PH context), we refer to  \cite{schaft2000l2}, \cite{schaft2002hamiltonian}. For Dirac structures defined on Hilbert spaces, and their composition, see e.\,g. Chapter 5 of \cite{golo2002interconnection} and \cite{kurula2010dirac}.

\subsubsection{Canonical PH systems of two conservation laws}
To define a port-Hamiltonian distributed parameter system, the Stokes-Dirac structure is completed by \emph{dynamic equations} that introduce evolution with respect to time, and \emph{constitutive relations}, which define the \emph{nature} of the resulting dynamic system of PDEs. We focus on PH systems based on the \emph{canonical} differential operator $\mathcal J$ as indicated in \eqref{eq:stokes-dirac-interconnection}. Moreover, we derive the constitutive equations for the effort variables from a single energy (Hamiltonian) functional. This results in a \emph{hyperbolic} system of conservation laws in PH form.

The flows induce the time evolution of the distributed \emph{state variables}\footnote{We use the same symbols for the state variables (as differential forms) and their degrees, which should in general not provoke any confusion. In this paragraph, we explicitly indicate the arguments $(z,t)$, for the Hamiltonian can depend on $z$ as in the case of the shallow water equations with variable bed profile. In the sequel, we will omit the arguments.} $p(z,t) \in L^2 \Lambda^p(\Omega)$, $q(z,t) \in L^2\Lambda^q(\Omega)$ with corresponding initial conditions:
\begin{equation}
	\label{eq:phs-flows}
    \bmat{- \partial_t p(z,t)\\ - \partial_t q(z,t)} =
    \bmat{f^p(z,t)\\ f^q(z,t)},
    \qquad
    \bmat{p(z,0)\\q(z,0)} =
    \bmat{p_0(z)\\ q_0(z)}.
\end{equation}
The closure or constitutive equations relate the state and co-state (or co-energy or effort) variables according to
\begin{equation}
    \label{eq:phs-efforts}
    \bmat{e^p(z,t)\\ e^q(z,t)} =
    \bmat{\delta_p H(p(z,t), q(z,t))\\ \delta_q H(p(z,t), q(z,t))},
\end{equation}
where the right hand side contains the \emph{variational derivatives} of the \emph{Hamiltonian} or \emph{energy} functional
\begin{equation}
	\label{c04-eq:energy-functional}
	H(p(z,t),q(z,t)) = \int_\Omega \mathcal H(p(z,t),q(z,t),z)
\end{equation}
with the \emph{Hamiltonian density} $n-$form $\mathcal H$. The variational derivatives are the unique differential $n-p$-form $\delta_p H$ and $n-q$-form $\delta_q H$ that satisfy\footnote{See e.\,g. \cite{duindam2009modeling}, p. 232.} 
\begin{equation}
	\label{c04-eq:variational-derivative}
	\begin{split}
	H(p+\delta p, q+\delta q) = \int_{\Omega} \mathcal H(p, q, z) + \int_{\Omega} \delta_p H \wedge \delta p + \delta_q H \wedge \delta q \;+\; o(\delta p, \delta q).
	\end{split}
\end{equation}

\begin{dfn}[\cite{schaft2002hamiltonian}, Definition 2.2]
We call
\begin{equation}
	\label{c04-eq:canonical-phs-diff-eq}
	\bmat{- \partial_t p\\ - \partial_t q}
		=
	\bmat{0 & (-1)^r \mathrm{d}\\ \mathrm{d} & 0} 
    \bmat{\delta_p H\\ \delta_q H},
    \qquad
	\bmat{f^\partial\\ e^\partial}
		=
	\bmat{\mathrm{tr} & 0\\ 0 & (-1)^p \mathrm{tr}} 
    \bmat{\delta_p H\\ \delta_q H}
\end{equation}
a \emph{distributed parameter port-Hamiltonian system} on the $n$-dimensional spatial manifold $\Omega$.
\end{dfn}

Imposing the port variables $f^\partial$ and/or $e^\partial$ on a subset of $\partial \Omega$ as \emph{control input} (and understanding the remaining boundary port variables as \emph{observation} or \emph{output}), makes the system representation \eqref{c04-eq:canonical-phs-diff-eq} a \emph{boundary control system} in the sense of \cite{fattorini1968boundary}. For 1D linear PH systems with a \emph{generalized} skew-symmetric system operator, \cite{legorrec2005dirac} gives conditions on the assignment of boundary in- and outputs for the system operator to generate a contraction semigroup. The latter is instrumental to show \emph{well-posedness} of a linear PH system, see \cite{jacob2012linear}. Essentially, \emph{at most half the number of boundary port variables} can be imposed as control inputs for a well-posed PH system in 1D.

Taking $\delta p = \dot p$, $\delta q = \dot q$ as variations in \eqref{c04-eq:variational-derivative}, and omitting the higher order terms, the time derivative of the energy functional \eqref{c04-eq:energy-functional} reads
\begin{equation}
	\label{c04-eq:energy-balance-distributed}
	\begin{split}
	\dot H =
	\int_{\Omega}
	\delta_p H \wedge \dot p +  \delta_q H \wedge \dot q
    = \langle \delta_p H | \dot p \rangle_{\Omega} + 
    \langle \delta_q H | \dot q \rangle_{\Omega}.
	\end{split}
\end{equation}
Replacing $\dot p$, $\dot q$ according to \eqref{c04-eq:canonical-phs-diff-eq} and using the integration-by-parts formula \eqref{c04-eq:integration-by-parts} yields
\begin{equation}
	\label{eq:2-energy-balance-boundary}
	\dot H = \int_{\partial \Omega} (-1)^p \left. \delta_q H \right|_{\partial \Omega} \wedge \left. \delta_p H \right|_{\partial \Omega} = (-1)^p \langle \delta_q H | \delta_p H  \rangle_{\partial\Omega}.
\end{equation}
Equating the right hand sides of the last two equations gives, together with the definition of boundary port variables in \eqref{c04-eq:canonical-phs-diff-eq}, the power balance equation 
\begin{equation}
    \label{eq:2-power-balance-with-explanation}
	\underbrace{\langle \delta_p H | -\dot p \rangle_{\Omega} + 
    \langle \delta_q H | -\dot q \rangle_{\Omega}}_{\substack{\textrm{power extracted from}\\\textrm{distributed storage}}} \quad + \quad 
    \underbrace{ \langle e^\partial | f^\partial \rangle_{\partial \Omega} }_{\substack{\textrm{power supplied}\\ \textrm{over the boundary}}} 
    \quad = \quad 0,
\end{equation}
which is a purely \emph{structural property}, as it follows directly from \eqref{c04-eq:power-continuity} and the definitions of distributed and boundary flows and efforts.

\begin{rmk}
Defining the \emph{flux functions}
\begin{equation}
    \bmat{\beta^p\\ \beta^q} = 
    \bmat{0 & (-1)^r \mathrm{d}\\ \mathrm{d} & 0} 
    \bmat{\delta_p H\\ \delta_q H},
\end{equation}
it is evident that \eqref{c04-eq:canonical-phs-diff-eq} represents a \emph{hyperbolic} system of two conservation laws. Note that we explicitly defined boundary port variables whose pairing describes a power flow over the system boundary. We therefore deal with \emph{open} systems of conservation laws.
\end{rmk}

\begin{rmk}
    For the \emph{same} Stokes-Dirac structure, PDE systems of different nature are obtained when flows and efforts are defined based on different dynamics and closure equations. For a quadratic Hamiltonian density $\mathcal H$ in $p$ and $q$, the resulting \emph{hyperbolic} PH system is linear, otherwise nonlinear. The linear case is treated e.\,g. in \cite{jacob2012linear}, where $\mathcal H$ is bounded and non-negative, and $H$ serves as the energy norm on the corresponding Hilbert space. For different definitions of flows and efforts, in particular if both efforts are not derived from the same functional, the resulting PDE system becomes \emph{parabolic}, see e.\,g. \cite{zwart2016building}, which allows to represent diffusive phenomena with the same Stokes-Dirac structure, see e.\,g. the heat conduction example in \cite{duindam2009modeling}, Section 4.2.2, or \cite{baaiu2009port}.
\end{rmk}

\begin{rmk}
    The division of the system variables into \emph{flows} (i.\,e. time derivatives of \emph{states}) and \emph{efforts} (or \emph{co-states}) stems from the \emph{duality} arizing from the variational formula \eqref{c04-eq:variational-derivative}, see also \eqref{c04-eq:energy-balance-distributed}. It takes into account their different geometric definition, such as the degree of the differential forms. Tonti, for example, distinguishes between \emph{configuration} and \emph{source} variables \cite{tonti2001direct}, which are \emph{states} and \emph{efforts} in our language. His \emph{energy variables} are products of these dual quantities, whereas in our context, we build the \emph{duality products} between flows and efforts in order to compute \emph{powers}. The space of dual power variables contains pairs of \emph{in- and output variables} (denoted boundary efforts and flows), which describe the energy flow over the system boundary and make the PH representation inherently \emph{control oriented}.  A central feature of PH modelling and control is the separation of the linear relations between the power variables -- described by a (Stokes-)Dirac structure -- from the constitutive and dynamics equations. This separation shall be maintained under \emph{structure-preserving} discretization.
\end{rmk}

\subsubsection{Examples}
For illustration, we give two examples for systems of two conservation laws that share the same Stokes-Dirac structure and can be written as PH distributed parameter systems. In the second example, we highlight the relations of the representations in terms of vector calculus and differential forms.

\begin{exm}[1D transmission line]
    \label{exa:1D-transmission-line}
    The simplest 1D example of a system of two conservation laws is an electric transmission line (the ``Telegrapher's equations'') with the spatial coordinate $z \in \Omega = (0, L)$, see e.\,g. \cite{golo2004hamiltonian}. With $p(z) = \psi(z) \in \Lambda^1(\Omega)$, the magnetic flux density one-form, $q(z) \in \Lambda^1(\Omega)$, the electric charge density one-form, $l(z) dz, c(z) dz \in \Lambda^1(\Omega)$ the distributed inductance and capacitance per length ($l(z)$ and $c(z)$ are smooth functions and $dz$ the basis one-form), the Hamiltonian density one-form is $\mathcal H(p,q) = \frac{1}{2} \left( p(z) \wedge \ast \frac{p(z)}{l(z)}+ q(z) \wedge \ast \frac{q(z)}{c(z)} \right)$. The \emph{Hodge star} operator $\ast: \Lambda^k(\Omega) \rightarrow \Lambda^{n-k}(\Omega)$ renders in the 1D case a one-form a zero-form and \emph{vice versa}\footnote{The \emph{Hodge star} induces an \emph{inner product} on the space of differential forms on a manifold $\Omega$ by  
    $
        (\alpha, \beta) := 
        \langle \alpha | \ast \beta \rangle_{\Omega} = \langle \beta | \ast \alpha \rangle_{\Omega} = (\beta, \alpha)$,
    $\alpha, \beta \in \Lambda^k(\Omega)$,
    see Section 8.4 of \cite{flanders1963differential} or Section 3.6 of \cite{holm2011geometric}. The inner product is not necessarily the standard $L^2$ norm, but may be equipped with another metric, see e.\,g. the energy norm for linear PH systems \cite{jacob2012linear}. The Hodge star is, hence, \emph{metric} dependent. A given inner product space induces a corresponding Hodge star.}. The variational derivatives of the Hamiltonian $H = \int_0^L \mathcal H$ are the current and the voltage along the line, $e^p(z) = \delta_p H = \ast \frac{p(z)}{l(z)} = i(z) \in \Lambda^0(\Omega)$ and $e^q(z) = \delta_q H = \ast \frac{q(z)}{c(z)} = v(z) \in \Lambda^0(\Omega)$. Note that in the 1D case, the \emph{disconnected} nature of the boundary with opposite orientation of its two parts requires to modify the definition of boundary port variables according to \eqref{c04-eq:canonical-phs-diff-eq}. With the boundary flow and effort \emph{vectors}
    \begin{equation}
        \label{eq:boundary-flows-efforts-1D}
        \vec f^\partial = \bmat{ e^p(0)\\ e^q(L)}, \quad
        \vec e^\partial = \bmat{ e^q(0)\\ -e^p(L)},
    \end{equation}
    the power transmitted over the boundary can be written as the standard inner product
    \begin{equation}
        \langle \vec e^\partial | \vec f^\partial \rangle = e^q(0) e^p(0) - e^q(L) e^p(L) = 
        - \langle e^q | e^p \rangle_{\partial \Omega}.
    \end{equation}
\end{exm}

\begin{exm}[2D shallow water equations]
	The shallow water equations describe the two-dimensional flow of an inviscid fluid with relatively low depth (``shallow''), which permits the averaging of the horizontal components of the velocity field and the omission of the vertical velocity component. The two equations that describe the conservation of mass and momentum over an infinitesimal, fixed surface element\footnote{Which corresponds to the \emph{Eulerian} representation of the fluid flow.} (we consider the fluid in a non-rotating system) can be written in vector calculus notation, with spatial coordinates $\vec z = \bmat{x & y}^T$, see e.\,g. \cite{fernandez2010coupling}, 
	\begin{equation}
		\begin{split}
			\partial_t  h + \mathrm{div}(h \vec u) &= 0, \\
			\partial_t (h \vec u) + \mathrm{div} (h \vec u \vec u) + \frac{1}{2} g \nabla h^2 + gh\nabla z_b &=0,
		\end{split}
	\end{equation}
	where $h$ denotes the water level over the bed, $z_b$ is the elevation of the bed profile, $\vec u = [u\;\; v]^T$ the 2-dimensional velocity field, $h \vec u = \vec F$ the discharge vector and $g$ the gravitational acceleration. $\vec u \cdot \vec u$ and $\vec u \vec u$ denote respectively the scalar and the tensor (dyadic) product of two vectors. With some rules of tensor calculus\footnote{See Appendix A.4 of \cite{bird2002transport}: $\nabla \cdot (s \vec I) = \nabla s$, $\nabla \cdot (\vec v \vec w) = \vec v \cdot \nabla \vec w + \vec w (\nabla \cdot \vec v)$, $\vec v \cdot \nabla \vec v = \frac{1}{2} \nabla (\vec v \cdot \vec v) - \vec v \times (\nabla \times \vec v)$. The last term with cross product and rotation has to be evaluated based on the 3D velocity vector with zero vertical component.}, and replacing the continuity equation, the momentum equation can be reformulated in terms of $\vec u$ and we obtain
	\begin{equation}
		\bmat{\partial_t h\\ \partial_t {\vec u} + q \vec F^\perp}
		=
		\bmat{0 & - \mathrm{div}\\ - \mathrm{grad} & 0}
		\bmat{\frac{1}{2} \vec u \cdot \vec u + g h + g z_b\\ h \vec u},
	\end{equation}
    where $q = \frac{1}{h} (\partial_x v - \partial_y u)$ denotes the \emph{potential vorticity}\footnote{The potential vorticity satisfies the balance equation $\partial_t q + \vec u \cdot \nabla q = 0$, i.\,e. it is advected with the fluid flow see e.\,g. \cite{arakawa1981potential}. It plays an important role in the long-time numerical simulation of large scale flow problems, see e.\,g. \cite{ringler2010unified}.}, and $\vec F^\perp = [hv\;\; -hu]^T$. The term $q \vec F^\perp$ represents the acceleration of the fluid due to the rotation of the flow. It stems from the rotational part of the transport term in the momentum equation. The total energy (per unit mass) is
    \begin{equation}
        H = \int_{\Omega} \frac{1}{2} h \vec u \cdot \vec u + \frac{1}{2} g h^2 + g h z_b \, d\vec z.
    \end{equation}
    To rewrite the equations in terms of differential forms, we use the relations, see e.\,g. \cite{abraham2012manifolds},\footnote{Index raising ($\sharp$) produces a vector field with the same components from a one-form. Index lowering ($\flat$) produces a one-form with identical components from a vector field. Raising and lowering in these \emph{musical isomorphisms} refers to the fact that upper (lower) indices are typically used for the components of vector fields (one-forms).} 
	\begin{equation}
		\nabla f = (\d f)^\sharp, 	
		\quad
		\mathrm {div}\, \vec f = \ast \d (\ast \vec f^\flat).
	\end{equation}
	Taking into account that $\ast \ast \lambda = (-1)^{k(n-k)} \lambda$ for a $k$-form $\lambda$, we obtain 
	\begin{equation}
        \label{c04-eq:swe-ode-1}
		\bmat{- \partial_t (\ast h)\\ - (\partial_t \vec u + q \vec F^\perp )^\flat}
		=
		\bmat{0 &  - \d \\  \d & 0}
		\bmat{p_{dyn}\\ - (\ast \vec F^\flat)},
	\end{equation}
    where $\ast h \in \Lambda^2(\Omega)$ and $\vec u^\flat \in \Lambda^1(\Omega)$ are the $2$-form and $1$-form associated with the water depth and the flow velocity ($p=2$, $q=1$). $p_{dyn} = \frac{1}{2} \vec u \cdot \vec u + g h + g z_b \in \Lambda^0(\Omega)$ is the hydrodynamic pressure function ($0$-form) and $\ast \vec F^\flat \in \Lambda^1(\Omega)$ is the $1$-form associated to the discharge per unit width. 
    Indeed the vector on the right can be expressed in terms of the variational derivatives $p_{dyn} = \delta_{\ast h} H$ and $-(\ast \vec F^\flat) = \delta_{\vec u^\flat} H$ of the Hamiltonian $H=\int_\Omega \mathcal H$  density $2$-form\footnote{In 2D we have $\ast dx = dy$, $\ast dy = - dx$.} $\mathcal H = \frac{1}{2} h ( \vec u^\flat \wedge \ast \vec u^\flat) + \frac{1}{2} g h \ast h  + g \ast h z_b$. If the rotational term $q \vec F^\perp$ can be neglected\footnote{If not, \eqref{c04-eq:swe-ode-1} still represents a PH system, as the rotational term does not contribute to the energy balance \cite{pasumarthy2008port}. It can be associated to the canonical Stokes-Dirac structure, with a different definition of the \emph{dynamic} equation for the $1$-form $\vec u^\flat$.}, \eqref{c04-eq:swe-ode-1} has the canonical structure  \eqref{c04-eq:canonical-phs-diff-eq}.
\end{exm}

\begin{rmk}
In this paper, we concentrate on \emph{canonical} systems of two conservation laws in \emph{arbitrary} spatial dimension. Beyond this basic class of PH systems (which however covers different linear and nonlinear physical phenomena), there exists a growing number of PH models for different physical phenomena, see e.\,g. \cite{vu2012port} for the modeling of the plasma in a fusion reactor, \cite{altmann2017port} for the reactive Navier-Stokes flow or \cite{zhou2017distributed} for irreversible thermodynamic systems to mention only a few interesting examples. In \cite{polner2014hamiltonian}, a PH formulation of the compressible Euler equations in terms of density, weighted vorticity and dilatation is presented. 
The PH representation is not unique. An important approach for \emph{mechanical} systems is based on a jet bundle formulation \cite{schoeberl2014jet}.
\end{rmk}

\subsection{Boundary ports with alternating causality}
The boundary term $\langle e^\partial | f^\partial \rangle_{\partial \Omega}$ in \eqref{eq:2-power-balance-with-explanation} pairs two power variables, one of which is considered as \emph{control input} imposed on $\partial \Omega$. The other, \emph{dual} variable plays the role of the \emph{collocated} and \emph{power-conjugated} output. The assignment of these roles to the boundary power variables is referred to as \emph{causality of the boundary port}. This choice of boundary port variables to define a Stokes-Dirac structure (an infinite-dimensional PH system) is not unique, see \cite{legorrec2005dirac} for the 1D case, nor must it be homogeneous on $\partial \Omega$. On parts of the boundary, $e^q = \delta_q H$ may define the control input, while this role may be assigned to $e^p = \delta_q H$ on the rest of it. The only constraint on the definition of pairs of boundary port variables is that their product accounts for the power flow over the whole boundary as in \eqref{eq:2-power-balance-with-explanation}. Equation \eqref{eq:2-energy-balance-boundary} may be interpreted as the balance equation for the Hamiltonian functional $H$. For a positive definite (or at least non-negative) \emph{storage} functional $H$, it immediately shows \emph{passivity}\footnote{Passivity is defined in complete analogy to the finite-dimensional case, see e.\,g. \cite{byrnes1991passivity}, Definition 2.4.} of the PH state representation.

In order to represent a larger class of boundary control problems for systems of two conservation laws, the following proposition generalizes the definition of the Stokes-Dirac structure to the case with \emph{multiple} pairs of in- and outputs on $\partial\Omega$ with \emph{different causalities}.

\begin{prop}
    \label{c04-prop:stokes-dirac-general-boundary-ports}
	Given the $n$-dimensional open and connected domain $\Omega$ with Lipschitz boundary $\partial \Omega$. Consider a partition of $\partial \Omega$ with subsets $\Gamma_i \subset \partial\Omega$, $i = 1,\ldots,n_\Gamma$, and $\hat \Gamma_j \subset \partial\Omega$, $j = 1,\ldots,\hat n_\Gamma$, with orientation according to $\partial \Omega$. Let $\bigcup_{i=1}^{n_\Gamma} \Gamma_i \cup \bigcup_{j=1}^{\hat n_\Gamma} \hat \Gamma_j = \partial \Omega$ and the intersections $\Gamma_i \cap \hat{\Gamma}_j$ be sets of measure zero. Define the boundary flow and effort forms 
	\begin{equation}
		\label{c04-eq:stokes-dirac-mixed-boundary-efforts}
		\begin{array}{rrcl}
			f^{\Gamma}_i =&\!\! \left. \mathrm{tr}\, e^p \right|_{\Gamma_i} 
            ,\\ 
			e^{\Gamma}_i =&\!\! (-1)^p \left. \mathrm{tr}\,  e^q \right|_{\Gamma_i}   
            ,
		\end{array} \qquad
		\begin{array}{rrcl}
			\hat f^{\Gamma}_j =&\!\! (-1)^p \left. \mathrm{tr}\,  e^q \right|_{\hat \Gamma_j}
            ,\\
			\hat e^{\Gamma}_j =&\!\! \left. \mathrm{tr}\,  e^p \right|_{\hat \Gamma_j}             ,
		\end{array}
	\end{equation} 
    as extensions of the effort forms to the corresponding subsets of $\partial \Omega$. The bond space $\mathcal F \times \mathcal E$ is composed of\footnote{For brevity, the domains of the differential forms are written as subscripts, $\Lambda^p_\Omega = \Lambda^p(\Omega)$, etc.}
    \begin{equation}
		\label{c04-eq:spaces-F-E-split-boundary}
		\begin{array}{rcccccc}
			\mathcal F &\;=\;& 
				L^2\Lambda^p_\Omega \times
				L^2\Lambda^{q}_\Omega &\;\times\;&
				L^2\Lambda^{n-p}_{\Gamma_1} \times \cdots \times L^2\Lambda^{n-p}_{\Gamma_{n_\Gamma}} &\;\times\;&
				L^2\Lambda^{n-q}_{\hat \Gamma_1} \times \cdots \times
				L^2\Lambda^{n-q}_{\hat \Gamma_{\hat n_{\Gamma}}}\\
			\mathcal E &\;=\;& 
				H^1\Lambda^{n-p}_{\Omega}\times
				H^1\Lambda^{n-q}_{\Omega} &\;\times\;&
				L^2\Lambda^{n-q}_{\Gamma_1} \times \cdots \times L^2\Lambda^{n-q}_{\Gamma_{n_\Gamma}}&\;\times\;&
				L^2\Lambda^{n-p}_{\hat \Gamma_1} \times \cdots \times
				L^2\Lambda^{n-p}_{\hat \Gamma_{\hat n_{\Gamma}}}.
		\end{array}
	\end{equation}
	The subspace $D \subset \mathcal F \times \mathcal E$, on which \eqref{eq:stokes-dirac-interconnection}  holds and the boundary ports are defined by \eqref{c04-eq:stokes-dirac-mixed-boundary-efforts}, is a Dirac structure.
\end{prop}

\begin{proof}
    First observe that with the choice of boundary ports, and by construction of the subsets $\Gamma_i$ and $\hat \Gamma_j$, the boundary power flow can be expressed as
    \begin{equation}
		\label{c04-eq:stokes-dirac-mixed-power-continuity}
		\sum_{i=1}^{n_\Gamma} \langle e^{\Gamma}_i | f^{\Gamma}_i \rangle_{\Gamma_i} 
		+
		\sum_{j=1}^{\hat n_\Gamma} \langle \hat f^{\Gamma}_j | \hat e^{\Gamma}_j \rangle_{\hat \Gamma_j} 
		=
		(-1)^p \langle e^q | e^p \rangle_{\partial \Omega}.
	\end{equation}
    The proof that the above subspace is a Dirac structure consists of decomposing $\Omega$ and exploiting the \emph{compositionality} property, see Remark 2.2 of \cite{schaft2002hamiltonian}, of the Stokes-Dirac structure on each subset. For a graphical illustration, see Fig. \ref{c04-fig:sketch-omega-gamma}.

	1. Decompose $\Omega$ in a set of $n$-dimensional submanifolds $\Omega_k$ and $\hat \Omega_l$, with the same orientation as $\partial \Omega$ on $\partial\Omega_k \cap \partial \Omega$ and $\partial\hat \Omega_l \cap \partial \Omega$. On each subset, a Stokes-Dirac structure is defined, with alternating causality (but unique on each subset). Then $\bigcup_k \partial\Omega_k \cap \partial \Omega = \bigcup_i \Gamma_i$, $\bigcup_l \partial\hat \Omega_l \cap \partial \Omega = \bigcup_j \hat \Gamma_j$ and $\Gamma_i \cap \partial \Omega_{kl} = \emptyset$, $\hat \Gamma_j \cap \partial \hat \Omega_{lk} = \emptyset$ for all $i,j,k,l$. $\partial \Omega_{kl} = - \partial \hat{\Omega}_{lk}$ denotes the common part of the boundary of $\Omega_k$ and $\hat \Omega_l$, respectively, where the minus sign underscores the inverse orientation by construction. 

    2. Define on each common boundary $\partial \Omega_{kl} = - \partial \hat{\Omega}_{lk}$ the interconnection conditions $f_k^{kl} = e^p|_{\partial \Omega_{kl}} = \hat e_l^{lk}$ and $e_k^{kl} = (-1)^p e^q|_{\partial \Omega_{kl}} = \hat f_l^{lk}$. Then, the terms $\langle e_k^{kl} | f_k^{kl} \rangle_{\partial \Omega_{kl}}$ and $\langle \hat f_l^{lk} | \hat e_l^{lk} \rangle_{\partial \hat \Omega_{lk}}$ in the overall power balance equation cancel each other out due to the reverse integration direction. The interconnection is hence power-preserving, and the composition of the separate Stokes-Dirac structures is, due to their \emph{compositionality} property, itself a Stokes-Dirac structure. 
\end{proof}

\begin{figure}
	\begin{center}
        \includegraphics[scale=0.6]{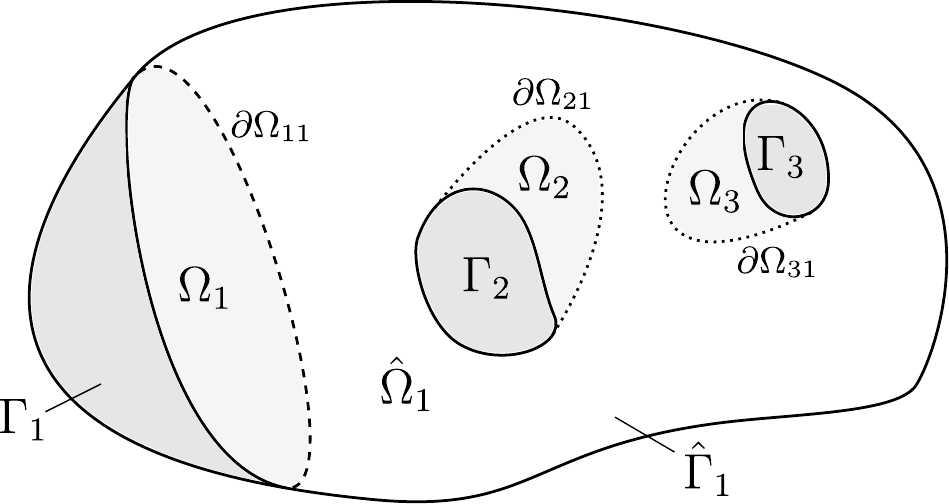}
	\end{center}
    \caption{Sketch of a domain $\Omega \subset \IR^3$ with subdomains $\Omega_1, \Omega_2$, $\Omega_3$ and $\hat \Omega_1$ and a partition of the boundary into $\Gamma_1,\Gamma_2,\Gamma_3$ and $\hat \Gamma_1$.}
	\label{c04-fig:sketch-omega-gamma}
\end{figure}

\begin{rmk}
	In the above proposition, boundary efforts and flows are defined as pure restrictions of either of the distributed efforts to the corresponding subsets $\Gamma_i$, $\hat{\Gamma}_j$ of the boundary. It is, however, also possible to define images of the previous ones under a transformation that preserves the inner product (isometry), e.\,g. \emph{scattering} variables \cite{legorrec2005dirac}.
\end{rmk}

\begin{conv}
	In terms of control, we consider the boundary efforts $u_i^q := e^\Gamma_i$, $i=1,\ldots n_\Gamma$ and $u_j^p := \hat e^\Gamma_j$, $j=1,\ldots,\hat n_\Gamma$, as \emph{boundary input} variables, while the boundary flows $y_i^p := f^\Gamma_i$, $y_j^q := \hat f^\Gamma_j$ are the (power conjugated) \emph{boundary outputs}.
\end{conv}

\subsection{Weak form of the Stokes-Dirac structure for two conservation laws}
\label{subsec:weak-form-stokes-dirac}

The first motivation to study the approximation of distributed parameter PH systems based on their \emph{weak} form is the fact that most of the common numerical methods in engineering, including commercial tools, are based on a Galerkin-type finite-dimensional approximation of the PDEs in weak form\footnote{We use the \emph{weak form} and not the \emph{variational form}. The reason is that we focus on the geometric structure of the equations and do not mention the associated variational problem. We refer to \cite{yoshimura2006dirac} and \cite{vankerschaver2010multi} for the link of the variational problem in Lagrangian mechanics in finite and infinite dimension with a Dirac structure. Note that this link is less obvious e.\,g. for non-Hamiltonian fluids, which are described by a non-canonical structure, see e.\,g. \cite{morrison1998hamiltonian}, \cite{camassa2014variational}.}.
Also in the context of existing works on linear PH distributed parameter systems in one spatial dimension, this perspective is natural. The statements on well-posedness and stability based on the theory of $C_0$ semigroups rely on the \emph{mild} solution of the abstract (operator) differential equation. These solutions, however, corresponds to the weak solutions, as known from the theory of PDEs, see \cite{jacob2012linear}, page 127: \emph{``In fact, the concept of a mild solution is the same as the concept of a weak solution used in the study of partial differential equations.''}
A third point, which motivates to discretize PH distributed parameter systems based on their weak form, is the close relation with \emph{discrete exterior calculus} (i.\,e. the mathematical formalism for integral modeling of conservation laws), which has been used in \cite{seslija2014explicit} for PH systems: ``\emph{Note that the process of integration to suppress discontinuity is, in spirit, equivalent to the idea of weak form used in the Finite Element method}'' \cite{desbrun2008discrete}.
Finally, also in the work of Bossavit on the mixed geometric discretization for computational electromagnetism \cite{bossavit1991differential}, \cite{bossavit1998computational}, the quality of a \emph{weak} formulation is addressed ``\emph{How weak is the weak solution in finite element methods}''  \cite{bossavit1998weak}.

The \emph{weak form} of the Stokes-Dirac structure of Proposition \ref{c04-prop:stokes-dirac-general-boundary-ports} is obtained by a \emph{duality pairing} (which involves the exterior product and integration) on $\Omega$ with \emph{test forms} of appropriate degrees which do \emph{not} vanish on the boundary\footnote{In the weak formulation of boundary value problems, mostly test functions with compact support inside $\Omega$ are chosen such that boundary conditions have to be imposed \emph{directly} on the solution. This is however not mandatory. By test functions which are non-zero on $\partial \Omega$, boundary conditions can be imposed in a \emph{weak} fashion, cf. \cite{quarteroni1994numerical}, Section 14.3.1, p. 483.}. The latter allows for a \emph{weak} imposition of the input boundary conditions $u_i^q = e^\Gamma_i$, $i=1,\ldots,n_\Gamma$ and $u_j^p=\hat e^\Gamma_j$, $j=1,\ldots,\hat n_\Gamma$. 

\begin{dfn}
The weak form of the Stokes-Dirac structure of Proposition \ref{c04-prop:stokes-dirac-general-boundary-ports} is given by the subspace $D \subset \mathcal F \times \mathcal E$ with $\mathcal F$ and $\mathcal E$ as in \eqref{c04-eq:spaces-F-E-split-boundary}, where
    \begin{equation}
        \label{c04-eq:stokes-dirac-weak-1}
    	\begin{alignedat}{2}
    		\langle v^p | f^p \rangle_\Omega &= \langle v^p | (-1)^r \mathrm d e^q \rangle_\Omega \qquad &&\forall v^p  \in H^1 \Lambda^{n-p}(\Omega),
            \\
    		\langle v^q | f^q \rangle_\Omega &= \langle v^q | \mathrm d e^p \rangle_\Omega \qquad  &&\forall v^q \in H^1 \Lambda^{n-q}(\Omega)
    	\end{alignedat}
    \end{equation}    
    holds and the boundary ports are defined by \eqref{c04-eq:stokes-dirac-mixed-boundary-efforts}.
\end{dfn}

Applying integration by parts according to \eqref{c04-eq:integration-by-parts}, we obtain the weak form of the Stokes-Dirac structure with \emph{weak treatment} of the boundary port variables.

\begin{prop}
    The \emph{weak form} of the Stokes-Dirac structure in Proposition \ref{c04-prop:stokes-dirac-general-boundary-ports} with \emph{weak treatment} of the boundary port variables is given by the subset $D \subset \mathcal F \times \mathcal E$, $\mathcal F$ and $\mathcal E$ as in \eqref{c04-eq:spaces-F-E-split-boundary}, where
    \begin{equation}
        \label{c04-eq:stokes-dirac-weak-2}
    	\begin{array}{rcccccc}
    		\langle v^p | f^p \rangle_\Omega &\!\!=\!\!& 
                (-1)^{r+q} \langle \mathrm d v^p | e^q \rangle_\Omega 
                &\!\!-\!\!& 
                \ds (-1)^{r+p+q} \sum_{i=1}^{n_\Gamma} \langle \mathrm{tr}\,v^p | e_i^\Gamma \rangle_{\Gamma_i}
                &\!\!-\!\!& 
                \ds (-1)^{r+p+q} \sum_{j=1}^{\hat{n}_\Gamma} \langle \mathrm{tr}\, v^p | \hat{f}_j^\Gamma \rangle_{\hat \Gamma_j}
            \\
    		\langle v^q | f^q \rangle_\Omega &\!\!=\!\!& 
                (-1)^{p} \langle \mathrm d v^q | e^p \rangle_\Omega 
                &\!\!-\!\!& 
                \ds (-1)^{p} \sum_{i=1}^{n_\Gamma} \langle \mathrm{tr}\,v^q | f_i^\Gamma \rangle_{\Gamma_i}
                &\!\!-\!\!& 
                \ds (-1)^{p} \sum_{j=1}^{\hat{n}_\Gamma} \mathrm{tr}\,\langle v^q | \hat e_j^\Gamma \rangle_{\hat\Gamma_j}
    	\end{array}
    \end{equation}  
    holds for all test forms $v^p \in H^1 \Lambda^{n-p}(\Omega)$ and $v^q\in  H^1 \Lambda^{n-q}(\Omega)$.
\end{prop}

\begin{proof}
    Equation \eqref{c04-eq:stokes-dirac-weak-2} follows from \eqref{c04-eq:stokes-dirac-weak-1} via integration by parts and the identities
    \begin{equation}
    \begin{split}
        (-1)^p \langle v^p | \mathrm{tr}\, e^q \rangle_{\partial \Omega} = \sum_{i=1}^{n_\Gamma} \langle v^p | e_i^\Gamma \rangle_{\Gamma_i} + \sum_{j=1}^{\hat n_\Gamma} \langle v^p | \hat f_j^\Gamma \rangle_{\hat \Gamma_j},\\
        \langle v^q | \mathrm{tr}\, e^p \rangle_{\partial \Omega} = \sum_{i=1}^{n_\Gamma} \langle v^q | f_i^\Gamma \rangle_{\Gamma_i} + \sum_{j=1}^{\hat n_\Gamma} \langle v^q | \hat e_j^\Gamma \rangle_{\hat \Gamma_j}.    
    \end{split}
    \end{equation}
    The latter are due to the definition \eqref{c04-eq:stokes-dirac-mixed-boundary-efforts} of boundary port variables and the definition of the subsets $\Gamma_i$, $\hat \Gamma_j$, which cover $\partial \Omega$ and whose intersections have zero measure.
\end{proof}

\begin{rmk}
    The latter representation of the Stokes-Dirac structure -- if considered on a single control volume -- is suitable for  \emph{discontinuous Galerkin} schemes, see e.\,g. \cite{hesthaven2007nodal}, where the boundary terms are replaced by suitable \emph{numerical fluxes}.
\end{rmk}

\begin{rmk}
Note that the two conservation laws are described by the canonical differential operator $\mathcal J$ in \eqref{eq:stokes-dirac-interconnection}, which contains only exterior derivatives. The weak form of the Stokes-Dirac structure is defined based on the \emph{metric-independent} duality product arising from the integration-by-parts formula \eqref{c04-eq:integration-by-parts}, applied to \emph{both} conservation laws in Eq. \eqref{eq:stokes-dirac-interconnection}. This is a difference to other approaches like the mixed mimetic discretization of the Stokes flow in \cite{kreeft2013mixed} or the structure-preserving PH discretization in \cite{farle2013porta}, where integration by parts is only applied to the equations that contain the \emph{metric-dependent} codifferential.
\end{rmk}

Using the effort forms as test forms, $v^p=e^p$, $v^q=e^q$, and adding both equations of \eqref{c04-eq:stokes-dirac-weak-2}, we obtain after some reformulations and exploiting \eqref{c04-eq:stokes-dirac-mixed-power-continuity}, 
\begin{equation}
	\label{c04-eq:power-continuity-weak}
	\langle e^p | f^p \rangle_\Omega + \langle e^q | f^q \rangle_\Omega 
	+ (-1)^{p} \langle e^q | e^p \rangle_{\partial \Omega} = 0,
\end{equation}
or, with the definition of boundary port variables,
\begin{equation}
	\label{c04-eq:power-continuity-weak-ports}
	\langle e^p | f^p \rangle_\Omega + \langle e^q | f^q \rangle_\Omega 
	+ \sum_{i=1}^{n_\Gamma} \langle e^{\Gamma}_i | f^{\Gamma}_i \rangle_{\Gamma_{i}}
	+ \sum_{j=1}^{\hat n_\Gamma} \langle \hat f^{\Gamma}_j | \hat e^{\Gamma}_j \rangle_{\hat \Gamma_{j}}
     = 0,
\end{equation}
which corresponds to the initially derived power continuity equation \eqref{c04-eq:power-continuity}.

We have arrived at a weak representation of the Stokes-Dirac structure of Proposition \ref{c04-prop:stokes-dirac-general-boundary-ports}, which suits to establish discretized mixed Galerkin models of PH systems of two conservation laws.

\section[Geometric discretization]{Geometric discretization of the port-Hamiltonian system}

In this section, we study the \emph{mixed Galerkin approximation} of the Stokes-Dirac structure in weak form as defined in the previous section. Expressing \eqref{c04-eq:stokes-dirac-weak-2} in approximation subspaces that retain the notion of the duality products as \emph{power pairings}, and defining \emph{in- and output port variables} whose pairings represent the transmitted power over the boundary, we obtain a finite number of equations for the Galerkin coefficients. On the so-defined subset of the discrete bond space, a discrete power continuity equation holds. Due to the different dimensions of the geometrically chosen approximation spaces, the bilinear forms that define the power pairings are, however, \emph{degenerate}. To obtain a finite-dimensional \emph{Dirac structure} with \emph{non-degenerate} power pairings, which is the basis to formulate a PH approximation model in \emph{state space form}, we introduce  \emph{power-preserving} mappings of the discrete flow and effort vectors onto finite-dimensional spaces of appropriate, identical dimension. The geometric discretization is completed by a \emph{consistent discretization} of the constitutive equations.

\subsection{Mixed Galerkin approximation with boundary port variables}
\label{subsec:mixed-galerkin}

We introduce the mixed Galerkin approximation of the weak form of the Stokes-Dirac structure for a system of two conservation laws. Mixed or \emph{duality} methods have been introduced to include constraints like the divergence-freedom of flows or to take account for the precise approximation of additional physical variables in the numerical approximation, see \cite{brezzi1991mixed} as a classical reference for mixed finite elements. The duality of the power variables in the Stokes-Dirac structure imposes the use of a mixed approximation. 

The \emph{boundary inputs} are weakly imposed as boundary conditions, and appear \emph{immediately} in the finite-dimensional system of equations for the Galerkin degrees of freedom. \emph{Boundary outputs} are constructed via the discrete power balance. This point of view, which leads to state space models in input-output form, distinguishes the structure-preserving discretization of PH systems from classical approaches to the numerical approximation of PDEs.

For the compactness of notation, we omit to explicitly write out the trace operator on the subsets of the boundary, i.\,e. $\langle v^p | e^q \rangle_{\Gamma_i} := \langle \mathrm{tr}\,v^p | \mathrm{tr}\, e^q \rangle_{\Gamma_i}$ etc. in the sequel. We start with the representation\footnote{In the sequel, we denote portions of the boundary with greek indices and elements of the approximation subspaces with latin indices.}
    \begin{equation}
        \label{c04-eq:stokes-dirac-weak-3}
    	\begin{array}{rcccccc}
    		\langle v^p | f^p \rangle_\Omega &\!\!=\!\!& 
                (-1)^{r+q} \langle \mathrm d v^p | e^q \rangle_\Omega 
                &\!\!-\!\!& 
                \ds (-1)^{r+q} \sum_{\mu=1}^{n_\Gamma} \langle v^p |  e^q \rangle_{\Gamma_\mu}
                &\!\!-\!\!& 
                \ds (-1)^{r+q} \sum_{\nu=1}^{\hat{n}_\Gamma} \langle v^p | e^q \rangle_{\hat \Gamma_\nu}
            \\
    		\langle v^q | f^q \rangle_\Omega &\!\!=\!\!& 
                (-1)^{p} \langle \mathrm d v^q | e^p \rangle_\Omega 
                &\!\!-\!\!& 
                \ds (-1)^{p} \sum_{\mu=1}^{n_\Gamma} \langle v^q | e^p \rangle_{\Gamma_\mu}
                &\!\!-\!\!& 
                \ds (-1)^{p} \sum_{\nu=1}^{\hat{n}_\Gamma} \langle v^q | e^p \rangle_{\hat\Gamma_\nu},
    	\end{array}
    \end{equation}
i.\,e. \eqref{c04-eq:stokes-dirac-weak-2} without the explicit denomination of the boundary port variables. For a \emph{mixed Galerkin} approximation of the Stokes-Dirac structure, we
\begin{itemize}
	\item use different (\emph{dual} or \emph{mixed}) bases to approximate the spaces of flow and effort forms and
	\item from these bases, we choose the appropriate ones to approximate the test forms (\emph{Galerkin} method). 
\end{itemize}
Taking the test forms from the effort bases is the most obvious choice for the approximation of the Stokes-Dirac structure, as the resulting (discrete) duality pairings have an immediate interpretation in terms of power, see Eq. \eqref{c04-eq:power-continuity-weak}.

\subsubsection{Approximation problem and compatibility condition}
The flow differential forms will be approximated by linear combinations of the basis forms of the subspaces
\begin{equation}
    \begin{array}{rcccl}
    \label{c04-eq:galerkin-subspaces-flows}
    \Psi_h^p &=\;&  \textrm{span} \{ \psi_1^p, \ldots, \psi^p_{N_p} \} &\subset\;& L^2\Lambda^p(\Omega),\\
    \Psi_h^q
    &=\;& \textrm{span} \{ \psi_1^p, \ldots, \psi^q_{N_q} \} &\subset\;& L^2\Lambda^q(\Omega).
    \end{array}
\end{equation}
The subspaces for the effort and test forms are, accordingly,
\begin{equation}
    \begin{array}{rcccl}
    \label{c04-eq:galerkin-subspaces-efforts}
    \Phi^p_h &=\;& \textrm{span} \{ \varphi_1^p, \ldots, \varphi^p_{M_p} \} &\subset\;& H^1 \Lambda^{n-p}(\Omega),\\
    \Phi_h^q
    &=\;& \textrm{span} \{ \varphi_1^q, \ldots, \varphi^q_{M_q} \} &\subset\;& H^1\Lambda^{n-q}(\Omega).
    \end{array}
\end{equation}
From the trace theorem for $H^1$ spaces (as discussed in Subsection \ref{subsubsec:lebesgue-sobolev}), we know that the extension of the latter spaces to the boundary is $L^2$.
The subscript $h>0$ denotes the discretization parameter\footnote{Which corresponds to the spatial extent of finite elements or the inverse of the polynomial approximation order.} and we assume an \emph{appropriate} choice of approximation spaces, i.\,e. for a given functional space $V$ and its approximation $V_h$  (see \cite{quarteroni1994numerical}, Section 5.2) it is true that $\inf_{v_h \in V_h} \| v - v_h \| \rightarrow 0$ for all $v\in V$ if $h\rightarrow 0$.
The \emph{mixed Galerkin approximation problem} is as follows: Find approximate flow and effort forms
\begin{equation}
\label{c04-eq:galerkin-approx-flows}
\begin{split}
   	f^{p}_h(\vec z) &= \sum_{k=1}^{N_p} f^{p}_{k} \psi^p_{k}(\vec z) 
        = \langle \vec f^p | \bs \psi^p(\vec z) \rangle 
        \in \Psi_h^p,\\
   	f^{q}_h(\vec z) &= \sum_{l=1}^{N_q} f^{q}_{l} \psi^q_{l}(\vec z) 
        = \langle \vec f^q | \bs \psi^q(\vec z) \rangle 
        \in \Psi_h^q,
\end{split}
\end{equation}
and
\begin{equation}
\label{c04-eq:galerkin-approx-efforts}
\begin{split}
   	e^{p}_h(\vec z) &= \sum_{i=1}^{M_p} e^{p}_{i} \varphi^p_{i}(\vec z) 
        = \langle \vec e^p | \bs \varphi^p(\vec z) \rangle 
        \in \Phi_h^p,\\
   	e^{q}_h(\vec z) &= \sum_{j=1}^{M_q} e^{q}_{j} \varphi^q_{j}(\vec z) 
        = \langle \vec e^q | \bs \varphi^q(\vec z) \rangle 
        \in \Phi_h^q,
\end{split}
\end{equation}
where $\langle \cdot| \cdot \rangle$ denotes the standard inner product on $\IR^n$ as in Definition \ref{c04-def:fin-dim-dirac-structure}, such that 
\begin{equation}
	\label{c04-eq:stokes-dirac-weak-form-2h}
	\begin{array}{rcccccl}
		\langle v^p_h | f^p_h \rangle_\Omega &\!\!=\!\!& 
            (-1)^{r+q} \langle \textrm d v^p_h | e^q_h \rangle_\Omega &\!\!-\!\!&
            \ds (-1)^{r+q} \sum_{\mu=1}^{n_\Gamma}  \langle v^p_h | e^q_h \rangle_{\Gamma_{\mu}} &\!\!-\!\!&
            \ds (-1)^{r+q} \sum_{\nu=1}^{\hat n_\Gamma}  \langle v^p_h | e^q_h \rangle_{\hat \Gamma_{\nu}},
		\\
		\langle v^q_h | f^q_h \rangle_\Omega &\!\!=\!\!& 
            (-1)^{p} \langle \textrm d v^q_h | e^p_h \rangle_\Omega &\!\!-\!\!& 
            \ds (-1)^{p} \sum_{\mu=1}^{n_\Gamma}  \langle v^q_h | e^p_h \rangle_{\Gamma_{\mu}} &\!\!-\!\!&
            \ds (-1)^{p} \sum_{\nu=1}^{\hat n_\Gamma} \langle v^q_h | e^p_h \rangle_{\hat\Gamma_{\nu}}
	\end{array}
\end{equation}
hold for all $v_h^p \in \Phi_h^p$, $v_h^q \in \Phi_h^q$. 
The \emph{discrete flow and effort vectors} 
\begin{equation}
    \begin{array}{rcl}
        \vec f^p &=& [f_1^p,\ldots, f_{N_p}^p]^T,\\ 
        \vec f^q &=& [f_1^q,\ldots, f_{N_q}^q]^T,
    \end{array}
    \quad \text{and} \quad
    \begin{array}{rcl}
        \vec e^p &=& [e_1^p,\ldots, e_{M_p}^p]^T, \\
        \vec e^q &=& [e_1^q,\ldots, e_{M_q}^q]^T
    \end{array}
\end{equation} 
contain the approximation coefficients, and the vectors (we omit the argument $\vec z$ in the sequel)
\begin{equation}
    \begin{array}{rcl}
        \bs \psi^p(\vec z) &=& [\psi_1^p(\vec z),\ldots, \psi_{N_p}^p(\vec z)]^T,\\ 
        \bs \psi^q(\vec z) &=& [\psi_1^q(\vec z),\ldots, \psi_{N_q}^q(\vec z)]^T,
    \end{array}
    \quad \text{and} \quad
    \begin{array}{rcl}
        \bs \varphi^p(\vec z) &=& [\varphi_1^p(\vec z),\ldots, \varphi_{M_p}^p(\vec z)]^T, \\
        \bs \varphi^q(\vec z) &=& [\varphi_1^q(\vec z),\ldots, \varphi_{M_q}^q(\vec z)]^T
    \end{array}
\end{equation} 
contain the approximation basis forms. The flow variables are understood as time derivatives of the distributed conserved quantities with negative sign, see \eqref{eq:phs-flows}. Thus, they are approximated in the same spatial bases,
\begin{equation}
\label{c04-eq:galerkin-approx-states}
\begin{split}
   	{p}_h(\vec z) &= \sum_{k=1}^{N_p} {p}_{k} \psi^p_{k}(\vec z) 
        = \langle \vec p| \bs \psi^p(\vec z) \rangle 
        \in \Psi_h^p, \\
   	{q}_h(\vec z) &= \sum_{l=1}^{N_q} {q}_{l} \psi^q_{l}(\vec z) 
        = \langle \vec q| \bs \psi^q(\vec z) \rangle 
        \in \Psi_h^q,
\end{split}
\end{equation}
and 
\begin{equation}
    \vec p = [p_1,\ldots, p_{N_p}]^T,
    \quad \text{and} \quad
    \vec q = [q_1,\ldots, q_{N_q}]^T,
\end{equation} 
denote the vectors of \emph{discrete} or \emph{integral conserved quantities}.

The mixed Galerkin approximation \eqref{c04-eq:stokes-dirac-weak-form-2h} of \eqref{c04-eq:stokes-dirac-weak-3} is \emph{exact} for flow and effort forms in the approximation spaces \eqref{c04-eq:galerkin-subspaces-flows}, \eqref{c04-eq:galerkin-subspaces-efforts} (in these subspaces, the residual error vanishes), if the following \emph{compatibility conditions} hold:
\begin{equation}
    \label{c04-eq:compatibility-conditions}
    \begin{split}
    \mathrm{span}\{ \psi^p_{1}, \ldots, \psi^p_{N_p} \} &= \mathrm{span}\{ \d \varphi^q_{1}, \ldots, \d \varphi^q_{M_q} \},\\
    \mathrm{span}\{ \psi^q_{1}, \ldots, \psi^q_{N_q} \} &= \mathrm{span}\{ \d \varphi^p_{1}, \ldots, \d \varphi^p_{M_p} \}.
    \end{split}
\end{equation}
In contrast to \cite{golo2004hamiltonian} (Assumptions 3 and 7), this \emph{compatibility of forms}\footnote{In other words, this is the de Rham property of the sequence of approximation subspaces.} is understood \emph{in the weak sense}. This means, more precisely -- consider the original weak formulation \eqref{c04-eq:stokes-dirac-weak-1} and the definition of the weak exterior derivative -- that for all test forms with compact support inside $\Omega$, i.\,e. $v^p \in H^1_0 \Lambda^{n-p}(\Omega)$, $v^q \in H^1_0 \Lambda^{n-q}(\Omega)$, there exist constants $a^p_{k}, a^q_{l}$, $b^p_{i}, b^q_{j}$ such that
\begin{equation}
    \begin{split}
        \sum_{k=1}^{N_p} a_{k}^p \langle v^p | \psi_{k}^p \rangle_\Omega + 
        \sum_{j=1}^{M_q} b_{j}^q \langle v^p | \d \varphi_{j}^q \rangle_\Omega &= 0,\\
        \sum_{l=1}^{N_q} a_{l}^q \langle v^q | \psi_{l}^q \rangle_\Omega + 
        \sum_{i=1}^{M_p} b_{i}^p \langle v^q | \d \varphi_{i}^p \rangle_\Omega &= 0.
    \end{split}
\end{equation}

\subsubsection{Approximation of the Stokes-Dirac structure}
\label{subsec:approx-stokes-dirac}

We approximate the weak formulation \eqref{c04-eq:stokes-dirac-weak-3} of the Stokes-Dirac structure by substituting the flow and effort forms with their finite-dimensional approximations \eqref{c04-eq:galerkin-approx-flows}, \eqref{c04-eq:galerkin-approx-efforts}. By choosing the test forms from the effort bases,
\begin{equation}
    \label{c04-eq:test-forms-approx}
    v_h^p = \langle \vec v^p | \bs \varphi^p \rangle, \quad
    v_h^q = \langle \vec v^q | \bs \varphi^q \rangle, \qquad
    \vec v^p \in \IR^{M_p}, \; 
    \vec v^q \in \IR^{M_q},
\end{equation}
the finite-dimensional inner products in the approximation will retain the interpretation in terms of \emph{power}. We obtain (the exterior derivative applies element-wise to a vector of differential forms)
\begin{equation}
\label{c04-eq:approx-system-of-eqs}
\begin{split}
	\Big\langle \langle \vec v^p | \bs \varphi^p \rangle \;\Big|\; \langle \vec f^p | \bs \psi^p \rangle \Big\rangle_{\Omega} - 
    	(-1)^{r+q} 
    	\Big\langle \langle \vec v^p | \d \bs\varphi^p \rangle \;\Big|\; \langle \vec e^q | \bs \varphi^q \rangle \Big\rangle_{\Omega} \hspace{6cm} \\
	+ (-1)^{r+q}
    \sum_{\mu=1}^{n_\Gamma}
        \Big\langle \langle \vec v^p | \bs\varphi^p \rangle \;\Big|\; \langle \vec e^q | \bs \varphi^q \rangle \Big\rangle_{\Gamma_\mu} 
	+ (-1)^{r+q}
    \sum_{\nu=1}^{\hat n_\Gamma}
        \Big\langle \langle \vec v^p | \bs\varphi^p \rangle \;\Big|\; \langle \vec e^q | \bs \varphi^q \rangle \Big\rangle_{\hat \Gamma_\nu}
    = 0,\\
	\Big\langle \langle \vec v^q | \bs \varphi^q \rangle \;\Big|\; \langle \vec f^q | \bs \psi^q \rangle \Big\rangle_{\Omega} - 
    	(-1)^{p} 
    	\Big\langle \langle \vec v^q | \d \bs\varphi^q \rangle \;\Big|\; \langle \vec e^p | \bs \varphi^p \rangle \Big\rangle_{\Omega} \hspace{6cm}\\
	+ (-1)^{p}
    \sum_{\mu=1}^{n_\Gamma}
        \Big\langle \langle \vec v^q | \bs\varphi^q \rangle \;\Big|\; \langle \vec e^p | \bs \varphi^p \rangle \Big\rangle_{\Gamma_\mu} 
	+ (-1)^{p}
    \sum_{\nu=1}^{\hat n_\Gamma}
        \Big\langle \langle \vec v^q | \bs\varphi^q \rangle \;\Big|\; \langle \vec e^p | \bs \varphi^p \rangle \Big\rangle_{\hat \Gamma_\nu}
    = 0.
\end{split}
\end{equation}
Evaluating the integrals over the products of basis forms, the system of equations can be written
\begin{equation}
	\label{c04-eq:finit-dim-var-problem}
	\begin{split}
		\Big \langle \vec v^p \Big| \vec M_p \vec f^p \Big \rangle  
            + \Big \langle \vec v^p \Big|
            \Big(\vec K_p + \sum_{\mu=1}^{n_\Gamma}\vec L_p^\mu + \sum_{\nu=1}^{\hat n_\Gamma} \hat{\vec L}_p^\nu \Big) \vec e^q \Big \rangle &= 0,
         \\		
		\Big \langle \vec v^q \Big| \vec M_q \vec f^q \Big \rangle  
            + \Big \langle \vec v^q \Big|  
            \Big (\vec K_q + \sum_{\mu=1}^{n_\Gamma}\vec L_q^\mu + \sum_{\nu=1}^{\hat n_\Gamma} \hat{\vec L}_q^\nu \Big) \vec e^p \Big \rangle &= 0,
	\end{split}
\end{equation}
with the coefficient matrices $\vec M_p \in \mathbb R^{M_p\times N_p}$, $\vec M_q \in \mathbb R^{M_q\times N_q}$, $\vec K_p, \vec L_p^\mu, \hat{\vec L}_p^\nu \in \mathbb R^{M_p\times M_q}$, $\vec K_q, \vec L_q^\mu, \hat{\vec L}_q^\nu \in \mathbb R^{M_q\times M_p}$, $\mu=1,\ldots,n_\Gamma$, $\nu=1,\ldots, \hat n_\Gamma$, composed of the elements
\begin{equation}
    \label{c04-eq:elements-M-K-L}
        \begin{array}{rcrrcr}
		[\vec M_p]_{ik} &=& \langle \varphi^p_{i} | \psi^p_{k} \rangle_\Omega,& \qquad 
			[\vec M_q]_{jl} &=& \langle \varphi^q_{j} | \psi^q_{l} \rangle_\Omega,\\[1ex]
		[\vec K_p]_{ij} &=& - (-1)^{r+q} \langle \d \varphi^p_{i} | \varphi^q_{j} \rangle_\Omega, &\qquad 
			[\vec K_q]_{ji} &=& - (-1)^p \langle \d \varphi^q_{j} | \varphi^p_{i} \rangle_\Omega,\\[1ex]
		[\vec L_p^\mu]_{ij} &=& (-1)^{r+q} \langle \varphi^p_{i} | \varphi^q_{j} \rangle_{\Gamma_\mu}, &\qquad 
			[\vec L_q^\mu]_{ji} &=& (-1)^p \langle \varphi^q_{j} | \varphi^p_{i} \rangle_{\Gamma_\mu},\\[1ex]
        [\hat{\vec L}_p^\nu]_{ij} &=& (-1)^{r+q} \langle \varphi^p_{i} | \varphi^q_{j} \rangle_{\hat\Gamma_\nu}, &\qquad 
			[\hat{\vec L}_q^\nu]_{ji} &=& (-1)^p \langle \varphi^q_{j} | \varphi^p_{i} \rangle_{\hat\Gamma_\nu}.
    \end{array}
\end{equation}
The equations of \eqref{c04-eq:finit-dim-var-problem} have to hold for arbitrary $\vec v^p \in \IR^{M_p}$, $\vec v^q \in \IR^{M_q}$, which yields the  equations for the discrete flow and effort vectors
\begin{equation}
    \label{c04-eq:finit-dim-var-problem-matrix}
    \begin{split}
        \vec M_p \vec f^p + (\vec K_p + \vec L_p) \vec e^q &= \vec 0,\\
        \vec M_q \vec f^q + (\vec K_q + \vec L_q) \vec e^p &= \vec 0.
    \end{split}
\end{equation}
By skew-symmetry of the wedge product, see Eq. \eqref{c04-eq:skew-symmetry-wedge-product}, it is straightforward to show that 
\begin{equation}
    [\vec L_p^\mu]_{ij} = [\vec L_q^\mu]_{ji},\qquad     
    [\hat{\vec L}_p^\nu]_{ij} = [\hat{\vec L}_q^\mu]_{ji},
\end{equation} 
i.\,e. $\vec L_p^\mu = (\vec L_q^\mu)^T$ and $\hat{\vec L}_p^\nu = (\hat{\vec L}_q^\nu)^T$. By defining 
\begin{equation}
    \label{c04-eq:definition-Lp-Lq}
    \vec L_p = \sum_{\mu=1}^{n_\Gamma} \vec L_p^\mu + \sum_{\nu=1}^{\hat n_\Gamma} \hat{\vec L}_p^\nu, \qquad 
    \vec L_q = \sum_{\mu=1}^{n_\Gamma} \vec L_q^\mu + \sum_{\nu=1}^{\hat n_\Gamma} \hat{\vec L}_q^\nu,
\end{equation}
we can show the following.
\begin{lem}
    The matrices $\vec K_p, \vec K_q$ and $\vec L_p, \vec L_q$ are related via $[\vec K_p\! +\! \vec L_p]_{ij} +
            [\vec K_q\! +\! \vec L_q]_{ji}  =
            [\vec L_p]_{ij} = [\vec L_q]_{ji}$, i.\,e.
    \begin{equation}
        \label{c04-eq:relations-K-L}
        (\vec K_p + \vec L_p) + (\vec K_q + \vec L_q)^T = \vec L_p = \vec L_q^T.
    \end{equation}
\end{lem}

\begin{proof}
By the definition \eqref{c04-eq:definition-Lp-Lq} and the corresponding parts of \eqref{c04-eq:elements-M-K-L}, the elements of $\vec L_p, \vec L_q$ are duality products over the effort basis forms on the complete boundary $\partial \Omega$. Thus, we have that 
\begin{multline}
    [\vec K_p + \vec L_p]_{ij}  + [\vec K_q + \vec L_q]_{ji} 
     =\\
    - (-1)^{r+q} \langle \d \varphi^p_i | \varphi^q_j \rangle_\Omega 
    + (-1)^{r+q} \langle \varphi^p_i | \varphi^q_j \rangle_{\partial \Omega}
    - (-1)^p \langle \d \varphi^q_j | \varphi^p_i \rangle_\Omega 
    + (-1)^p \langle \varphi^q_j | \varphi^p_i \rangle_{\partial \Omega}.    
\end{multline}
Using skew-symmetry of the wedge product \eqref{c04-eq:skew-symmetry-wedge-product} and the integration-by-parts formula for differential forms \eqref{c04-eq:integration-by-parts}, the right hand side can be rewritten as
\begin{equation}
    (-1)^p \langle \varphi^q_{j}(z) | \varphi^p_{i}(z) \rangle_{\partial \Omega}
    = [\vec L_q]_{ji}
    = [\vec L_p]_{ij},
\end{equation}
which proves the claim.
\end{proof}

\begin{dfn}
    The quadratic forms over the discrete effort vectors with the corresponding matrices $\vec L_p, \vec L_p^\mu, \hat{\vec L}_p^\nu$ and $\vec L_q, \vec L_q^\mu, \hat{\vec L}_q^\nu$ describe the approximate power transmitted over the boundary $\partial \Omega$ or its parts. We refer to these matrices as \emph{boundary power matrices}. 
\end{dfn}

The boundary power matrices $\vec L_p = \vec L_q^T$, will have reduced rank. The reason is that basis forms for interior effort degrees of freedom will be, in general, zero on the boundary. This is true e.\,g. for finite elements, see Section \ref{sec:example}, and also for the 1D geometric pseudo-spectral collocation method \cite{moulla2012pseudo}.

\subsubsection{Discrete boundary port variables}

To define the \emph{pairs of discrete boundary port variables} that will be assigned either the role of boundary controls or the role of outputs on the boundary subsets, we characterize  mappings on the spaces of discrete efforts variables.

\begin{dfn}
    The vectors of \emph{discrete boundary port variables}\footnote{Discrete boundary variables have index $b$, in contrast to index $\partial$ for the original distributed quantities.}  $\vec e^{b,\mu}, \vec f^{b,\mu}_0 \in \IR^{M_{b}^\mu}$ and $\hat{\vec e}^{b,\nu}, \hat{\vec f}^{b,\nu}_0 \in \IR^{\hat M_{b}^\nu}$,  associated with the boundary subdomains $\Gamma_\mu \subset \partial \Omega$, $\mu=1,\ldots, n_\Gamma$, $\hat{\Gamma}_\nu \subset \partial \Omega$, $\nu=1,\ldots,\hat{n}_\Gamma$, satisfy
    \begin{equation}
        \label{c04-eq:boundary-power-flow-decomposition}
        \langle \vec e^q | \vec L_q^\mu \vec e^p \rangle =:
        \langle \vec e^{b,\mu} | \vec f^{b,\mu}_0 \rangle,
        \qquad
        \langle \vec e^p | \hat{\vec L}_p^\nu \vec e^q \rangle =:
        \langle \hat{\vec e}^{b,\nu} | \hat{\vec f}^{b,\nu}_0 \rangle,
    \end{equation}
    i.\,e. their duality products (which are standard Euclidean scalar products on the finite-dimensional bond space) match the discrete expression of the power flow over $\Gamma_\mu$ and $\hat{\Gamma}_\nu$, respectively.
\end{dfn}

   We decompose the boundary power matrices for each boundary subdomain in matrix products
    \begin{equation}
        \label{c04-eq:decompositions-L}
        \vec L_q^\mu = (\vec T_{q}^\mu)^T \vec S_{p,0}^\mu, \qquad
        \hat{\vec L}_p^\nu =
        (\hat{\vec T}_{p}^\nu)^T \hat{\vec S}_{q,0}^\nu.
    \end{equation}
    The \emph{boundary trace matrices}\footnote{This denomination refers to the trace theorem for the extension of a $H^m$ function to the boundary.} $\vec T_{q}^\mu \in \IR^{M_{b}^\mu \times M_{q}}$, $\hat{\vec T}_{p}^\nu \in \IR^{\hat M_{b}^\nu \times M_{p}}$ define the effort degrees of freedom 
    \begin{equation}
        \label{c04-eq:discrete-boundary-efforts}
            \vec e^{b,\mu} = \vec T_{q}^\mu \vec e^q,
            \qquad
            \hat{\vec e}^{b,\nu} = \hat{\vec T}_{p}^\nu \vec e^p
    \end{equation}
    that lie on the boundary and are assigned the roles of \emph{input variables}. We call $\vec S_{p,0}^\mu \in \IR^{M_{b}^\mu \times M_{p}}$,  $\hat{\vec S}_{q,0}^\nu \in \IR^{\hat M_{b}^\nu \times M_{q}}$ the \emph{collocated boundary output} matrices. They define the boundary flow variables 
    \begin{equation}
        \label{c04-eq:discrete-boundary-flows}
        \vec f^{b,\mu}_0 = \vec S_{p,0}^\mu \vec e^p,
        \qquad
        \hat{\vec f}^{b,\nu}_0 = \hat{\vec S}_{q,0}^\nu \vec e^q,
    \end{equation}    
    which, together with the discrete efforts \eqref{c04-eq:discrete-boundary-efforts}, satisfy \emph{exactly} the discrete power balance \eqref{c04-eq:boundary-power-flow-decomposition} on the different portions of the boundary\footnote{The subscript $0$ indicates that these discrete output variables will be re-defined when we derive a PH state space model based on a (non-degenerate) Dirac structure.}.
Because of 
\begin{equation}
    \label{eq:consistent-boundary-power}
    \begin{split}
        \langle e^{\Gamma}_\mu | f^{\Gamma}_\mu \rangle_{\Gamma_\mu} =
        (-1)^p \langle e^q | e^p \rangle_{\Gamma_\mu} 
        &\approx
        (-1)^p \Big\langle \langle \vec e^q | \bs \varphi^q \rangle \Big| \langle \vec e^p | \bs \varphi^p \rangle \Big \rangle_{\Gamma_\mu} =
        \langle \vec e^q | \vec L_q^\mu \vec e^p \rangle = \langle \vec e^{b,\mu} | \vec f^{b,\mu}_0 \rangle,\\
        \langle \hat e^{\Gamma}_\nu | \hat f^{\Gamma}_\nu \rangle_{\hat \Gamma_\nu} =
        (-1)^p \langle e^p | e^q \rangle_{\hat\Gamma_{\nu}} 
        &\approx
        (-1)^p \Big\langle \langle \vec e^p | \bs \varphi^p \rangle \Big| \langle \vec e^q | \bs \varphi^q \rangle \Big \rangle_{\hat\Gamma_{\nu}} =
        \langle \vec e^p | \hat{\vec L}_p^\nu \vec e^q \rangle = \langle \hat{\vec e}^{b,\nu} | \hat{\vec f}^{b,\nu}_0 \rangle,
    \end{split}
\end{equation}
the definition of discrete boundary port variables is consistent with the distributed definition \eqref{c04-eq:power-continuity-weak-ports}. Summation over the individual boundary power matrices according to \eqref{c04-eq:definition-Lp-Lq}, yields a matrix equation that expresses the boundary power balance,
    \begin{equation}
        \label{c04-eq:relation-T-L-b}
        \vec T_{q}^T \vec S_{p,0} + \hat{\vec S}_{q,0}^T \hat{\vec T}_{p} =  \vec L_q = \vec L_p^T,
    \end{equation}
    where
    \begin{equation}
        \vec T_q  = \bmat{\vec T_{q}^1\\ \vdots \\ \vec T_{q}^{n_\Gamma}}, \quad
        \vec S_{p,0}  = \bmat{\vec S_{p,0}^1\\ \vdots \\ \vec S_{p,0}^{n_\Gamma}}, \qquad
        \hat{\vec S}_{q,0}  = \bmat{\hat{\vec S}_{q,0}^1\\ \vdots \\ \hat{\vec S}_{q,0}^{\hat n_\Gamma}}, \quad
        \hat{\vec T}_p  = \bmat{\hat{\vec T}_{p}^1\\ \vdots \\ \hat{\vec T}_{p}^{\hat n_\Gamma}}.
    \end{equation}  
The overall vectors of discrete boundary port variables comprise the contributions of each boundary subset with corresponding causality\footnote{The causality of a pair of port variables changes if the role of in- and output is permuted.},
\begin{equation}
    \label{c04-eq:discrete-boundary-maps}
    \vec e^b = \vec T_q \vec e^q, \quad \vec f^b_0 = \vec S_{p,0} \vec e^p, \qquad
    \hat{\vec e}^b = \hat{\vec T}_p \vec e^p, \quad \hat{\vec f}^b_0 = \hat{\vec S}_{q,0} \vec e^q,
\end{equation}
with
\begin{equation}
    \label{c04-eq:discrete-boundary-flows-efforts-vectors}
    \vec e^b = \bmat{ \vec e^{b,1}\\ \vdots \\ \vec e^{b,n_\Gamma}}, \quad    
    \vec f^b_0 = \bmat{ \vec f^{b,1}_0\\ \vdots \\ \vec f^{b,n_\Gamma}_0}, \qquad
    \hat{\vec e}^b = \bmat{ \vec e^{b,1}\\ \vdots \\ \vec e^{b,\hat n_\Gamma}}, \quad    
    \hat{\vec f}^b_0 = \bmat{ \vec f^{b,1}_0\\ \vdots \\ \vec f^{b,\hat n_\Gamma}_0}.
\end{equation}

\subsubsection{Power balance on the discrete bond space}
The vectors of discrete flows and efforts $\vec f^{p/q}$, $\vec e^{p/q}$ that satisfy \eqref{c04-eq:finit-dim-var-problem-matrix}, together with the discrete boundary ports of different causality, define a subset of the bond space 
    \begin{equation}
        \mathcal F \times \mathcal E =
               \IR^{N_p} \times \IR^{N_q} \times \IR^{M_{b}} \times \IR^{\hat M_{b}}
                \times
               \IR^{M_p} \times \IR^{M_q} \times \IR^{M_{b}} \times \IR^{\hat M_{b}},
    \end{equation}
with $M_b = \sum_{\mu=1}^{n_\Gamma} M_{b}^\mu$, $\hat{M}_b = \sum_{\nu=1}^{\hat n_\Gamma} \hat M_{b}^\nu$. On this subspace, a discrete power balance holds that approximates the continuous one \eqref{c04-eq:power-continuity-weak-ports}.

\begin{prop}
    The subspace
    \begin{equation}
    \label{eq:subspace-degenerate}
        D = \{ (\vec f^p, \vec f^q, \vec f^{b}_0, \hat{\vec f}^b_0,
        \vec e^p, \vec e^q, \vec e^b, \hat{\vec e}^b) \in \mathcal F \times \mathcal E
        \; | \;
        \eqref{c04-eq:finit-dim-var-problem-matrix}
        \textrm{ holds}
        \}
    \end{equation}
    with the boundary port variables defined by \eqref{c04-eq:discrete-boundary-efforts} and \eqref{c04-eq:discrete-boundary-flows} satisfies the isotropy condition $D \subset D^\perp$ with respect to the bilinear form $\langle\langle \cdot, \cdot \rangle\rangle_M$ that results from symmetrization of the duality product
    \begin{equation}
        \label{c04-eq:degenerate-dualtiy-product}
        \langle \cdot | \cdot \rangle_{M} := 
        \langle \vec e^p | \vec M_p \vec f^p \rangle +
        \langle \vec e^q | \vec M_q \vec f^q \rangle +
        \langle \vec e^b | \vec f^b_0 \rangle +
        \langle \hat{\vec e}^b | \hat{\vec f}^b_0 \rangle.
    \end{equation}
\end{prop}

\begin{proof}
   The proposition generalizes Proposition 18 in \cite{moulla2012pseudo} and follows the same lines. We write out the symmetrized bilinear form, replacing \eqref{c04-eq:finit-dim-var-problem-matrix}:
   \begin{multline}
        \langle \vec e^p_1 | \vec M_p \vec f^p_2 \rangle +
        \langle \vec e^q_1 | \vec M_q \vec f^q_2 \rangle +
        \langle \vec e^b_1 | (\vec f^b_0)_2 \rangle +
        \langle \hat{\vec e}^b_1 | (\hat{\vec f}^b_0)_2 \rangle
        +
        \langle \vec e^p_2 | \vec M_p \vec f^p_1 \rangle +
        \langle \vec e^q_2 | \vec M_q \vec f^q_1 \rangle +
        \langle \vec e^b_2 | (\vec f^b_0)_1 \rangle +
        \langle \hat{\vec e}^b_2 | (\hat{\vec f}^b_0)_1 \rangle =\\
        - \langle \vec e_1^p | (\vec K_p + \vec L_p) \vec e_2^q \rangle 
        - \langle \vec e_2^q | (\vec K_q + \vec L_q) \vec e_1^p \rangle 
        - \langle \vec e_1^q | (\vec K_q + \vec L_q) \vec e_2^p \rangle 
        - \langle \vec e_2^p | (\vec K_p + \vec L_p) \vec e_1^q \rangle 
        \\
        + \langle \vec T_q \vec e_1^q | \vec S_{p,0} \vec e_2^p \rangle 
        + \langle \hat{\vec T}_p \vec e_1^p | \hat{\vec S}_{q,0} \vec e_2^q \rangle 
        + \langle \vec T_q \vec e_2^q | \vec S_{p,0} \vec e_1^p \rangle 
        + \langle \hat{\vec T}_p \vec e_2^p | \hat{\vec S}_{q,0} \vec e_1^q \rangle. \hspace{3cm}
   \end{multline}
    Exploiting the matrix equalities \eqref{c04-eq:relations-K-L} and \eqref{c04-eq:relation-T-L-b}, we obtain
    \begin{equation}
        - \langle \vec e_1^p | \vec L_p \vec e_2^q \rangle         
        - \langle \vec e_1^q | \vec L_q \vec e_2^p \rangle 
        + \langle \vec e_1^q | \vec L_q \vec e_2^p \rangle 
        + \langle \vec e_1^p | \vec L_p \vec e_2^q \rangle         
        = 0,
    \end{equation}
    which proves isotropy of $D$ with respect to $\langle\langle \cdot, \cdot \rangle\rangle_M$.
\end{proof}

The discrete power continuity equation, which represents the counterpart of \eqref{c04-eq:power-continuity-weak-ports} \emph{in the approximation subspaces}, finally reads
\begin{equation}
    \label{c04-eq:power-continuity-unreduced-b2}
    \langle \vec e^p| \vec M_p \vec f^p \rangle +
    \langle \vec e^q| \vec M_q \vec f^q \rangle +
    \langle \vec e^{b}| \vec f^{b}_0 \rangle +
    \langle \hat{\vec e}^{b}| \hat{\vec f}^{b}_0 \rangle = 0.
\end{equation}

The subspace \eqref{eq:subspace-degenerate} is, however, \emph{not} a Dirac structure, as the duality product $\langle \cdot | \cdot \rangle_M$  defined in \eqref{c04-eq:degenerate-dualtiy-product} is \emph{degenerate} in general. Its value can be zero for nonzero discrete flows and/or efforts that lie in the kernel of $\vec M_p$, $\vec M_q$, or their transposes. This motivates the introduction of \emph{power-preserving mappings} on the discrete bond space in Subsection \ref{subsec:projections}.

\begin{rmk}
The problem of a degenerate duality product does not appear in the approach according to \cite{farle2013porta}, which is based on a \emph{metric-dependent} Dirac structure. The parameters in the power-preserving maps represent however \emph{degrees of freedom to tune} the resulting numerical methods.
\end{rmk}

\subsubsection{Discrete conservation laws}
\label{c04-subsubsec:explicit-flow-effort}

Assume the matrices in the second terms of \eqref{c04-eq:finit-dim-var-problem-matrix} can be factorized as
\begin{equation}
\label{c04-eq:relations-MDKL}
\begin{split}
    \vec K_p + \vec L_p &=  - (-1)^r \vec M_p \vec d_p\\
    \vec K_q + \vec L_q &=  - \vec M_q \vec d_q.
\end{split}
\end{equation}
Then the set of linear equations that relates discrete flow and effort degrees of freedom has the form
\begin{equation}
	\label{c04-eq:discrete-conservation-laws}
    \bmat{\vec f^p\\ \vec f^q } =
    \bmat{ \vec 0 &  (-1)^r \vec d_p\\ \vec d_q & \vec 0}  
    \bmat{\vec e^p\\\vec e^q}.
\end{equation}
This is a \emph{direct discrete representation} of the two conservation laws with $\vec d_p \in \IR^{N_p\times M_q}$ and $\vec d_q \in \IR^{N_q\times M_p}$ \emph{discrete derivative matrices} that replace the exterior derivative in the distributed parameter setting. For a mixed FE approximation based on Whitney forms of lowest polynomial degree, see e.\,g. \cite{bossavit1998computational}, the representation \eqref{c04-eq:discrete-conservation-laws} is obtained by integrating only over the respective discrete, oriented geometric objects (volumes, faces or edges) on the discretization mesh instead of the whole domain $\Omega$. The matrices $\vec d_p$ and $\vec d_q$ are then the transposed incidence matrices\footnote{In order to avoid confusion with the actuated system \emph{boundary}, we use, as in \cite{seslija2014explicit} or \cite{schaft2011discrete}, the term \emph{incidence} matrix instead of \emph{boundary} matrix.}, which relate the geometric objects on the mesh. For some more comments on the direct discrete representation of conservation laws, see Section \ref{sec:example}.

\subsection{Power-preserving mappings and conjugated output maps}
\label{subsec:projections}

The discrete power balance \eqref{c04-eq:power-continuity-unreduced-b2} contains the duality pairings $\langle \vec e^p| \vec M_p \vec f^p \rangle$ and $\langle \vec e^q| \vec M_q \vec f^q \rangle$, which are \emph{degenerate} in general, i.\,e. the matrices $\vec M^p$ and $\vec M^q$ may be non-quadratic and have reduced rank, see Table \ref{c04-tab:ranks-matrices} for the example considered in Section \ref{sec:example}. We motivate the definition of \emph{power-preserving mappings} on the space of discrete bond variables by the following example.

\begin{exm}
    Consider the discrete power balance, a simplified representation of  \eqref{c04-eq:power-continuity-unreduced-b2},
    $\langle \vec e | \vec M \vec f \rangle  + \langle \vec e^b |  \vec f^b_0 \rangle = 0$
    with the \emph{degenerate} bilinear form $\langle \vec e | \vec M \vec f \rangle$. Let $\vec e \in \IR^{n_e}$, $\vec f \in \IR^{n_f}$, $n_f \neq n_e$ and the matrix $\vec M$ of reduced rank $r_M < \min(n_e, n_f)$.
    Now choose $r_M$ vectors $\vec e_i$ and $\vec f_i$, $i=1,\ldots, r_M$ such that the image spaces of $\vec M$ and $\vec M^T$ are spanned by
    \begin{equation}
    \begin{split}
        \textrm{span} \{ \vec M \vec f_1, \ldots, \vec M \vec f_{r_M} \} &=:
        \textrm{span} \{ \vec w_1, \ldots, \vec w_{r_M} \} = \textrm{im}(\vec M),\\
        \textrm{span} \{ \vec M^T \vec e_1, \ldots, \vec M^T \vec e_{r_M} \} &=:
        \textrm{span} \{ \vec v_1, \ldots, \vec v_{r_M} \} = \textrm{im}(\vec M^T).
    \end{split}
    \end{equation}
    Suppose that the matrix $\vec M$ can be decomposed as
    \begin{equation}
        \label{eq:factorization-M}
        \vec M = \vec P_e^T \vec P_f \qquad \text{with} \qquad
        \vec P_e = \bmat{ \vec w_1^T \\ \vdots\\ \vec w_{r_M}^T}, \quad
        \vec P_f = \bmat{ \vec v_1^T \\ \vdots\\ \vec v_{r_M}^T},
    \end{equation}
    then the degenerate bilinear form can be replaced by the \emph{non-degenerate} duality product $\langle \tilde{\vec e} | \tilde{\vec f} \rangle$ with $\tilde {\vec e} = \vec P_e \vec e$, $\tilde {\vec f} = \vec P_f \vec f$, and the discrete power balance becomes $\langle \tilde{\vec e} | \tilde{\vec f} \rangle  + \langle \vec e^b |  \vec f^b_0 \rangle = 0$. By the definition of the rows of $\vec P_e$ and $\vec P_f$, i.\,e. $\vec w_i^T = \vec f_i^T \vec M^T$ and $\vec v_i^T = \vec e_i^T \vec M$,  it is easy to see that $\vec P_{e} \vec e = \vec 0$ for $\vec e \in \ker(\vec M^T)$ and $\vec P_{f} \vec f = \vec 0$ for $\vec f \in \ker(\vec M)$. This means that $\vec P_e$ and $\vec P_f$ describe mappings from the quotient spaces $\IR^{n_e} / \ker(\vec M^T)$ and $\IR^{n_f} / \ker(\vec M)$ to $\IR^{r_M}$, which map the equivalence classes\footnote{The maps from $\IR^{n_e}$ and $\IR^{n_f}$ to the quotient spaces are \emph{projections}.}
    \begin{equation}
        [\vec e] = \{ \vec e' \in \IR^{n_e} \,|\, \exists \vec e'' \in \ker(\vec M^T), \; \vec e' = \vec e + \vec e'' \}
        \quad
        \text{and}
        \quad 
        [\vec f] =
        \{ \vec f' \in \IR^{n_f} \,|\, \exists \vec f'' \in \ker(\vec M), \; \vec f' = \vec f + \vec f'' \}
    \end{equation}
    onto an embedding of $\IR^{n_e} \times \IR^{n_f}$, endowed with coordinates $(\tilde{\vec e}, \tilde{\vec f})$. We call $\tilde{\vec e}, \tilde{\vec f} \in \IR^{\tilde N}$ \emph{minimal} discrete power variables with $\tilde N = r_M$ in the considered case.

    If no factorization \eqref{eq:factorization-M} exists -- this is the case if the dimension of the minimal bond variables is lower than the rank of $\vec M$, $\tilde N < r_M$ --  the ``internal'' power term $\langle \vec e | \vec M \vec f \rangle$ can not be matched with $\langle \tilde{\vec e} | \tilde {\vec f} \rangle$. Preservation of the total discrete power balance will in such a case be achieved by an appropriate redefinition of the output $\vec f^b_0 \rightarrow \vec f^b$ such that $\langle \tilde{\vec e} | \tilde{\vec f} \rangle  + \langle \vec e^b |  \vec f^b \rangle = 0$ holds, see the following paragraph. For an illustration, consider Example \ref{ex:2x1-grid-new}: The original output vector $\hat{\vec f}^b_0$ does \emph{not} contain the rotational components contained in $\hat{\vec f}^b$ as depicted in Fig. \ref{fig:fb-hat}.
\end{exm}

We use the argumentation sketched above to construct a \emph{Dirac structure} on a \emph{minimal}  discrete bond space. To replace $\langle \vec e^p| \vec M_p \vec f^p \rangle$ and $\langle \vec e^q| \vec M_q \vec f^q \rangle$ in  \eqref{c04-eq:power-continuity-unreduced-b2} by \emph{non-degenerate} duality pairings, we determine \emph{power-preserving mappings}
\begin{equation}
    \label{c04-eq:effort-flow-maps} 
    \tilde{\vec e}^p = \vec P_{ep} \vec e^p, \quad
    \tilde{\vec e}^q = \vec P_{eq} \vec e^q \qquad \text{and} \qquad
    \tilde{\vec f}^p = \vec P_{fp} \vec f^p, \quad
    \tilde{\vec f}^q = \vec P_{fq} \vec f^q,
\end{equation}
such that
\begin{equation}
    \tilde{N}_p := \dim \tilde{\vec e}^p = \dim \tilde{\vec f}^p \leq \mathrm{rank} (\vec M_p) 
    \qquad \text{and} \qquad
    \tilde{N}_q := \dim \tilde{\vec e}^q = \dim \tilde{\vec f}^q \leq \mathrm{rank} (\vec M_q).
\end{equation}
We refer to the vectors $\tilde{\vec f}^p,\tilde{\vec e}^p \in \IR^{\tilde N_p}$, $\tilde{\vec f}^q,\tilde{\vec e}^q \in \IR^{\tilde N_q}$ as \emph{minimal} discrete flows and efforts, as they can be interpreted as coordinates of an embedding in the original discrete bond space.

\begin{exm}
    In the 1D case, $p=q=1$, using Whitney finite elements or the pseudo-spectral method \cite{moulla2012pseudo}, we have, $N = N_p = N_q$ and $M = M_p = M_q$ with $M = N+1$. Fixing $\tilde{\vec f}^p = \vec f^p$, $\tilde{\vec f}^q = \vec f^q$, minimal discrete efforts can be defined as $\tilde{\vec e}^p = \vec M_p^T \vec e^p$ and  $\tilde{\vec e}^q = \vec M_q^T \vec e^q$. 
\end{exm}

The following definition summarizes the core property of power-preserving mappings.

\begin{dfn}
    The discrete flow and effort mappings  \eqref{c04-eq:effort-flow-maps} are \emph{power-preserving} if they satisfy a discrete power balance
    \begin{equation}
        \label{c04-eq:power-continuity-projected}
        \langle \tilde{\vec e}^p| \tilde{\vec f}^p \rangle +
        \langle \tilde{\vec e}^q| \tilde{\vec f}^q \rangle +
        \langle \vec e^b| \vec f^b \rangle +
        \langle \hat{\vec e}^b| \hat{\vec f}^b \rangle
        = 0
    \end{equation}    
    with the given boundary inputs $\vec e^b$, $\hat{\vec e}^b$ according to \eqref{c04-eq:discrete-boundary-efforts} and possibly modified boundary outputs 
    \begin{equation}
        \label{c04-eq:discrete-boundary-flows-modified}
        \vec f^{b} = \vec S_{p}^\mu \vec e^p,
        \qquad
        \hat{\vec f}^{b} = \hat{\vec S}_{q}^\nu \vec e^q.
    \end{equation}  
\end{dfn}

\begin{rmk}
    If the mappings satisfy $\vec P_{ep}^T \vec P_{fp} = \vec M_p$ and $\vec P_{eq}^T \vec P_{fq} = \vec M_q$, the ``interior'' part of the power balance \eqref{c04-eq:power-continuity-unreduced-b2} is exactly represented by the minimal flows $\tilde{\vec f}$ and efforts $\tilde{\vec e}$, and \eqref{c04-eq:power-continuity-projected} holds with the original, collocated outputs $\vec f_0^b$, $\hat{\vec f}_0^b$. If, however, $\tilde N_q < \mathrm{rank}(\vec M_q)$ and/or $\tilde N_p < \mathrm{rank}(\vec M_p)$, a part of the power, originally described by $\langle \vec e^p| \vec M_p \vec f^p \rangle + \langle \vec e^q| \vec M_q \vec f^q \rangle$, must be ``swapped'' to the boundary terms of \eqref{c04-eq:power-continuity-projected} via the re-definition of the outputs. This way, the power-balance is maintained globally, and \emph{conservativeness} of the finite-dimensional approximation is \emph{guaranteed}.
\end{rmk}

To characterize the power-preserving mappings and modified output maps that guarantee power continuity  \eqref{c04-eq:power-continuity-projected}, we substitute in this equation the definitions of the effort and flow maps, the in- and outputs, and substitute $\vec f^p$, $\vec f^q$ according to the discrete representation \eqref{c04-eq:discrete-conservation-laws} of the conservation laws. The new power variables are now expressed in terms of the original discrete efforts,
\begin{equation}
	\label{c04-eq:image-representation}
	\begin{split}
		\underbrace{\bmat{\tilde{\vec f}^p\\ \tilde{\vec f}^q\\ \hat{\vec f}^b \\ \vec f^b }}_{\bar{\vec f}} =
		\underbrace{\bmat{\vec 0 &  (-1)^r \vec P_{fp} \vec d_p\\
			\vec P_{fq} \vec d_q & \vec 0\\
			\vec 0 & \hat{\vec S}_q\\
   			\vec S_p & \vec 0}}_{\vec E^T}
		\underbrace{\bmat{\vec e^p\\ \vec e^q}}_{\vec e}, \qquad 
		\underbrace{\bmat{\tilde {\vec e}^p\\ \tilde {\vec e}^q\\ \hat{\vec e}^b\\ \vec e^b}}_{\bar{\vec e}} =
		\underbrace{\bmat{\vec P_{ep} & \vec 0\\
			\vec 0 & \vec P_{eq}\\
			\hat{\vec T}_p & \vec 0\\
			\vec 0 &\vec T_q }}_{\vec F^T}
		\underbrace{\bmat{\vec e^p\\ \vec e^q}}_{\vec e}.
	\end{split}
\end{equation}
Equation \eqref{c04-eq:power-continuity-projected} must hold for arbitrary $\vec e^p$, $\vec e^q$, and we obtain the following matrix condition.

\begin{prop}
    The effort, flow and output maps are \emph{power-preserving}, if they satisfy the matrix equation
    \begin{equation}
    \label{c04-eq:charact-projections}
		 (-1)^r \vec d_p^T \vec P_{fp}^T \vec P_{ep}  +
		 \vec P_{eq}^T \vec P_{fq} \vec d_q +  
         \vec T_{q}^T \vec S_{p} + \hat{\vec S}_{q}^T \hat{\vec T}_{p} = \vec 0.
	\end{equation}
\end{prop}

The power-preserving maps are not unique. Different parametrizations of the matrices yield different finite-dimensional Dirac structures that approximate the original Stokes-Dirac structure of Proposition \ref{c04-prop:stokes-dirac-general-boundary-ports}. Together with a consistent approximation of the constitutive equations, we obtain PH approximate models with different numerical properties. A favorable parametrization will depend on the nature of the system (e.\,g. if the closure equations make the system hyperbolic or parabolic), the distribution and type of boundary inputs, and the application case. In any case, the power-preserving maps generate a \emph{minimal} space of power variables on which an approximate Dirac structure is defined.

In Section \ref{sec:example}, we will illustrate the construction of the power-preserving maps on the example of Whitney approximation forms on a rectangular simplicial mesh in 2D. The degrees of freedom in the mappings will allow for a trade-off between  centered schemes and upwinding in the discretized PH models.

\begin{rmk}
    Equation \eqref{c04-eq:charact-projections} relates the ``discrete differentiation matrices'' $\vec d_p$, $\vec d_q$ and the ``discrete trace matrices'' $\vec T_q$, $\hat{\vec T}_p$, paired with $\vec S_p$, $\hat{\vec S}_q$. This is an apparent reference to Stokes' theorem \eqref{c04-eq:stokes-theorem}, which is instrumental in deriving this discrete representation of power continuity (see also Eq. (43) in \cite{moulla2012pseudo}).
\end{rmk}

\subsection{Dirac structure on the minimal bond space}

The power-preserving maps that satisfy \eqref{c04-eq:charact-projections} define a \emph{Dirac structure}. We verify that \eqref{c04-eq:image-representation} is an \emph{image representation} of this Dirac structure on the minimal discrete bond space. If the effort maps are invertible, an unconstrained input-output representation exists.

\begin{prop}[Image representation]
\label{c04-prop:image-representation}
    Consider the discrete flow and effort vectors $\bar{\vec f}$ and $\bar{\vec e}$ as indicated in \eqref{c04-eq:image-representation}. $(\bar{\vec f}, \bar{\vec e})$ is an element of the 
    bond space 
    \begin{equation}
    \bar{\mathcal F} \times \bar{\mathcal E} = \IR^{\tilde N_p + \tilde N_q + M_b + \hat M_b} \times \IR^{\tilde N_p + \tilde N_q + M_b + \hat M_b}.
    \end{equation}
    Let $\tilde N_p + \tilde N_q + \hat{M}_b + M_b = M_p + M_q$ and assume that the matrix condition \eqref{c04-eq:charact-projections} is satisfied. If 
    \begin{equation}
    \label{c04-eq:rank-condition-P-T}
        \mathrm{rank} (\bmat{\vec P_{ep}\\ \hat{\vec T}_p} ) = M_p
        \qquad \textrm{and} \qquad
        \mathrm{rank} (\bmat{\vec P_{eq}\\ {\vec T}_q} ) = M_q,
    \end{equation}
    then the subspace 
    \begin{equation}
        \label{c04-eq:subspace-tilde-D}
        \bar{D} = \{ (\bar{\vec f}, \bar{\vec e}) \in \bar{\mathcal F}\times \bar{\mathcal E} \; |\; \bar{\vec f} = \vec E^T \vec{e}, \; \bar{\vec e} = \vec F^T \vec{e}, \; \vec{e} \in \IR^{M^p + M^q} \}
    \end{equation}
    is a Dirac structure.
\end{prop}

\begin{proof}
    According to the definition of the image representation of a Dirac structure (see e.\,g. \cite{schaft2000l2}, Section 4.4.1), the dimensions of $\bar{\vec f}$ and $\bar{\vec e}$ must be less\footnote{This is the case of a \emph{relaxed} image representation.} or equal $\dim(\vec e)$, which is ensured by $\tilde N_p + \tilde N_q + \hat{M}_b + M_b = M_p + M_q$. The condition  $\mathrm{rank} ( [ \vec F \;\; \vec E] ) = M_p + M_q$ is satisfied by \eqref{c04-eq:rank-condition-P-T}, from which $\mathrm{rank}( \vec F ) = M_p + M_q$ follows. Moreover, the skew-symmetry condition $\vec E \vec F^T + \vec F \vec E^T = \vec 0$ must hold. $\vec E \vec F^T + \vec F \vec E^T$ according to \eqref{c04-eq:image-representation} gives
    \begin{equation}
        \bmat{
            \vec 0 &
            \vec d_q^T \vec P_{fq}^T \vec P_{eq} + (-1)^r \vec P_{ep}^T \vec P_{fp} \vec d_p + \vec S_p^T \vec T_q + \hat{\vec S}_p^T \hat{\vec T}_q\\
            (-1)^r \vec d_p^T \vec P_{fp}^T \vec P_{ep} + \vec P_{eq}^T \vec P_{fq} \vec d_q + \vec T_q^T \vec S_p + \hat{\vec S}_q^T \hat{\vec T}_p  &
            \vec 0,
        }
    \end{equation}    
    which equals zero as the matrix equation \eqref{c04-eq:charact-projections} holds.
\end{proof}

\begin{cor}[Input-output representation]
    Under the conditions of Proposition \ref{c04-prop:image-representation}, the Dirac structure admits an \emph{unconstrained input-output representation}
    \begin{equation}
    \begin{split}
        \label{c04-eq:Dirac-structure-io}
        \bmat{-\tilde{\vec f}^p\\ -\tilde{\vec f}^q}
        &=
        \underbrace{\bmat{ \vec 0 & \vec J_p\\ \vec J_q & \vec 0 }}_{\vec J = - \vec J^T}
        \bmat{\tilde{\vec e}^p \\ \tilde{\vec e}^q}
        +
        \underbrace{\bmat{ \vec 0 & \vec B_p\\ \vec B_q & \vec 0 }}_{\vec B}
        \bmat{\hat{\vec e}^b\\ \vec e^b},\\
        \bmat{\hat{\vec f}^b\\ \vec f^b}
        &=
        \underbrace{\bmat{ \vec 0 & \vec C_q\\ \vec C_p & \vec 0 }}_{\vec C = \vec B^T}
        \bmat{\tilde{\vec e}^p \\ \tilde{\vec e}^q}
        +
        \underbrace{\bmat{ \vec 0 & \vec D_q\\ \vec D_p & \vec 0 }}_{\vec D = - \vec D^T}
        \bmat{\hat{\vec e}^b\\ \vec e^b},
        \end{split}
    \end{equation}
    with
    \begin{equation}
        \label{c04-eq:skew-symmetry-io-repr}
        \vec J_p = - \vec J_q^T, \quad \vec B_p = \vec C_p^T, \; \vec B_q = \vec C_q^T, \quad \vec D_q = - \vec D_p^T.
    \end{equation}
\end{cor}

\begin{proof}
    The (skew-)symmetry conditions can be summarized as
    \begin{equation}
        \label{c04-eq:skey-symm-proof}
        \bmat{ - \vec J_p & - \vec B_p\\ \vec C_q & \vec D_q } +
        \bmat{ - \vec J_q & - \vec B_q\\ \vec C_p & \vec D_p }^T
        =
        \bmat{ \vec 0 & \vec 0 \\ \vec 0 & \vec 0}.
    \end{equation}
    The submatrices in Eq. \eqref{c04-eq:Dirac-structure-io} are obtained from evaluation of $\bar{\vec f}$ in \eqref{c04-eq:image-representation} and exploiting   invertibility of the matrices in \eqref{c04-eq:rank-condition-P-T}. We can write 
    \begin{equation}
        \label{eq:matrices-state-space-model}
        \bmat{ - \vec J_p & - \vec B_p\\ \vec C_q & \vec D_q } =
        \bmat{ (-1)^r \vec P_{fp} \vec d_p\\ \hat{\vec S}_q }
        \bmat{\vec P_{eq}\\ {\vec T}_q}^{-1}, \quad
        \bmat{ - \vec J_q & - \vec B_q\\ \vec C_p & \vec D_p } 
        =
        \bmat{\vec P_{fq} \vec d_q\\ \vec S_p}
        \bmat{\vec P_{ep}\\ \hat{\vec T}_p} ^{-1}.
    \end{equation}
    Substituting these relations in \eqref{c04-eq:skey-symm-proof} and multiplying with the non-singular matrices $\bmat{\vec P_{ep}^T & \hat{\vec T}_p^T}$ from the left and $\bmat{\vec P_{eq}\\ \vec T_q}$ from the right yields the left hand side of \eqref{c04-eq:charact-projections}. The right hand side being zero, this proves (skew-)symmetry of the matrices \eqref{c04-eq:skew-symmetry-io-repr} of the input-output representation.
\end{proof}

The proposition is a generalization of Proposition 20 in \cite{moulla2012pseudo} for the 1D case and the pseudo-spectral method. Note that the rank condition \eqref{c04-eq:rank-condition-P-T} on the effort and flow and boundary maps is sufficient (not necessary) for the subspace \eqref{c04-eq:subspace-tilde-D} to be a Dirac structure. The fact that both matrices in \eqref{c04-eq:rank-condition-P-T} are assumed square and invertible, guarantees the input-output representation in the corollary.

\subsection{Finite-dimensional port-Hamiltonian model}

To build from the input-output representation of the Dirac structure a finite-dimensional PH model for the \emph{canonical system of two conservation laws}, we replace the minimal discrete flow variables by time derivatives of \emph{discrete states}\footnote{If a flow variable is defined differently, as in the case of the 2D SWE with the additional rotation term, this has to be accounted for also in the discrete equation.}
\begin{equation}
    \label{c04-eq:p-q-tilde-discrete}
    -\tilde{\vec f}^p =: \dot {\tilde{\vec p}} \in \IR^{\tilde N_p}, \qquad
    -\tilde{\vec f}^q =: \dot {\tilde{\vec q}} \in \IR^{\tilde N_q}.
\end{equation}
Then, the minimal efforts need to be replaced by the partial derivatives of a suitable discrete Hamiltonian $\tilde H_d(\tilde{\vec p}, \tilde{\vec q})$
\begin{equation}
    \label{c04-eq:etilde-from-Hd}
    \tilde{\vec e}^p = \left( \frac{\partial \tilde H_d}{\partial \tilde{\vec p}}\right)^T \in \IR^{\tilde N_p}, \qquad
    \tilde{\vec e}^q = \left( \frac{\partial \tilde H_d}{\partial \tilde{\vec q}}\right)^T \in \IR^{\tilde N_q}. 
\end{equation}
The discrete Hamiltonian must be defined in such a way that the discrete effort variables represent a \emph{consistent} approximation of their continuous counterparts. We present the discretization of the constitutive equations in more detail in the FE example of Section \ref{sec:example}.

With the state, input and output vectors 
\begin{equation}
    \label{eq:ssmodel-x-u-y}
    \vec x = \bmat{ \tilde{\vec p}\\ \tilde{\vec q} }, \quad
    \vec u = \bmat{\hat{\vec e}^b\\ \vec e^b}, \quad
    \vec y = \bmat{\hat{\vec f}^b\\ \vec f^b},
\end{equation}
the resulting state space model ($\vec x \in \IR^{\tilde N_p + \tilde N_q}$, $\vec u, \vec y \in \IR^{M_b + \hat M_b}$)
    \begin{equation}
        \label{c04-eq:finite-dim-ph-approx}
        \begin{split}
            \dot{\vec x} &= \vec J \nabla H_d(\vec x) + \vec B \vec u\\
            \vec y &= \vec B^T \nabla H_d(\vec x) + \vec D \vec u
        \end{split}
    \end{equation}
has explicit PH form and the discrete energy satisfies the balance equation
\begin{equation}
    \dot{H}_d = - \vec y^T \vec u,    
\end{equation}
which is the finite-dimensional counterpart of \eqref{eq:2-energy-balance-boundary}. The PH form allows to easily interconnect the finite-dimensional model of the system of two conservation laws with other subsystems in a power-preserving way, which is the basis for energy-based control design \emph{by interconnection} see e.\,g. \cite{macchelli2011energy}.

\section{Examples}
\label{sec:example}

In this section, we first illustrate the construction of power-preserving mappings and consistent Hodge matrices for the case of Whitney approximation forms and a 2D rectangular grid. The interpretation of the mappings in terms of weighted balance domains to compute the co-state variables (i.\,e. the ``internal'' discrete efforts) is illustrated and a simulation study highlights the effects of different parametrizations. In the second subsection, we consider the example of the 1D wave equation, again with Whitney approximation forms, in order to illustrate the difference of our approach to \cite{golo2004hamiltonian}. We study the numerical approximation of the eigenvalues and discuss the effect of \emph{upwinding} in the context of our approach.

\subsection{Wave equation on a 2D rectangular grid}
\label{subsec:example-2D}

To illustrate the steps towards an approximate PH state space model with desired boundary inputs by \emph{geometric discretization}, we consider a $2$-dimensional rectangular domain $\Omega = (0,L_x) \times (0,L_y) \subset \IR^2$, with boundary $\partial \Omega$, covered by a regular, \emph{oriented} simplicial triangulation $\mathcal T_h$, as sketched in Fig. \ref{c04-fig:3x2-mesh}. The system equations that relate distributed flow and effort differential forms with $f^p\in L^2\Lambda^2(\Omega)$, $q \in L^2\Lambda^1(\Omega)$, $e^p \in H^1\Lambda^0(\Omega)$, $e^q\in H^1 \Lambda^1(\Omega)$ are, according to \eqref{eq:stokes-dirac-interconnection},
\begin{equation}
    \label{eq:f-d-e-2D}
	\bmat{f^p\\ f^q}
		=
	\bmat{0 & - \mathrm{d}\\ \mathrm{d} & 0} 
    \bmat{e^p\\ e^q}.
\end{equation}
The effort forms are derived from a Hamiltonian functional,
\begin{equation}
    \label{eq:const-eq-2D}
    e_p = \delta_p H, \quad e_q = \delta_q H, \qquad
    H = \int_{\Omega} \mathcal H,
\end{equation} 
with $p \in L^2 \Lambda^2(\Omega)$ and $q \in L^2 \Lambda^1(\Omega)$ the conserved quantities and $\mathcal H$ the Hamiltonian density $n$-form. The dynamics equations are
\begin{equation}
    \label{eq:dynamic-eq-2D}
    \dot p = -f^p, \quad \dot q = - f^q.
\end{equation}
The boundary input variables (the causality of the boundary ports) will be specified in the \emph{discrete} setting by the choice of the boundary trace matrices $\vec T_q$ and $\hat{\vec T}_p$.

\begin{figure}
    \centering
    \begin{subfigure}[c]{0.3\textwidth}
        \centering
        \includegraphics[scale=0.8]{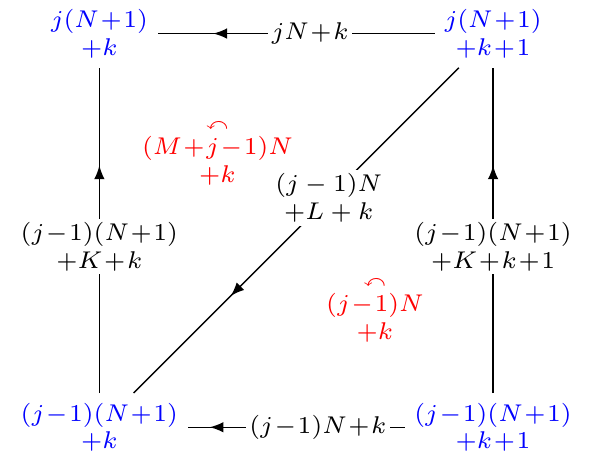}
        \caption{Unit square cell.}
    \end{subfigure}
    \begin{subfigure}[c]{0.5\textwidth}
        \centering
        \includegraphics[scale=0.8]{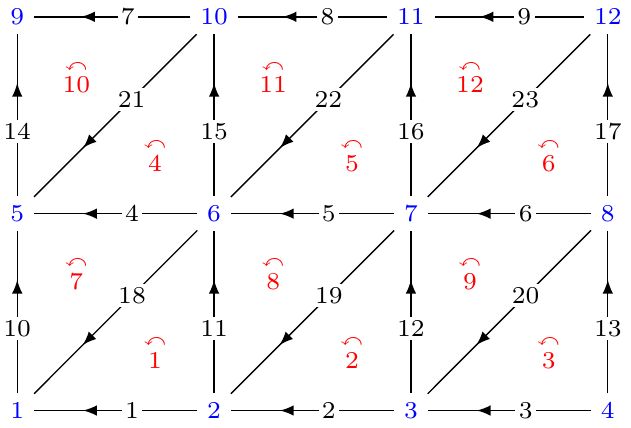}
        \caption{Example: $3\times 2$ mesh.}
    \end{subfigure}

    \caption{Numbering of nodes, edges and faces on a $N\times M$ rectangular simplicial mesh. $K=(M+1)N$, $L=K+M(N+1)$.}
    \label{c04-fig:3x2-mesh}
\end{figure}

\subsubsection{Whitney forms}
The approximation bases for flows and efforts \eqref{c04-eq:galerkin-subspaces-flows}, \eqref{c04-eq:galerkin-subspaces-efforts} are composed of \emph{Whitney forms} \cite{whitney1957geometric} of lowest polynomial degree, which can be constructed based on the \emph{barycentric} node weights \cite{bossavit2002generating}. The degrees of freedom are directly associated to the \emph{nodes}, directed \emph{edges} and \emph{faces} of the mesh.  The well-known geometric discretization of Maxwell's equations \cite{bossavit1998computational} is based on Whitney forms, and the resulting finite-dimensional models feature the (co-)incidence matrices of the underlying discretization meshes \cite{bossavit1999yee}. They can be considered a \emph{direct representation} of the physical laws on the discrete balance regions of the triangulation. In contrast to \cite{bossavit1999yee}, \cite{tonti2001direct}, where the conservation laws are evaluated on \emph{dual} or \emph{staggered} grids, we start with a \emph{single} mesh. Nevertheless, in our approach, the mappings of the original degrees of freedom allow the interpretation of the minimal discrete flows and efforts in terms of \emph{topological duality}.

\begin{exm}[Whitney forms over a 2D simplex]
	Consider the triangle $f_1 = \{ (x,y)\, |\, x,y\geq 0, \, 0 \leq x+y \leq h\}$, with vertices $n_1 = (0,0)$, $n_2 = (h,0)$, $n_3 = (0,h)$, which are connected by the directed edges $e_1$, $e_2$ and $e_3$ as shown in Fig. \ref{fig:whitney-forms}. The node, edge and face forms are constructed according to \cite{bossavit2002generating}:
	\begin{equation}
	\begin{split}
		w^{n_1} = 1 - \frac{x}{h} - \frac{y}{h}, \quad
		w^{n_2} =  \frac{x}{h}, \quad
		w^{n_3} =  \frac{y}{h},
	\end{split}
	\end{equation}
	\begin{equation}
		\begin{split}
			w^{e_1} = \frac{h-y}{h^2} dx + \frac{x}{h^2} dy, \quad
			w^{e_2} = - \frac{y}{h^2} dx + \frac{x}{h^2} dy, \quad
			w^{e_3} = - \frac{y}{h^2} dx + \frac{x-h}{h^2} dy,
		\end{split}
	\end{equation}
	\begin{equation}
		\begin{split}
			w^{f_1} = \frac{2}{h^2} dx dy.
		\end{split}
	\end{equation}
	The $0$-, $1$- and $2$-forms verify $w^{n_i}(n_j) = \delta_{ij}$, $\int_{e_i} w^{e_j} = \delta_{ij}$ ($\delta_{ij}$ the Kronecker-Delta) and $\int_{f_1} w^{f_1} = 1$.

\end{exm}

    \begin{figure}

    \begin{subfigure}[b]{0.2\textwidth}
        \includegraphics[scale=1]{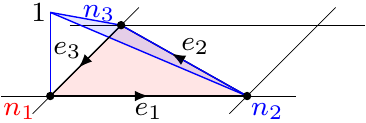}
        \caption{Node form $w^{n_1}$}
    \end{subfigure}
    \quad
    \begin{subfigure}[b]{0.5\textwidth}
        \includegraphics[scale=1]{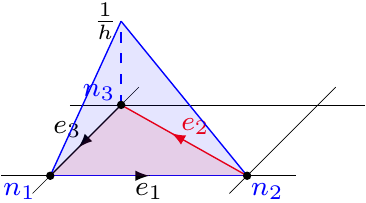}
        \includegraphics[scale=1]{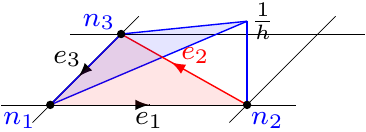}
        \caption{Components $-c_x$, $c_y$ of edge form $w^{e_2} = c_x dx + c_y dy$}
    \end{subfigure}
    \quad
    \begin{subfigure}[b]{0.2\textwidth}
        \includegraphics[scale=1]{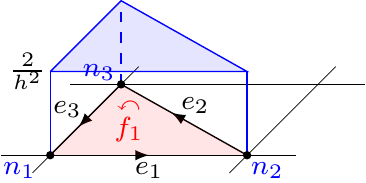}
        \caption{Face form $w^{f_1}$}
    \end{subfigure}
 
    \caption{Illustrations of the Whitney forms over a $2$-simplex.}
    \label{fig:whitney-forms}
    \end{figure}

\subsubsection{Mesh, matrices and dimensions}
Using Whitney basis forms, the degrees of freedom in the mixed Galerkin approach are associated to integrals of distributed quantities on the $k$-simplices of the mesh. The dimensions of the (initial) discrete flow and effort vectors equal the numbers of corresponding nodes, edges and faces on the grid. The same holds for the discrete efforts on the boundary, which are designated in- or outputs and are localized at the corresponding boundary nodes and edges, see Table \ref{c04-tab:dimensions-f-e}.

\begin{table}[h]
    \centering
    \caption{Dimensions of discrete flow and efforts spaces on the rectangular $N\times M$ simplicial grid.}
    {\small
    \begin{tabular}{lcccc}
        \hline
        Vector(s) & $\vec f^p$ & $\vec f^q$, $\vec e^q$ & $\vec e^p$  &  $\vec e^p_b$, $\vec e^q_b$ \\
        \hline
        Dimension & $2NM$ & $3NM\!+\!N\!+\!M$ & $(N\!+\!1)(M\!+\!1)$ & $2(N\!+\!M)$\\
        Symbol(s) & $N_p$ & $N_q=M_q$ & $M_p$  & $M_p^b=M_q^b$ \\
        \hline
    \end{tabular}
    }
    \label{c04-tab:dimensions-f-e}

\end{table}
\begin{table}[h]
    \centering
    \caption{Sizes and ranks of the matrices resulting from the mixed Galerkin approximation and the direct discrete model, respectively. $N,M >2$.}
    {\small
    \begin{tabular}{lccccccc}
        \hline
        Matrix & $\vec M_p$     & $\vec K_p + \vec L_p$    
               & $\vec d_p$     
               & $\vec M_q$     & $\vec K_q + \vec L_q$    
               & $\vec d_q$     
               & $\vec L_p=\vec L_q^T$
        \\
        \hline
        Size & $M_p\times N_p$  & $M_p \times N_q$    
             & $N_p\times N_q$ 
             & $N_q \times N_q$ & $N_q \times M_p$
             & $N_q \times M_p$
             & $M_p \times N_q$    
        \\
        Rank & $M_p-2$          & $M_p-2$          
             & $N_p$ 
             & $2(M_p-2)$       & $M_p-1$
             & $M_p-1$
             & $2(M+N)-1$
        \\
        \hline
    \end{tabular}
    }
    \label{c04-tab:ranks-matrices}
\end{table}

The mixed Galerkin approximation of the Stokes-Dirac structure yields a set of matrices with different sizes and ranks, see Table \ref{c04-tab:ranks-matrices}. The construction of  power-preserving mappings and conjugated output matrices that satisfy matrix equation \eqref{c04-eq:charact-projections}, is based on rank considerations of the involved matrix products.

\subsubsection{Power-preserving mappings, discrete in- and outputs}
We illustrate at three elementary examples the construction of the power-preserving flow and effort maps and conjugated output matrices that satisfy Eq. \eqref{c04-eq:charact-projections}. The structure of the resulting matrices can be extrapolated to the case of $N\times M$ grids with arbitrarily distributed boundary inputs of mixed causality.

\begin{exm}[Elementary $1\times1$ grid]
    Consider the sample grid in Fig. \ref{fig:sample_grid_4_nodes}. The mixed Galerkin discretization of \eqref{eq:f-d-e-2D} with Whitney forms yields the discrete representation \eqref{c04-eq:discrete-conservation-laws} with $(-1)^r=-1$ with the face degrees of freedom (flows) $\vec f^p = - \dot{\vec p} \in \IR^2$, the edge degrees of freedom (flows and efforts) $\dot{\vec q} = - \vec f^q, \vec e^q \in \IR^5$ the node degrees of freedom (efforts) $\vec e^p \in \IR^4$. The discrete derivative matrices, which satisfy the discrete complex property $\vec d_p \vec d_q = \vec 0$, are the co-incidence matrices of the graph
    \begin{equation}
        \vec d_p =
        \bmat{
        1 & 0 & 0 & 1 & 1\\
        0 & 1 & 1 & 0 & -1
        }
        , \qquad 
        \vec d_q =
        \bmat{
        -1 & 1 & 0 & 0\\
        0 & -1 & 1 & 0\\
        0 & 0 & -1 & 1\\
        1 & 0 & 0 & -1\\
        0 & -1 & 0 & 1
        }.
    \end{equation}

\paragraph{Input trace matrices and effort maps.}
We assign all effort degrees of freedom at the boundary edges the role of inputs\footnote{With this choice, we can easily derive the construction of the power-preserving mappings and output matrices. The same power-preserving mappings are valid with arbitrary boundary causality, while the output matrices can be easily adapted, as in the case of the simulation examples.}, summarized in $\vec e^b \in \IR^4$. The interior edge is related to the minimal effort $\tilde e^q \in \IR$.
\begin{equation}
    \bmat{ \vec e^b\\ \tilde{e}^q } =
    \bmat{ \vec T_q\\ \vec P_{eq}} \vec e^q
    \qquad \text{with} \qquad 
    \vec T_q = \bmat{\vec I_4 & \vec 0_{4\times 1}}, \quad \vec P_{eq} = \bmat{ \vec 0_{1\times 4} & 1}.
\end{equation}
No node plays the role of an input node, hence,
\begin{equation}
    \tilde{\vec e}^p = \vec P_{ep} \vec e^p     \qquad \text{with} \qquad  \vec P_{ep} = \vec I_4.
\end{equation}

\paragraph{Mapping of the conserved quantities on the faces.} For the mapping of the vector of integral conserved quantities\footnote{We refer to the ``original'' discrete vectors $\vec p$, $\vec q$ as \emph{discrete conserved quantities}, while we call $\tilde{\vec p}$, $\tilde{\vec q}$ the \emph{state vectors} of the resulting PH state space model.}  $\vec p \in \IR^2$ on the two faces (triangles), we argue as follows. The vector of discrete states $\tilde{\vec p} \in \IR^4$, which is dual to the vector $\tilde{\vec e}^p \in \IR^4$ of node efforts, shall
\begin{enumerate}
    \item contain weighted sums of the discrete conserved quantities on the faces that touch the corresponding node and
    \item the sum of its elements must reflect the total conserved quantity. In the example according to Fig. \ref{fig:sample_grid_4_nodes}, this means
    \begin{equation}
        \sum_{i=1}^{4} \tilde p_i = \sum_{j=1}^{2} p_j. 
    \end{equation}
\end{enumerate}
With $\tilde{\vec p} = \vec P_{fp} \vec p$, the second condition translates to\footnote{$\vec 1_n$ denotes a column vector whose $n$ elements are all 1.}  $\mathbf{1}_4^T \vec P_{fp} = \mathbf{1}_2^T$, i.\,e. the column sums of the matrix $\vec P_{fp} \in \IR^{4\times 2}$ must equal one. A matrix that satisfies this condition is
\begin{equation}
    \vec P_{fp} = \bmat{ \gamma_{I} & 0 \\ \beta_{I} & \alpha_{I\!I} \\ 0 & \gamma_{I\!I} \\ \alpha_{I} & \beta_{I\!I}} \qquad \text{with} \qquad
    \alpha_{j} + \beta_{j} + \gamma_{j} = 1, \quad j\in \{I, I\!I\}.
\end{equation}
The weights of the conserved quantities $p_1$, $p_2$ in the definition of the  states $\tilde p_i$, which are associated to the nodal efforts $\tilde e^p_i$, $i=1,\ldots,4$, are printed in Fig. \ref{fig:sample_grid_4_nodes} in red and green, respectively.

\begin{figure}
    \centering
    \includegraphics[scale=0.8]{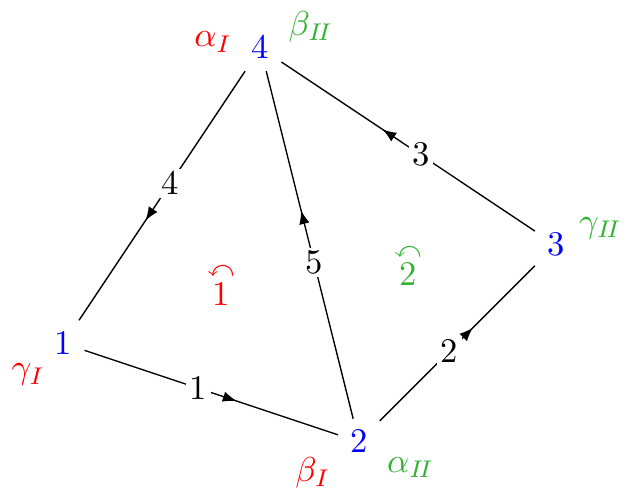}
    \caption{Sample grid  of 4 nodes to illustrate the construction of power-preserving mappings. The coefficients $\alpha_j, \beta_j, \gamma_j=1-\alpha_j-\beta_j$, $j\in\{I,I\!I\}$ weight the contribution of the integral conserved quantities $p_1$ and $p_2$ in the definition of the states $\tilde p_i$, which are associated to the edge efforts $\tilde e^p_i$, $i=1,\ldots,4$.}
    \label{fig:sample_grid_4_nodes}
\end{figure}

\paragraph{Output matrix for the nodal efforts.} The matrix equation \eqref{c04-eq:charact-projections} without a matrix $\hat{\vec T}_p$ can  be written in the form
\begin{equation}
    \label{eq:matrix-equation-without-Tp-hat}
    (-1)^r \vec d_p^T \vec P_{fp}^T \vec P_{ep} + 
    \left[
    \begin{array}{cc}
    \vec T_q^T &
    \vec P_{eq}^T
    \end{array}
    \right]
    \left[
    \begin{array}{c}
    \vec S_p\\
    \vec P_{fq} \vec d_q
    \end{array}
    \right]
    =
    \vec 0.
\end{equation}
Exploiting that $\bmat{\vec T_q^T & \vec P_{eq}^T}$ is a permutation matrix, the equation can be multiplied from the left with its transpose (which equals its inverse), and we obtain as the first line the output matrix associated to node efforts
\begin{equation}
    \label{eq:example-Sp}
    \vec S_p = - (-1)^r \vec T_q \vec d_p^T \vec P_{fp}^T \vec P_{ep} = 
    \bmat{\gamma_{I} & \beta_{I} & 0 & \alpha_{I}\\
        0 & \alpha_{I\!I} & \gamma_{I\!I} & \beta_{I\!I}\\
        0 & \alpha_{I\!I} & \gamma_{I\!I} & \beta_{I\!I}\\
        \gamma_{I} & \beta_{I} & 0 & \alpha_{I}}.
\end{equation}
The discrete output vector $\vec f^b = \vec S_p \vec e^p$ contains -- on this very simple grid -- two pairs of identical elements, which each represent convex sums of the node efforts. Regarding for example the outer boundary of face 1 in Fig. \ref{fig:sample_grid_4_nodes}, this identity is no surprise. If we delete node 1 (from the graph), and consider edges 1 and 4 as a single edge 14, the power $e_1^b f_1^b + e_4^b f_4^b$ which is transmitted over both edges must equal $(e_1^b + e_4^b) f_{14}^b$, which is the case for $f_1^b = f_4^b$.

\paragraph{Mapping of the edge states.} In analogy to \eqref{eq:example-Sp}, the matrix equation 
\begin{equation}
    \label{eq:example-Pfq-dq}
    \vec P_{fq} \vec d_q = -(-1)^r \vec P_{eq} \vec d_p^T \vec P_{fp}^T \vec P_{ep}
\end{equation}
determines the matrix $\vec P_{fq}$. The solution consists of a particular part to which a linear combination of the rows of $\vec d_p$ (recall that $\vec d_p \vec d_q = \vec 0$) can be added:
\begin{equation}
\begin{split}
\vec P_{fq} &= \vec P_{fq}^p + 
\left[ \begin{array}{cc} c_{1} & c_{2}\end{array} \right]
\vec d_p\\
&= \left[ \begin{array}{ccccc} -\gamma_I & -\gamma_{I\!I} & \alpha_I -\beta_{I\!I} & 0 & 0 \end{array} \right]
+
\left[ \begin{array}{cc} c_{1} & c_{2}\end{array} \right]
\left[ \begin{array}{ccccc} 1 & 0 & 0 & 1 &1\\ 0 & 1 & 1 & 0 & -1 \end{array} \right].
\end{split}
\end{equation}
With $c_1 = \frac{\gamma_I}{2}$ and $c_2=-\alpha_I + \beta_{I\!I} + \frac{\gamma_{I\!I}}{2}$, we get a matrix of the form
\begin{equation}
\begin{split}
\vec P_{fq} &= \vec P_{fq}^\perp + \vec P_{fq}^\parallel \\
    &=
\left[
\begin{array}{ccccc}
- \frac{\gamma_I}{2} &
- \frac{\gamma_{I\!I}}{2} &
\frac{\gamma_{I\!I}}{2} &
\frac{\gamma_I}{2} &
0
\end{array}
\right]
+\left[
\begin{array}{ccccc}
0 &
0 &
0 &
0 &
\frac{\alpha_I-\beta_I}{2} + \frac{\alpha_{I\!I}-\beta_{I\!I}}{2}
\end{array}
\right],
\end{split}
\end{equation}
where $\vec P_{fq}^\perp$ contains the weights of the conserved quantities $q_j$ on the edges ``across'' the edge on which the minimal effort $\tilde e_i^q$ is defined. Accordingly, $\vec P_{fq}^\parallel$ contains the weight associated to exactly this edge. Note that only $\vec P_{fq}^\perp$ will contribute to the definition of the discrete Hodge matrix $\vec Q_q$, which relates the efforts \emph{across} edges of the grid with the states \emph{along the dual edges}.

\end{exm}

The construction which we demonstrated for the simplest quadrilateral grid, can be extended to a rectangular grid, which is shown in the next example.

\begin{exm}[$2\times 1$ grid, unique boundary causality]
    \label{ex:2x1-grid}
    We now consider the $2\times1$ rectangular grid as depicted in Fig. \ref{fig:2x1-grid}, whose co-incidence matrices are the discrete derivative matrices 
    \begin{equation}
        \vec d_p =
        \bmat{
            -1 & 0 & 0 & 0 & 0 & 1 & 0 & 1 & 0\\
            0 & -1 & 0 & 0 & 0 & 0 & 1 & 0 & 1\\
            0 & 0 & 1 & 0 & -1 & 0 & 0 & -1 & 0\\
            0 & 0 & 0 & 1 & 0 & -1 & 0 & 0 & -1
        }, \qquad 
        \vec d_q =
        \bmat{
            1 & -1 & 0 & 0 & 0 & 0 \\
            0 & 1 & -1 & 0 & 0 & 0 \\
            0 & 0 & 0 & 1 & -1 & 0\\
            0 & 0 & 0 & 0 & 1 & -1 \\
            -1 & 0 & 0 & 1 & 0 & 0 \\
            0 & -1 & 0 & 0 & 1 & 0 \\
            0 & 0 & -1 & 0 & 0 & 1 \\
            1 & 0 & 0 & 0 & -1 & 0 \\
            0 & 1 & 0 & 0 & 0 & -1
        }.
    \end{equation}

\begin{figure}
    \centering
    \includegraphics[scale=1.2]{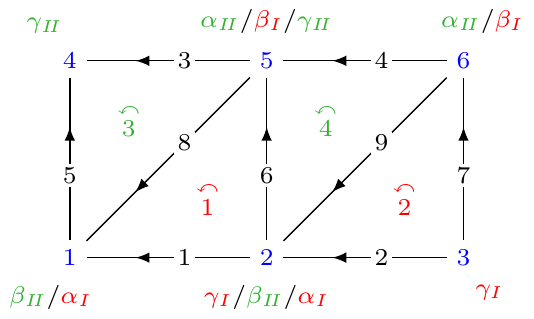}
    \caption{Illustration of the $2\times 1$ grid in Example \ref{ex:2x1-grid}. Each state $\tilde p_i$, which is associated to a ``nodal'' co-state $\tilde e^p_i$, $i=1,\ldots,6$, is defined by a weighted sum of the conserved quantities $p_j$, $j=1,\ldots,4$, on the adjacent triangles. The weights are printed next to the nodes. Red color and index $I$ refer to the lower triangles, green color and index $I\!I$ to the upper triangles. }
    \label{fig:2x1-grid}
\end{figure}

\paragraph{Input trace matrices and effort mappings.} As in the previous example, we start with a single causality on the boundary and the only input trace matrix
\begin{equation}
    \vec T_{q,1} =
    \bmat{
        1 & 0 & 0 & 0 & 0 & 0 & 0 & 0 & 0\\
        0 & 1 & 0 & 0 & 0 & 0 & 0 & 0 & 0\\
        0 & 0 & 1 & 0 & 0 & 0 & 0 & 0 & 0\\
        0 & 0 & 0 & 1 & 0 & 0 & 0 & 0 & 0\\
        0 & 0 & 0 & 0 & 1 & 0 & 0 & 0 & 0\\
        0 & 0 & 0 & 0 & 0 & 0 & 1 & 0 & 0
    },
\end{equation}
The remaining edges and all nodes are the discrete objects on which the elements of the co-state vectors $\tilde{\vec e}^q$ and $\tilde{\vec e}^p$ are defined. This fact is represented by the effort mapping matrices
\begin{equation}
    \label{eq:Peq-Pep-2x1-without-Tp-hat}
    \vec P_{eq,1} =
    \bmat{
        0 & 0 & 0 & 0 & 0 & 1 & 0 & 0 & 0 \\
        0 & 0 & 0 & 0 & 0 & 0 & 0 & 1 & 0 \\
        0 & 0 & 0 & 0 & 0 & 0 & 0 & 0 & 1
    },
    \quad
    \vec P_{ep,1} = \vec I_6.
\end{equation}
We use the index $1$ for this case with \emph{only edge inputs}, and refer to the corresponding matrices in the following example.

\paragraph{Mapping of the conserved quantities on the faces.} With the same arguments as for the simple example before, we can construct the matrix to define the discrete states $\tilde{\vec p} = \vec P_{fp,1} \vec p$, see also the illustration of the weights in Fig. \ref{fig:2x1-grid}:
\begin{equation}
    \vec P_{fp,1} =
    \bmat{
        \alpha_I & 0 & \beta_{II} & 0\\
        \gamma_I & \alpha_I & 0 & \beta_{II} \\
        0 & \gamma_I & 0 & 0\\
        0 & 0 & \gamma_{II} & 0\\
        \beta_I & 0 & \alpha_{II} & \gamma_{II}\\
        0 & \beta_I & 0 & \alpha_{II} 
    }, \qquad
    \alpha_{I/I\!I} + \beta_{I/I\!I} + \gamma_{I/I\!I} = 1.
\end{equation}

\paragraph{Output matrices for the nodal efforts.} According to \eqref{eq:example-Sp} we obtain for the nodal output matrix
\begin{equation}
    \label{eq:2x1grid-0-Sp}
    \vec S_{p,1} =
    \bmat{
        -\alpha_I & -\gamma_I & 0 & 0 & -\beta_{I} & 0\\
        0 & -\alpha_I & -\gamma_I & 0 & 0 & -\beta_{I}\\
        \beta_{I\!I} & 0 & 0 & \gamma_{I\!I} & \alpha_{I\!I} & 0 \\
        0 & \beta_{I\!I} & 0 & 0 & \gamma_{I\!I} & \alpha_{I\!I}\\
        -\beta_{I\!I} & 0 & 0 & -\gamma_{I\!I} & -\alpha_{I\!I} & 0 \\
        0 & \alpha_I & \gamma_I & 0 & 0 & \beta_{I}
    }.
\end{equation}
Note that again there are two pairs of identical outputs (modulo the sign depending on the orientation of the input edge), which is due to the fact that by merging the adjacent edges, nodes 3 and 4 could be removed from the graph.

\paragraph{Mapping of the edge states.} The solution of the matrix equation \eqref{eq:example-Pfq-dq} for the matrices as defined above (again, the rows of $\vec d_p$ can be used to adjust the solution) results in a matrix
\begin{equation}
        \vec P_{fq,1} = \vec P_{fq,1}^\perp + \vec P_{fq,1}^\parallel + \vec P_{fq,1}^{rot}
\end{equation}
with
\begin{align}
\vec P_{fq,1}^\perp &= 
    \left[
    \begin{array}{ccccccc|cc}
        \alpha_{I} & 0 & 0 & \alpha_{I\!I} & 0 & 0 & 0 & 0 & 0\\ \hline
        -\frac{\gamma_{I}}{2} & 0 & -\frac{\gamma_{I\!I}}{2} & 0 & -\frac{\gamma_{I\!I}}{2} & -\frac{\gamma_{I}}{2} & 0 & 0 & 0\\ 
        0 & -\frac{\gamma_{I}}{2} & 0 & -\frac{\gamma_{I\!I}}{2} & 0 & -\frac{\gamma_{I\!I}}{2} & -\frac{\gamma_{I}}{2} & 0 & 0 
    \end{array}
    \right],
    \\
\vec P_{fq,1}^\parallel  &= 
    \left[
    \begin{array}{ccccccc|cc}
        0 & 0 & 0 & 0 & 0 & \beta_{I}+\beta_{I\!I}-1 & 0 & 0 & 0\\ \hline
        0 & 0 & 0 & 0 & 0 & 0 & 0 & \frac{\alpha_I-\beta_I}{2} + \frac{\alpha_{I\!I}-\beta_{I\!I}}{2} & 0\\ 
        0 & 0 & 0 & 0 & 0 & 0 & 0 & 0 & \frac{\alpha_I-\beta_I}{2} + \frac{\alpha_{I\!I}-\beta_{I\!I}}{2}
    \end{array}
    \right],
    \\
\vec P_{fq,1}^{rot}  &= 
    \left[
    \begin{array}{ccccccc|cc}
    -\delta_{I} & \delta_{I\!I} & \delta_{I} & -\delta_{I\!I} & -\delta_{I} & \delta_{I}+\delta_{I\!I} & -\delta_{I\!I} & 0 & 0\\ \hline
        0 & 0 & 0 & 0 & 0 & 0 & 0 & 0 & 0\\ 
        0 & 0 & 0 & 0 & 0 & 0 & 0 & 0 & 0 
    \end{array}
    \right],
\end{align}
and the abbreviation 
\begin{equation}
    \delta_{I/II} = \frac{1}{8} + \frac{1}{4}(\alpha_{I/II}-\beta_{I/II}).
\end{equation}
\end{exm}

\begin{figure}
    \centering
    \begin{subfigure}[b]{0.24\textwidth}
        \centering
        \includegraphics[scale=0.75]{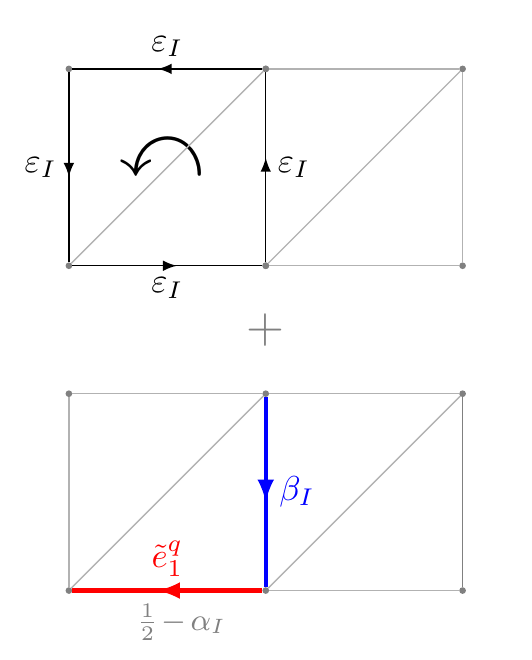}     

        \caption{$\tilde q_1$}
        \label{fig:q-tilde-1}
    \end{subfigure}
    \begin{subfigure}[b]{0.24\textwidth}
        \centering
        \includegraphics[scale=0.75]{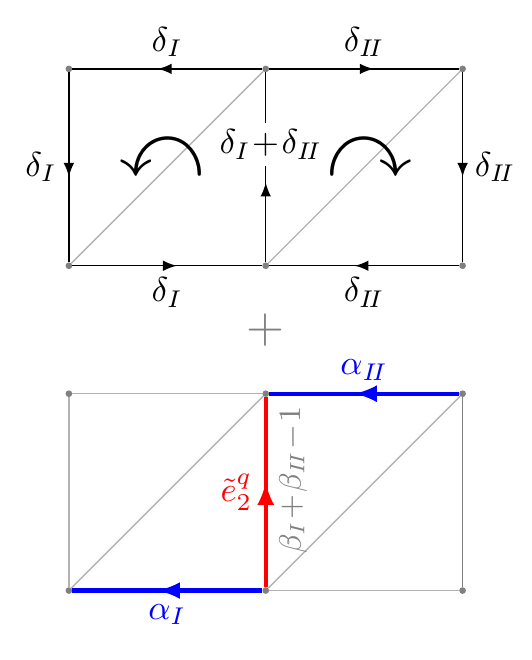} 
    
        \caption{$\tilde q_2$}
        \label{fig:q-tilde-2}
    \end{subfigure}
    \begin{subfigure}[b]{0.24\textwidth}
        \centering
        \includegraphics[scale=0.75]{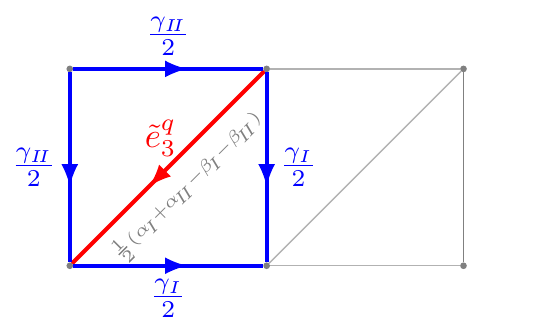} 
    
        \caption{$\tilde q_3$}
        \label{fig:q-tilde-3}
    \end{subfigure}
    \begin{subfigure}[b]{0.24\textwidth}
        \centering
        \includegraphics[scale=0.75]{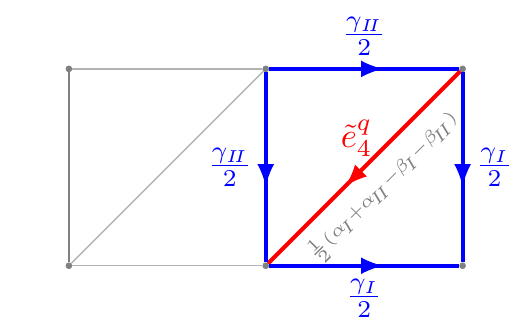} 
    
        \caption{$\tilde q_4$}
        \label{fig:q-tilde-4}
    \end{subfigure}
    \caption{Illustration of the components of $\tilde{\vec q}$ in Example \ref{ex:2x1-grid-new} in terms of the original conserved quantities $q_j$, $j=1,\ldots,9$, on the edges of the grid. The components across the considered effort edge/aligned with the effort edge are drawn in blue/red. The round black arrows indicate the sense of the rotational components in $\tilde q_1$ and $\tilde q_2$ for positive values of $\varepsilon_{I}$ and $\delta_{I/I\!I}$, respectively.}
    \label{fig:q-tilde}
\end{figure}

\begin{exm}[$2\times 1$ grid, mixed boundary causality]
    \label{ex:2x1-grid-new}
    Still considering the grid in Fig. \ref{fig:2x1-grid}, we assign the efforts in nodes $1$ and $2$ the role of (boundary) inputs $\hat e^b_1$ and $\hat e^b_2$ and remove the effort on edge  $1$ from the input vector $\vec e^b$. The corresponding input trace matrices are
    \begin{equation}
        \hat{\vec T}_p =
        \bmat{ 1 & 0 & 0 & 0 & 0 & 0\\ 0 & 1 & 0 & 0 & 0 & 0 }, \qquad 
        \vec T_q = 
        \bmat{  0 & 1 & 0 & 0 & 0 & 0 & 0 & 0 & 0\\ 0 & 0 & 1 & 0 & 0 & 0 & 0 & 0 & 0\\ 0 & 0 & 0 & 1 & 0 & 0 & 0 & 0 & 0\\ 0 & 0 & 0 & 0 & 1 & 0 & 0 & 0 & 0\\ 0 & 0 & 0 & 0 & 0 & 0 & 1 & 0 & 0 }.
    \end{equation}
    The matrix equation \eqref{c04-eq:charact-projections}  for power preservation can now be written as
    \begin{equation}
        \label{eq:matrix-equation-with-Tp-hat}
        \bmat{(-1)^r\vec d_p^T \vec P_{fp}^T & \hat{\vec S}_q^T }
        \bmat{\vec P_{ep}\\ \hat{\vec T}_p}
        +
        \bmat{\vec P_{eq}^T & \vec T_q^T}
        \bmat{\vec P_{fq} \vec d_q\\ \vec S_p} = \vec 0.
    \end{equation}
    For the moment, we assume that by appropriate choice of $\hat{\vec S}_q$, the first term can be made $(-1)^r \vec d_p^T \vec P_{fp,1}^T \vec P_{ep,1}$. We obtain the flow map $\vec P_{fq}$ and the output matrix $\vec S_p$ in the second term  (with $\vec P_{ep,1}= \vec I$) by the solution of 
    \begin{equation}
        \vec P_{fq} \vec d_q = -(-1)^r \vec P_{eq} \vec d_p^T \vec P_{fp,1}^T, \qquad
        \vec S_{p} = -(-1)^r \vec T_{q} \vec d_p^T \vec P_{fp,1}^T.
    \end{equation}
    The output matrix $\vec S_p$ contains the rows of $\vec S_{p,1}$ that correspond to the input edges represented by the rows of $\vec T_q$. In the present case, we have to delete the first row in \eqref{eq:2x1grid-0-Sp} and obtain
\begin{equation}
    \label{eq:Sp-2x1-grid-new}
    \vec S_{p} =
    \bmat{
        0 & -\alpha_I & -\gamma_I & 0 & 0 & -\beta_{I}\\
        \beta_{I\!I} & 0 & 0 & \gamma_{I\!I} & \alpha_{I\!I} & 0 \\
        0 & \beta_{I\!I} & 0 & 0 & \gamma_{I\!I} & \alpha_{I\!I}\\
        -\beta_{I\!I} & 0 & 0 & -\gamma_{I\!I} & -\alpha_{I\!I} & 0 \\
        0 & \alpha_I & \gamma_I & 0 & 0 & \beta_{I}
    }.
\end{equation}
The construction of $\vec P_{fq}$ follows the same lines as in the previous examples. The horizontal edge, on which a discrete co-state is defined, gives rise to a new element of the discrete state vector $\tilde{\vec q} \in \IR^4$, which is illustrated in Fig. \ref{fig:q-tilde}. The matrix $\vec P_{fq}$ becomes
\begin{equation}
        \vec P_{fq} = \vec P_{fq}^\perp + \vec P_{fq}^\parallel + \vec P_{fq}^{rot}
\end{equation}
with
\begin{align}
\vec P_{fq}^\perp &= 
    \left[
    \begin{array}{ccccccc|cc}
        0 & 0 & 0 & 0 & 0 & - \beta_{I} & 0 & 0 & 0\\
        \alpha_{I} & 0 & 0 & \alpha_{I\!I} & 0 & 0 & 0 & 0 & 0\\ \hline
        -\frac{\gamma_{I}}{2} & 0 & -\frac{\gamma_{I\!I}}{2} & 0 & -\frac{\gamma_{I\!I}}{2} & -\frac{\gamma_{I}}{2} & 0 & 0 & 0\\ 
        0 & -\frac{\gamma_{I}}{2} & 0 & -\frac{\gamma_{I\!I}}{2} & 0 & -\frac{\gamma_{I\!I}}{2} & -\frac{\gamma_{I}}{2} & 0 & 0 
    \end{array}
    \right],
    \\
\vec P_{fq}^\parallel  &= 
    \left[
    \begin{array}{ccccccc|cc}
        \frac{1}{2}-\alpha_{I} & 0 & 0 & 0 & 0 & 0 & 0 & 0 & 0 \\
        0 & 0 & 0 & 0 & 0 & \beta_{I}+\beta_{I\!I}-1 & 0 & 0 & 0\\ \hline
        0 & 0 & 0 & 0 & 0 & 0 & 0 & \frac{\alpha_I-\beta_I}{2} + \frac{\alpha_{I\!I}-\beta_{I\!I}}{2} & 0\\ 
        0 & 0 & 0 & 0 & 0 & 0 & 0 & 0 & \frac{\alpha_I-\beta_I}{2} + \frac{\alpha_{I\!I}-\beta_{I\!I}}{2}
    \end{array}
    \right],
    \\
\vec P_{fq}^{rot}  &= 
    \left[
    \begin{array}{ccccccc|cc}
    -\varepsilon_{I} & 0 & \varepsilon_{I} & 0 & -\varepsilon_{I} & \varepsilon_{I} & 0 & 0 & 0 \\
    -\delta_{I} & \delta_{I\!I} & \delta_{I} & -\delta_{I\!I} & -\delta_{I} & \delta_{I}+\delta_{I\!I} & -\delta_{I\!I} & 0 & 0\\ \hline
        0 & 0 & 0 & 0 & 0 & 0 & 0 & 0 & 0\\ 
        0 & 0 & 0 & 0 & 0 & 0 & 0 & 0 & 0 
    \end{array}
    \right],
\end{align}
and the abbreviation 
\begin{equation}
    \varepsilon_{I/II} = \frac{1}{8} - \frac{1}{4}(\alpha_{I/II}-\beta_{I/II}).
\end{equation}
Figure \ref{fig:q-tilde} illustrates the different components whose (vector) sums constitute the states $\tilde q_i$, $i=1,\ldots,4$ in the example. With
\begin{equation}
    \vec P_{fp} = \vec P_{ep} \vec P_{fp,1} =
    \bmat{
        0 & \gamma_I & 0 & 0\\
        0 & 0 & \gamma_{II} & 0\\
        \beta_I & 0 & \alpha_{II} & \gamma_{II}\\
        0 & \beta_I & 0 & \alpha_{II} 
    }
\end{equation}
and 
\begin{equation}
\begin{split}
    \hat{\vec S}_q &= (-1)^r \hat{\vec T}_p 
    \bmat{\vec d_q^T \vec  P_{fq}^T & \vec S_p^T} 
    \bmat{\vec P_{eq} \\ \vec T_q}\\
    &=
    \bmat{
    \alpha_{I}-\frac{1}{2} & 0 & -\beta_{I\!I} & 0 & \beta_{I\!I} & -\alpha_{I} & 0 & \beta_{I\!I}-\alpha_{I} & 0\\ \gamma_{I}-\frac{1}{2} & \alpha_{I} & 0 & -\beta_{I\!I} & 0 & \beta_{I\!I}-\gamma_{I} & -\alpha_{I} & -\gamma_{I} & \beta_{I\!I}-\alpha_{I}
    },
\end{split}
\end{equation}
see Fig. \ref{fig:fb-hat}, the parametrization of power-preserving effort and flow maps and output matrices is completed.
\end{exm}

\begin{figure}
    \centering
    \begin{subfigure}[b]{0.3\textwidth}
        \centering
        \includegraphics[scale=0.8]{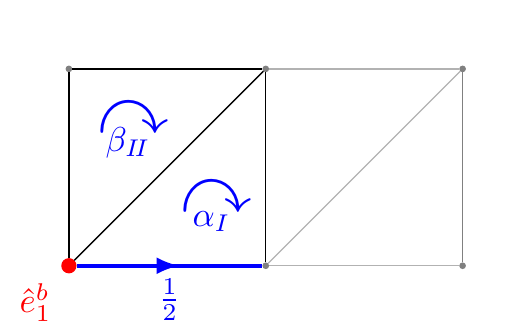} 
    
        \caption{$\hat f^b_1$}
        \label{fig:fb-hat-1}
    \end{subfigure}
    \begin{subfigure}[b]{0.3\textwidth}
        \centering
        \includegraphics[scale=0.8]{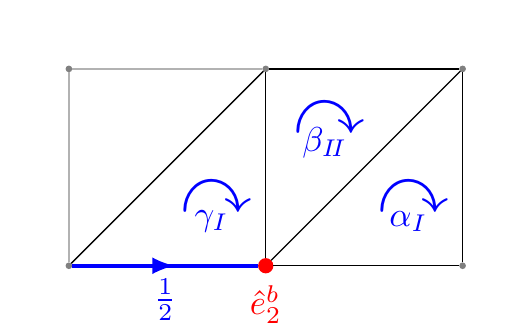} 
    
        \caption{$\hat f^b_2$}
        \label{fig:fb-hat-2}
    \end{subfigure}
    
    \caption{Illustration of the geometric objects  on which the elements of the output vector $\hat{\vec f}^b$ are defined in Example \ref{ex:2x1-grid-new} (blue, with weights). The dual elements of the input vector $\hat{\vec e}^b$ are defined on the red nodes.}
    \label{fig:fb-hat}
\end{figure}

\subsubsection{Generalization to $N\times M$ meshes and remarks}

\paragraph{$N\times M$ meshes.} The construction as presented on the three elementary examples above can be generalized in a straightforward manner to arbitrary $N\times M$ rectangular meshes. The direct interpretation of the discretized system equations as discrete conservation laws in the case of Whitney approximation forms allows for a construction of the matrices based on the properties of the 2-complex (generalized oriented graph) on the discretization mesh. In the above examples, we used only two sets of convex weights $(\alpha_j, \beta_j, \gamma_j)$, $j \in \{I, I\!I\}$ for the upper and lower triangles. It is, however, possible to assign different combinations of convex weights to each triangle, for example on  non-rectangular meshes over more complex geometries.

\paragraph{Input trace matrices and effort maps.}
Identifying the elements of the  input vector $\vec u = \bmat{ (\vec u^p)^T& (\vec u^q)^T}^T = \bmat{ (\hat{\vec e}^b)^T & (\vec e^b)^T}^T$ with effort degrees of freedom on the boundary nodes and edges corresponds to a consistent imposition of the effort boundary conditions in the finite-dimensional model. To arrive at the input-output representation \eqref{c04-eq:Dirac-structure-io}, the matrices
\begin{equation}
    \label{eq:permutation-matrices}
    \vec \Pi_p := \bmat{\vec P_{ep}\\ \hat{\vec T}_p} \quad
    \text{and} \quad 
    \vec \Pi_q := \bmat{\vec P_{eq}\\ {\vec T}_q}
\end{equation}
should be square and invertible. With the presented choice,  $\vec \Pi_p$ and $\vec \Pi_q$ become \emph{permutation} \emph{matrices} and the property $\vec \Pi_{p/q}^{-1} = \vec \Pi_{p/q}^{T}$  makes the matrices of the state space model as indicated in  \eqref{eq:matrices-state-space-model} particularly simple.

\paragraph{Flow/state maps.} By the presented construction,  each element $\tilde p_i$, $i=1,\ldots, \tilde N_p$, of $\tilde{\vec p} = \vec P_{fp} {\vec p}$ is related to a $2$-chain (a weighted formal sum of $2$-simplices), located around the node associated to $\tilde e_i^p$. The node and the weighted $2$-chain can be considered as \emph{topologically dual} objects. The property  $\alpha_\nu + \beta_\nu + \gamma_\nu = 1$,  $\nu \in \{I, I\!I\}$ ensures that the balance of the discrete conserved quantities
\begin{equation}
    \sum_{i=1}^{\tilde N_p} \tilde p_i 
    =
    \sum_{j=1}^{N_p} p_j - \epsilon_p
\end{equation}
holds. If (boundary) input nodes are defined, the error $\epsilon_p \neq 0$ occurs, because the weighted contribution of $p_j$ on $2$-simplices next to the input nodes is neglected in  the definition of discrete states. It is easy to imagine that the error $\epsilon^p$, which tends to zero with grid refinement, can be related to well-known effects from the discretization with staggered grids, like ghost values, see e.\,g. \cite{kotyczka2017discrete} for a discussion from the PH point of view.

A related interpretation of the (minimal) states in terms of topologically dual objects  holds for the different elements of the  vector $\tilde{\vec q} = \vec P_{fq} \vec q$. As shown in Fig. \ref{fig:q-tilde}, each element $\tilde q_i$ of $\tilde{\vec q}$ can be considered dual to a discrete effort $\tilde e_i^q$ on a horizontal, vertical or diagonal edge (drawn in red). $\tilde q_i$ is localized on a formal sum of the adjacent $1$-simplices (edges), which can be decomposed  into components across and along the effort edge and a rotational part. Only the ``across'' part contributes to the discrete constitutive equations as discussed in the next section. While the effort edges are considered \emph{outer oriented} (``across''), the formal sums of edges, on which the $\tilde q_i$ are defined, are \emph{inner oriented} (``along''), which describes the geometric nature of the different system variables.

\begin{rmk}
     The reconstruction of the rotational components of $\tilde{\vec q}$ from the given quantities can be used to discretize the vorticity term in the shallow water equation \eqref{c04-eq:swe-ode-1}.
\end{rmk}

\paragraph{Power-conjugated discrete outputs.}
Like the minimal flows and efforts, the discrete power-conjugated outputs $\vec f^b = \vec S_p \vec e^p$ and $\hat{\vec f}^b = \hat{\vec S}_q \vec e^q$ are constructed as weighted sums of the discrete efforts in the vicinity of the corresponding boundary input. The components $f_i^b$ are defined by a convex sum of node efforts, see e.\,g. \eqref{eq:Sp-2x1-grid-new}. The $\hat f_i^b$ are composed of rotational parts and a component associated to the neighbouring, outer oriented boundary edge, as illustrated in Fig. \ref{fig:fb-hat}.

If the effort maps and input trace matrices form permutation matrices \eqref{eq:permutation-matrices}, the feedthrough matrices in the PH state space model according to \eqref{c04-eq:Dirac-structure-io} become $\vec D_q = \hat{\vec S}_q \vec T_q^T$ and $\vec D_p = \vec S_p \hat{\vec T}_p^T$. By the collocated  construction of $\vec f^b$ and $\hat{\vec f}^b$, these matrices have only non-zero elements at the interfaces between two boundary regions $\Gamma_i$ and $\hat \Gamma_j$ with different causality. This feedthrough is \emph{physical} as it only stems from the definition of neighboring in- and outputs, and can be completely avoided by setting the boundary inputs zero at these interfaces. For 1D systems, where the two parts of the boundary are not connected, \emph{no feedthrough term} occurs at all. The absence of an undesired direct feedthrough (undesired at least for the numerical approximation of hyperbolic systems) distinguishes our method from the structure-preserving discretization according to \cite{golo2004hamiltonian}, where the feedthrough stems from the convex sum of nodal efforts to define the discrete co-state variables.

\subsubsection{Discrete constitutive equations}
\label{subsubsec:discrete-constitutive-eq}

To obtain a consistent numerical approximation of the system of conservation laws, the discrete states $\tilde{\vec p}$, $\tilde{\vec q}$ and the efforts or co-states $\tilde{\vec e}^p$, $\tilde{\vec e}^q$ must be related via discrete constitutive relations that are consistent with the continuous ones. We consider the case of linear constitutive equations with Hamiltonian density $\mathcal H = \frac{1}{2} p \wedge \ast p + \frac{1}{2} q \wedge \ast q$  and $e^p = \ast p$, $e^q = - \ast q$. The discrete constitutive equations will be expressed by
\begin{equation}
    \tilde {\vec e}^p = \vec Q_p \tilde {\vec p}, \qquad 
    \tilde {\vec e}^q = \vec Q_q \tilde {\vec q}
\end{equation}
with positive definite, diagonal matrices $\vec Q_p$, $\vec Q_q$ that represent \emph{diagonal discrete Hodge operators} \cite{specogna2015diagonal}. The discrete states $\tilde{\vec p}$ and $\tilde{\vec q}$ are constructed (as $\tilde{\vec f}^p$ and $\tilde{\vec f}^q$) as linear combinations of integral conserved quantities on the $2$- and $1$-simplices of the discretization grid. The faces, based on which $\tilde p_i$ is constructed, surround the node to which $\tilde e_i^p$ is associated. A similar \emph{geometric} duality\footnote{This geometric duality is immediately given if the two conservation laws are modeled on \emph{two} shifted grids, i.\,e. dual meshes \cite{seslija2014explicit}.} can be observed for the $\tilde e_i^q$-edges and the neighbouring edges that constitute $\tilde q_i$. One can even imagine $\tilde{\vec p}$, $\tilde{\vec q}$ localized on a (virtual) \emph{dual} grid, whose localization and shift with respect to the original (primal) grid are \emph{parameterized} by the convex set of mapping parameters  $(\alpha_j, \beta_j, \gamma_j)$, which we assume all to be positive and related via $\alpha_j + \beta_j + \gamma_j = 1$. Moreover, we consider a mesh with equal step size $h_x = h_y = h$ in both coordinate directions.

For the consistent discretization of the time-invariant constitutive equations, we consider a steady state. In this case, the elements of $\tilde{\vec e}^p$ must represent ``average'' values of $p$ on the \emph{weighted} sum of balance areas\footnote{Precisely, the average value of the coefficient function of the $2$-form $p$.}  on which the states $\tilde p_i$  are defined. The diagonal matrix $\vec Q_p$ with elements
\begin{equation}
    [\vec Q_p]_{i,i} = \frac{2}{h^2 \sum_{j=1}^{N_p} [\vec P_{fp}]_{i,j}  }, \qquad i=1,\ldots, \tilde N_p
\end{equation}
represents a consistent Hodge matrix.

Accordingly, the elements of $\tilde{\vec e}^q$ must reflect the integral flux of the vector field\footnote{Index raising of the $1$-form $q$.} $q^\sharp$ across the corresponding horizontal, vertical or diagonal edges. Only the parts of $\tilde q_i$, which are associated to the edges \emph{perpendicular} to the $\tilde e^q_i$-edge, contribute to this flux. This  reasoning yields a \emph{diagonal} matrix $\vec  Q_q$ that replaces the Hodge star in \eqref{eq:const-eq-2D} with diagonal elements\footnote{Note that our grids according to Fig. \ref{c04-fig:3x2-mesh} have square cells and  unique orientations of horizontal, vertical and diagonal edges.}
\begin{equation}
        [\vec Q_q]^{hor/ver}_{i,i} = \frac{1}{\sum_{j=1}^{N_q} \left| [\vec P_{fq}^\perp]_{i,j} \right|} \qquad 
        \text{and}
        \qquad 
        [\vec Q_q]^{dia}_{i,i} = \frac{2}{\sum_{j=1}^{N_q} \left| [\vec P_{fq}^\perp]_{i,j} \right|}
        \quad
\end{equation}
for the efforts on horizontal/vertical and diagonal edges, respectively.

\subsubsection{Simulation study}
We consider the linear wave equation in port-Hamiltonian form \eqref{eq:f-d-e-2D}--\eqref{eq:dynamic-eq-2D}  on a square domain $\Omega = (0,20) \times (0,20)$ to illustrate the effects of different mapping parameters. We impose the boundary conditions
\begin{equation}
    \begin{split}
        e^p(0,0) = u(t) =
        \begin{cases} 
            \sin^2(\frac{\pi}{8} t), & 0 \leq t < 8\\
            0                          & t \geq 8
        \end{cases}, \qquad
        e^q(x,y) = 0 \quad \text{on} \quad  \partial \Omega.
    \end{split}
\end{equation}
by the input trace matrices
\begin{equation}
    \hat{\vec T}_p = \bmat{1 & 0 & \ldots & 0},\qquad \vec T_q = \mathbb{I}_b^1,
\end{equation}
where $\mathbb{I}_b^1 \in \mathbb R^{M_b \times N_q}$ is the matrix composed of unit row vectors associated to boundary edges. The inputs to the simulation model according to Eq. \eqref{eq:ssmodel-x-u-y} are
\begin{equation}
    \hat e^b(t) = u(t), \quad \vec e^b(t) = \vec 0.
\end{equation}

\begin{table}[h]
\caption{Parameter sets used in the simulations.}
\label{tab:projection-paramters}
\centering
{\small
\begin{tabular}{c|ccc|cc|ccc|cc}
    & $\alpha_I$ & $\beta_{I}$ & $\gamma_{I}$ & $\delta_{I}$ & $\varepsilon_{I}$ & $\alpha_{I\!I}$ & $\beta_{I\!I}$ & $\gamma_{I\!I}$ & $\delta_{I\!I}$ & $\varepsilon_{I\!I}$ \\
    \hline
    \#1 & $1/3$ & $1/3$ & $1/3$ & $1/8$ & $1/8$ & $1/3$ & $1/3$ & $1/3$ & $1/8$ & $1/8$ \\ 
    \#2 & $1/2$ & $1/4$ & $1/4$ & $3/16$ & $1/16$ & $1/4$ & $1/2$ & $1/4$ & $1/16$ & $3/16$  \\ 
    \#3 & $2/3$ & $1/12$ & $1/4$ & $13/48$ & $-1/48$ & $1/12$ & $2/3$ & $1/4$ & $-1/48$ & $13/48$ \\ 
    \#4 & $15/16$ & $1/32$ & $1/32$ & $45/128$ & $-13/128$ & $1/32$ & $15/16$ & $1/32$ & $-13/128$ & $45/128$  \\
    \hline
\end{tabular}
}
\end{table}

\begin{figure}[h]
    \includegraphics[width=0.45\textwidth]{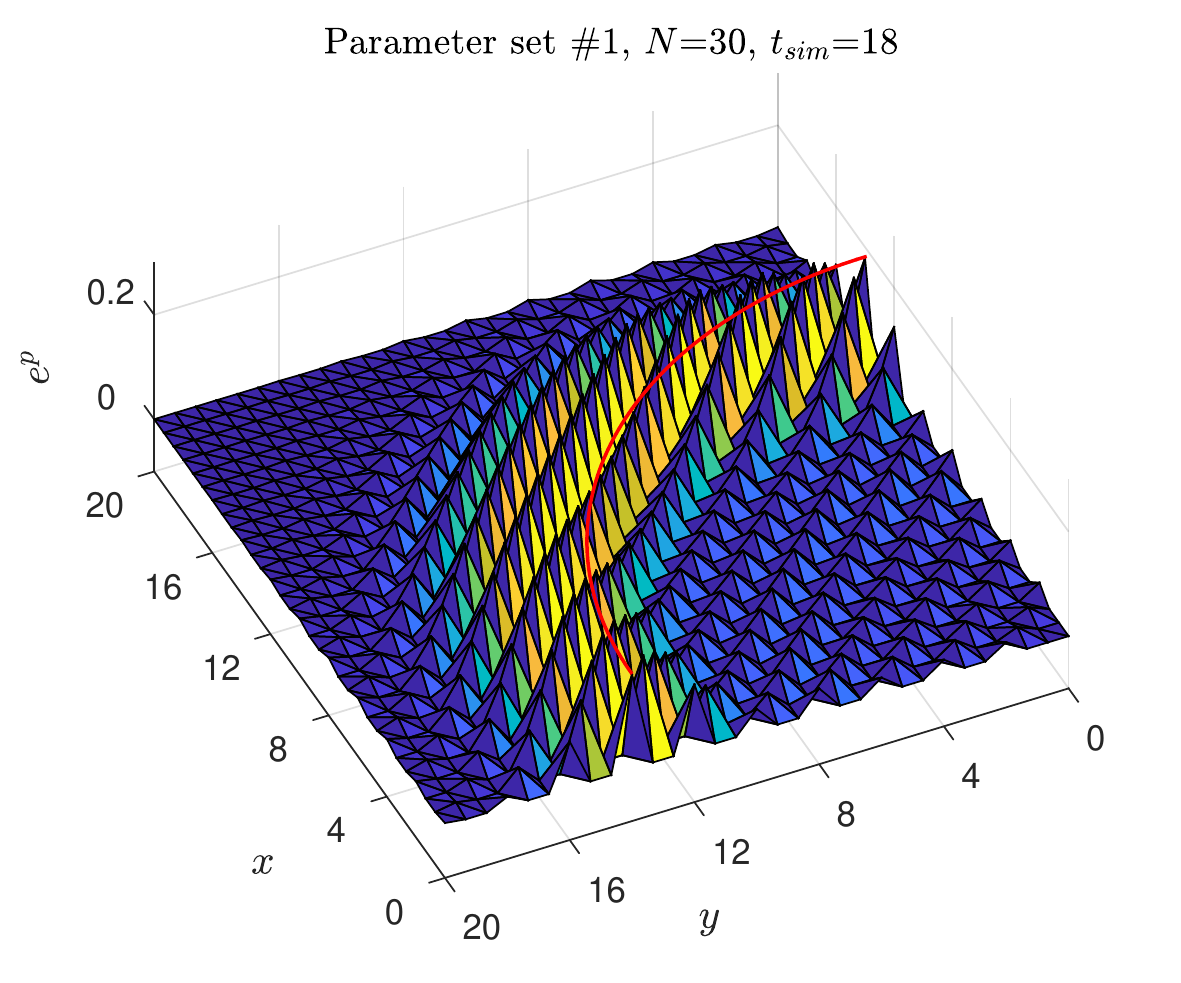}
    \includegraphics[width=0.45\textwidth]{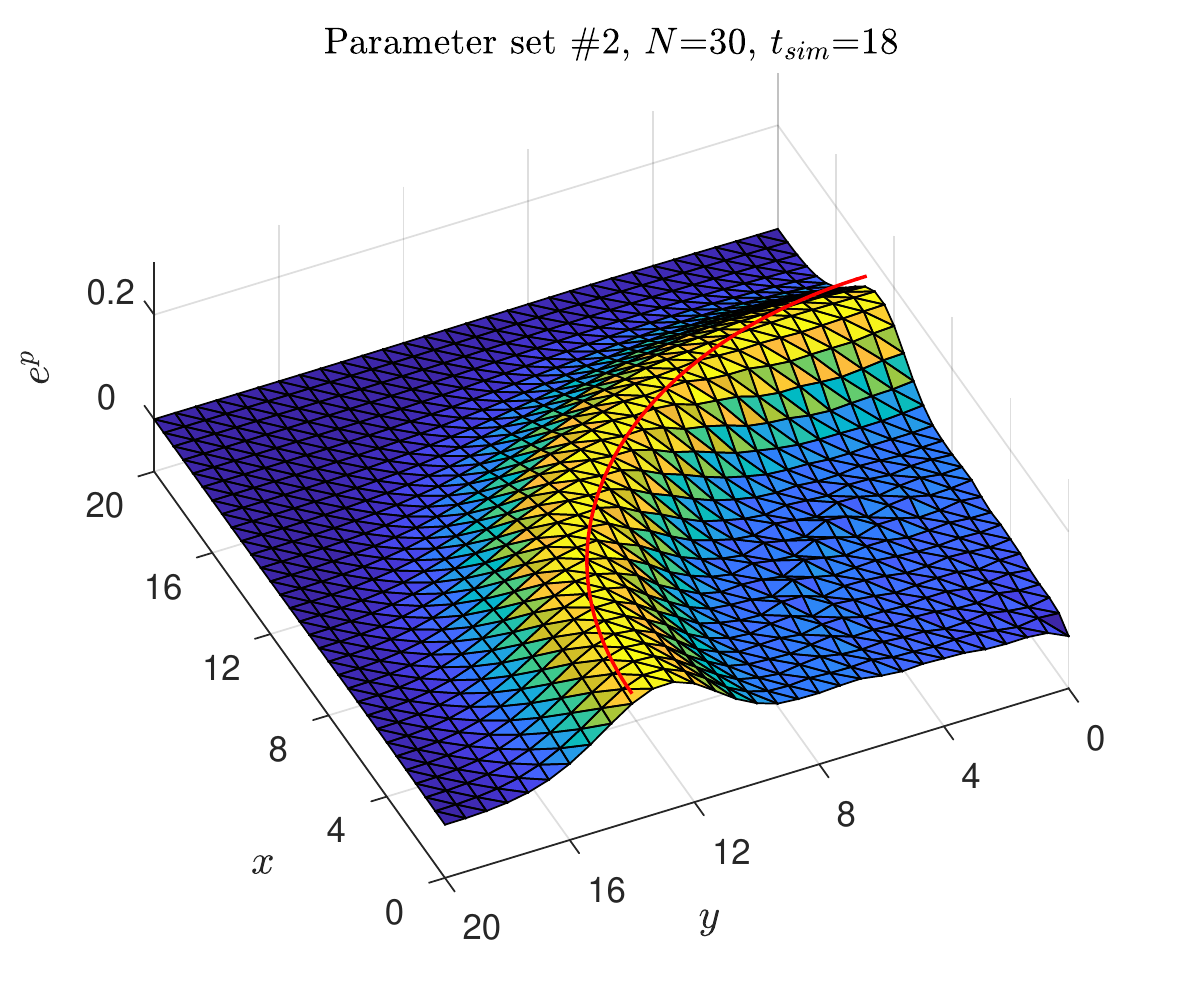}

    \includegraphics[width=0.45\textwidth]{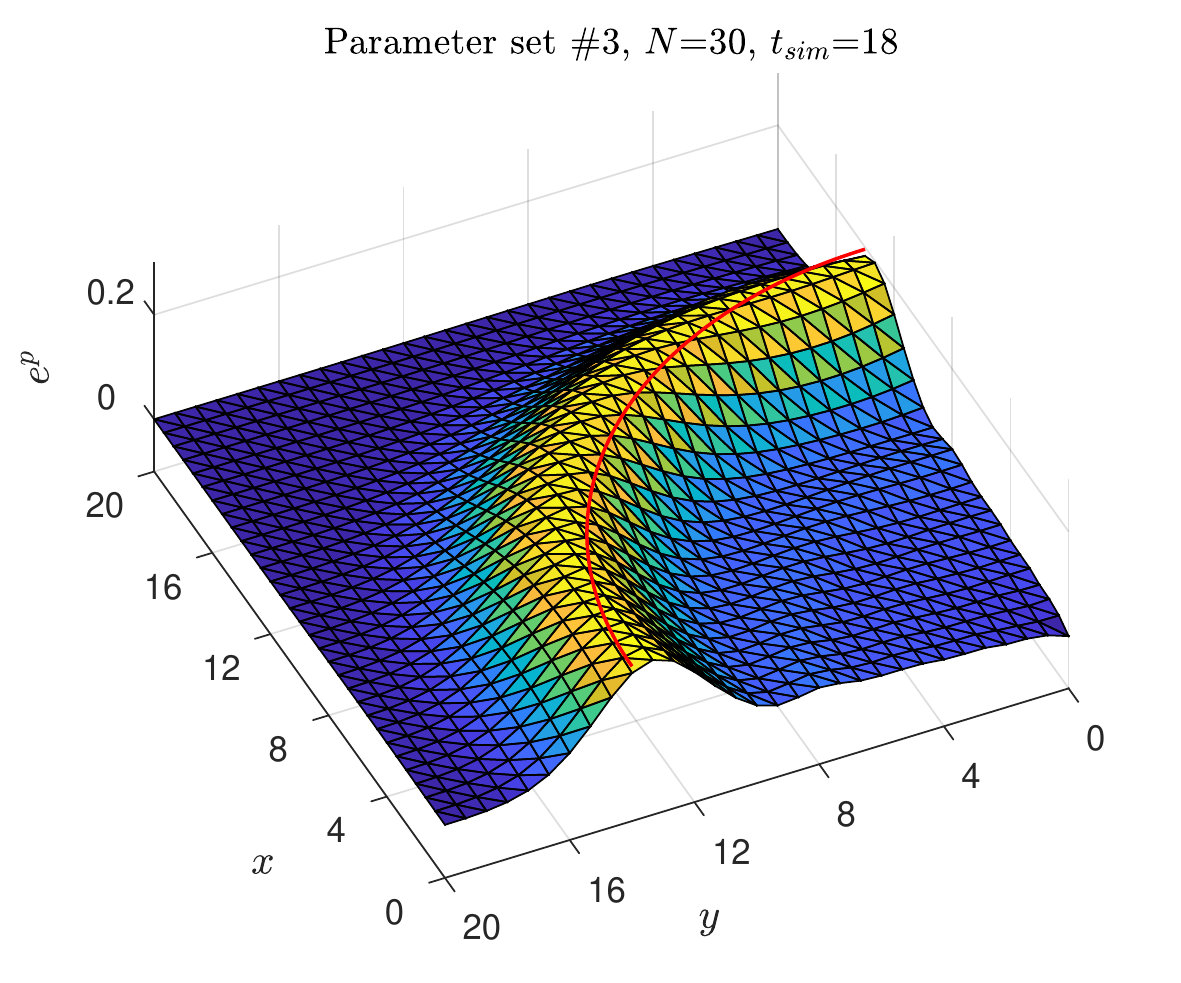}
    \includegraphics[width=0.45\textwidth]{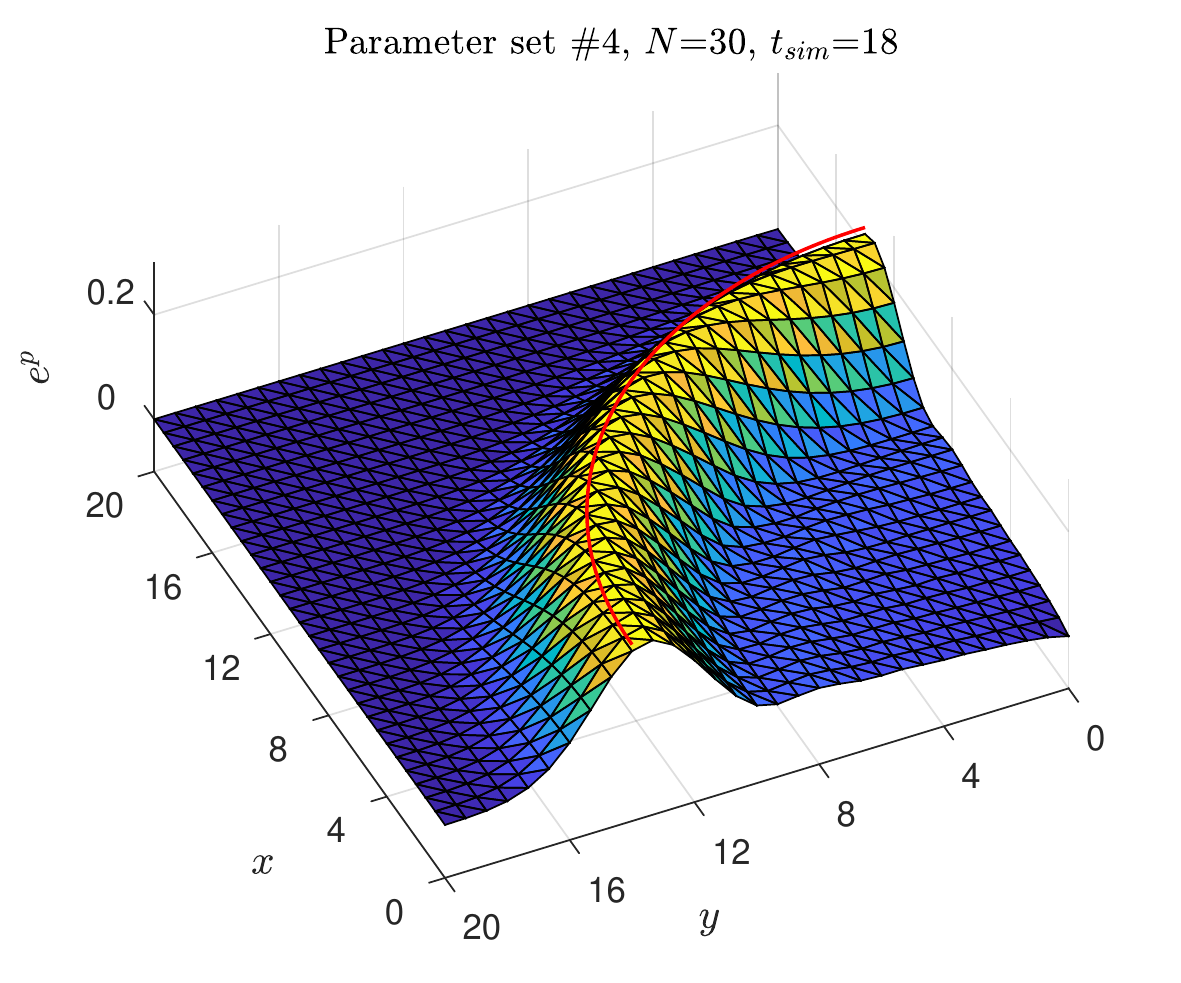}

    \caption{Propagation of a wave due to point-wise boundary excitation under different parametrizations of the method. Snapshots at $t_{sim} = 18$.}
    \label{sim-pressure-wave}
\end{figure}

Fig. \ref{sim-pressure-wave} shows the simulated propagation of the wave in radial direction under different parametrizations of the method, see Table \ref{tab:projection-paramters}. The red line displays a circle with radius $t_{sim} - T/2 = 14$, as a reference for the maximum of the wave front\footnote{The plots in Fig. \ref{sim-pressure-wave} represent the discrete, minimal efforts $\tilde{e}_i^p$ \emph{in the nodes} of the mesh.} at time $t_{sim}$, based on the exact solution. 
The parameter sets in Table \ref{tab:projection-paramters} represent different weightings of the $2$-simplices in the propagation direction to compute $\tilde f^p_i$, see Fig. \ref{fig:illustration-Pfp}. For parameter set \#1 (equal weights $1/3$ in the definition of $\tilde f^p_i$ associated to a nodal effort $\tilde e^p_i$), the propagation of the wave front in the effort variable $e^p$ is reproduced in a completely unsatisfactory manner. Parameter set \#2 leads also to undesired dispersion. Moreover, the quarter circle shape of the wave is perturbed, which is due to the non-isotropic mesh and the inadequate parametrization. Parameter set \#3 shows less dispersion and parametrization \#4 reproduces appropriately the circular wave front

A direct explanation of the unsatisfactory behaviour of the numerical solutions \#1 and \#2 can be found by studying the definition of the matrix $\vec P_{fq}^\perp$, which is visualized in the upper drawings of Fig. \ref{fig:q-tilde} for the elementary example. Consider first the parametrizations \#3 and \#4 in Table \ref{tab:projection-paramters}. With 
\begin{equation}
    \label{eq:condition-delta-epsilon}
    \textrm{sgn}(\delta_I) = -\textrm{sgn}(\delta_{I\!I}) \qquad \text{and} \qquad \textrm{sgn}(\varepsilon_I) = -\textrm{sgn}(\varepsilon_{I\!I}),
\end{equation}
the rotational parts in the definition of the discrete states $\tilde{\vec q}$ are composed of discrete rotations of $\vec q$ \emph{in the same sense}. This is not the case for parametrizations \#1 and \#2, which is a hint that reasonable parameter sets for the numerical approximation of hyperbolic systems should satisfy condition \eqref{eq:condition-delta-epsilon}, or, equivalently, $\alpha_I - \beta_I \lessgtr \frac{1}{2}$ and at the same time $\alpha_{I\!I} - \beta_{I\!I} \gtrless \frac{1}{2}$.

Note that all four simulation models are \emph{stable by construction}. In the following section, we discuss the quality of the numerical scheme in terms of the eigenvalue approximation. The 1D case allows in a straightforward manner to implement negative values for the mapping parameter and thereby enforce \emph{upwinding} in the numerical solution.

\begin{figure}[h]
    \centering
    \includegraphics[scale=0.75]{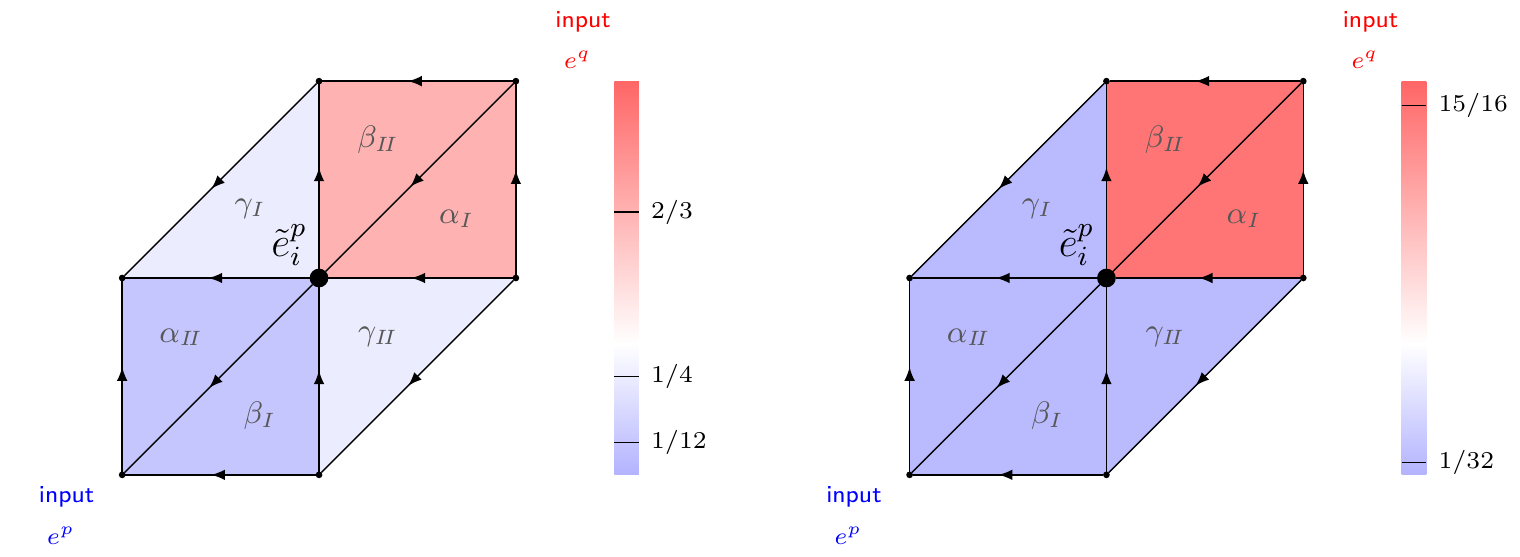}
    \caption{Illustration of $2$-simplices and weights that contribute to the definition of the discrete state $\tilde p_i$ for the parametrizations \# 3 (left) and \#4 (right). Note that information of the conserved quantity $p$, which is directly influenced by the boundary input effort $e^q$ in the upper right corner, is preferred for the computation of the node effort $\tilde e^p_i$ (``upwinding'').}
    \label{fig:illustration-Pfp}
\end{figure}

\subsection{Eigenvalue approximation for the 1D wave equation}

\label{subsec:spectral-approx}
For a short analysis of the spectral approximation properties, we consider the 1D wave equation on a domain $\Omega = (0,1)$. The degrees of the differential forms for both conserved quantities are $p=q=1$, the exponent in the canonical differential operator of Eq. \eqref{eq:stokes-dirac-interconnection} becomes $r=pq+1=2$. With the quadratic Hamiltonian $H = \int_0^1 (\frac{1}{2} p \wedge \ast p + \frac{1}{2} q \wedge \ast q)$, the PH representation with boundary efforts and flows as defined in Eq. \eqref{eq:boundary-flows-efforts-1D}, is
\begin{equation}
	\bmat{\partial_t p\\ \partial_t q}
		=
	\bmat{0 & - \mathrm{d}\\ - \mathrm{d} & 0} 
    \bmat{e^p\\ e^q}, \qquad 
    \vec e^\partial = \bmat{ e^q(0)\\ -e^p(L)}, \quad
    \vec f^\partial = \bmat{ e^p(0)\\ e^q(L)}, \qquad
    e^p = \ast p, \quad e^q = \ast q.
\end{equation} 
The system corresponds to the transmission line model in Example \ref{exa:1D-transmission-line} with length $L=1$ and $l(z)=c(z)=1$. We consider two identical grids with $N+1$ nodes and $N$ edges for both conserved quantities, with the corresponding Whitney node and edge forms to approximate the spatial distribution of effort $0$-forms and flow $1$-forms. According to the choice of boundary efforts (= boundary inputs), we set the boundary input matrices (trace matrices)
\begin{equation}
    \vec T_q = \bmat{1 & 0 & \ldots & 0} \in \IR^{1\times (N+1)}, \qquad 
    \hat {\vec T}_p = \bmat{0 & \ldots & 0 & - 1} \in \IR^{1\times (N+1)}.
\end{equation}
Together with the effort maps 
\begin{equation}
    \vec P_{eq} = \bmat{\vec 0_{N\times 1} & \vec I_N}, \qquad
    \vec P_{ep} = \bmat{\vec I_N & \vec 0_{N\times 1}},
\end{equation}
all effort degrees of freedom are assigned either the roles of \emph{inputs} $e^b = \vec T_q \vec e^q \in \IR$ and $\hat e^b = \hat{\vec T}_p \vec e^p \in \IR$ or \emph{interior} discrete efforts $\tilde{\vec e}^q = \vec P_{eq} \vec e^q \in \IR^N$ and $\tilde{\vec e}^p = \vec P_{ep} \vec e^p \in \IR^N$. With the discrete derivative matrices (or co-incidence matrices)
\begin{equation}
    \vec d_p = \vec d_q = \bmat{-1 & 1 & & \\ & \ddots & \ddots & \\ & & -1 & 1} \in \IR^{N\times(N+1)},
\end{equation}
we obtain 
\begin{equation}
    \vec P_{fq} = \vec P_{fp}^T = \bmat{1-\alpha & \alpha & & \\  & \ddots & \ddots & \\ & & 1-\alpha & \alpha \\ & & & 1-\alpha} \in \IR^{N\times N}
\end{equation}
as flow maps and 
\begin{equation}
    \vec S_p = \bmat{1-\alpha & \alpha & 0 & \ldots & 0} \in \IR^{1\times (N+1)}, \qquad
    \hat{\vec S}_q = \bmat{0 & \ldots & 0 & \alpha & 1-\alpha} \in \IR^{1\times (N+1)}
\end{equation}
as boundary output matrices, which verify the matrix equation \eqref{c04-eq:charact-projections} and define $\tilde{\vec f}^q = \vec P_{fq} \vec f^q \in \IR^n$, $\tilde{\vec f}^p = \vec P_{fp} \vec f^p \in \IR^n$ and $f^b = \vec S_p \vec e^p \in \IR$, $\hat f^b = \hat{\vec S}_q \vec e^q \in \IR$. The consistent approximation of the constitutive equations, which takes into account the definition of co-states and the mapping of the discrete states, is expressed by the diagonal Hodge matrices (in accordance to Section \ref{subsubsec:discrete-constitutive-eq} for the 2D case)
\begin{equation}
    \label{eq:Hodge-matrices}
    \vec Q_p = \frac{1}{h} \textrm{diag} \{ \frac{1}{1-\alpha}, 1, \ldots, 1  \}, 
    \qquad 
    \vec Q_q = \frac{1}{h} \textrm{diag} \{ 1, \ldots, 1 , \frac{1}{1-\alpha} \}.
\end{equation}

\begin{figure}
    \centering
    \includegraphics[scale=0.8]{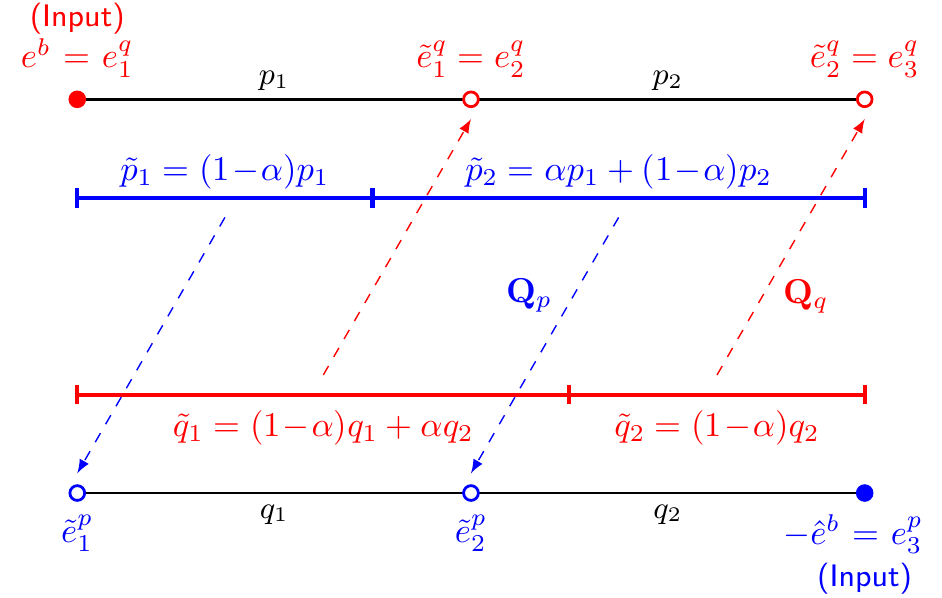}   \quad
    \includegraphics[scale=0.8]{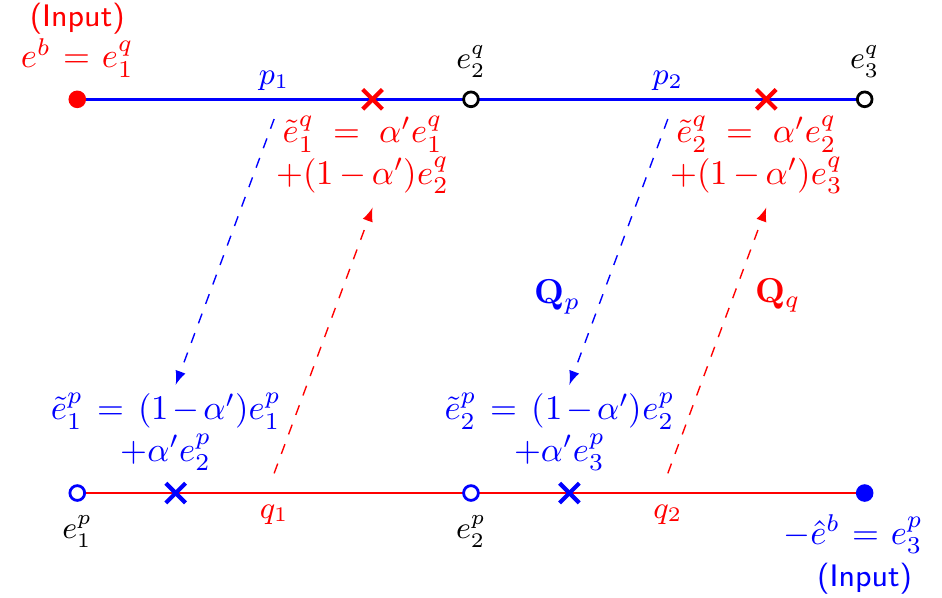}   
    \caption{Illustration of the difference between our approach and the method according to \cite{golo2004hamiltonian}. In our approach, the discrete efforts at the \emph{interior} nodes are computed -- via appropriate discrete Hodge matrices -- based on \emph{convex sums} of the original discrete states (left sketch). Following \cite{golo2004hamiltonian}, the original discrete states remain unchanged, but the \emph{co-states} are computed as convex sums of the node efforts (right sketch). For $\alpha = 0$ and $\alpha'=0$ both methods coincide. Then, for example, the co-state $\tilde e_1^q = e_2^q$ is determined based on the original discrete state $q_1$ in both cases.}
    \label{fig:upwind}
\end{figure}

We compare the results of our method with those obtained with the approach in \cite{golo2004hamiltonian}, where mapping the \emph{efforts} at the boundary nodes of each discretization interval using a parameter\footnote{We use a prime to distinguish from the $\alpha$ in our method.} $\alpha'$ yields non-degenerate power pairings and a PH model in state space form. The \emph{strong} compatibility conditions, which restrict the parameter value to $\alpha'=\frac{1}{2}$ for the case of lowest order Whitney forms in the original work, can be relaxed by a weak formulation of the problem. In contrast to our method, the state space models according to \cite{golo2004hamiltonian} feature a \emph{direct feedthrough}.\footnote{The exception with zero feedthrough matrix is $\alpha' = 0$, which corresponds to $\alpha = 0$ in our approach. With these parameter values, both methods produce models that coincide with those obtained from discrete modeling / finite volumes on regularly staggered grids \cite{seslija2014explicit}, \cite{kotyczka2016finite}.} The fundamental difference between both approaches is illustrated by the sketches in Fig. \ref{fig:upwind} and the explanation below. The discrete Hodge matrices according to \cite{golo2004hamiltonian} are $\vec Q_p = \vec Q_q = \frac{1}{h} \textrm{diag} \{ 1, \ldots, 1  \}$. For $\alpha < \frac{1}{2}$ and $\alpha' < \frac{1}{2}$, the state information from the directions in which the associated effort variables \emph{are imposed} as boundary inputs, obtains a higher weight. This type of \emph{upwinding} leads to a very good approximation of the eigenvalues for values close to zero of $\alpha$ and $\alpha'$.

We consider the spectrum of the canonical differential operator of the Stokes-Dirac structure under homogeneous Dirichlet boundary conditions on the efforts (Neumann-Dirichlet conditions for the PDE in second order form). The exact eigenvalues are $\pm \frac{2k-1}{2} \pi i$, $k=1,2,3,\ldots$, see \cite{hamroun2009approche}. As the structure-preserving discretization is conservative, also the approximate eigenvalues have zero real parts. We display in Table \ref{tab:spectrum} the imaginary parts for different values of the \emph{flow} mapping parameter $\alpha$. Table \ref{tab:spectrum-golo} shows the corresponding values for the structure-preserving discretization according to \cite{golo2004hamiltonian} with different \emph{effort} mapping parameters $\alpha'$. The relative errors for the first, 5th and 20th eigenvalue are plotted in the diagrams of Figs. \ref{fig:error-plots-our-method} and \ref{fig:error-plots-golo}.

For all displayed parametrizations around $\alpha = \alpha' = 0$, the order of the first eigenvalue approximation error is $\mathcal O(h)$ with $h = \frac{1}{N}$, see the left diagrams in Figs. \ref{fig:error-plots-our-method} and \ref{fig:error-plots-golo}. This is in accordance with the consistency order $1$ for the \emph{non-centered} approximation of the node efforts (see \cite{kotyczka2017discrete} for the discussion from the finite volumes point of view). 
We observe that for the parametrizations $\alpha = - \frac{1}{12}$ and $\alpha' = \frac{1}{12}$, the approximation quality of the higher eigenvalues is improved. The result of this \emph{upwinding} compared to the situation $\alpha=0$/$\alpha'=0$ is a remarkable improvement of the solution of the boundary value problem\footnote{The same holds for the initial value problem, which is not illustrated here.} as shown in the previous section for the 2D case. Note that the same effect can be achieved if in the finite volume approach on regularly staggered grids \cite{kotyczka2016finite} (which corresponds to $\alpha=0$/$\alpha'=0$), the control volumes to compute the numerical fluxes are shifted. Tables \ref{tab:spectrum} and \ref{tab:spectrum-golo} as well as Figures \ref{fig:error-plots-our-method} and \ref{fig:error-plots-golo} show a very similar evolution of the eigenvalues under grid refinement. Note however, that our approach, in contrast to \cite{golo2004hamiltonian}, produces no structural feedthrough, which is \emph{appropriate} for hyperbolic systems\footnote{The feedthrough, together with the over-estimation of the highest eigenvalues for $\alpha' \rightarrow 0.5$, fits to the good results the method according to \cite{golo2004hamiltonian} achieves for the discretization of \emph{parabolic} systems \cite{baaiu2009structure}, where the instantaneous propagation of information must be approximated.}. Moreover, as has been shown in the previous section, the extension to 2D (and prospectively 3D) is straightforward.

\begin{table}
\centering
\caption{Eigenvalue imaginary parts for the discretized 1D wave equation with different flow mapping parameters $\alpha$ and grid sizes $N$, compared to the exact values.}
{\small
\begin{tabular}{c|c|ccc|ccc|ccc}
    \hline
    $k$    &  Exact   & \multicolumn{3}{c|}{$\alpha=-1/12$} & \multicolumn{3}{c|}{$\alpha=0$} & \multicolumn{3}{c}{$\alpha=1/6$} \\ 
    & & $N=20$ & $N=40$ & $N=80$  & $N=20$ & $N=40$ & $N=80$  & $N=20$ & $N=40$ & $N=80$  \\ \hline
  1 & 1.5708 & 1.5263 & 1.5482 & 1.5594 & 1.5321 & 1.5513 & 1.5610 & 1.5440 & 1.5576 & 1.5642 \\ 
  2 & 4.7124 & 4.5798 & 4.6449 & 4.6783 & 4.5873 & 4.6516 & 4.6825 & 4.6074 & 4.6663 & 4.6910 \\ 
  3 & 7.8540 & 7.6352 & 7.7422 & 7.7974 & 7.6156 & 7.7449 & 7.8021 & 7.5975 & 7.7562 & 7.8131 \\ 
  4 & 10.996 & 10.692 & 10.840 & 10.917 & 10.599 & 10.827 & 10.919 & 10.468 & 10.815 & 10.927 \\ 
  5 & 14.137 & 13.746 & 13.939 & 14.036 & 13.521 & 13.892 & 14.031 & 13.176 & 13.830 & 14.030 \\ 
 10 & 29.845 & 28.593 & 29.428 & 29.641 & 26.613 & 28.814 & 29.490 & 23.192 & 27.852 & 29.269 \\ 
 20 & 61.261 & 46.492 & 59.316 & 60.828 & 39.883 & 54.899 & 59.422 & 26.831 & 47.252 & 57.198 \\ 
 40 & 124.09 & --     & 93.244 & 120.76 & --     & 79.940 & 111.47 & --     & 53.664 & 95.338 \\ 
 80 & 249.76 & --     & --     & 186.62 & --     & --     & 159.97 & --     & --     & 107.33 \\ 

    \hline
\end{tabular}
}
\label{tab:spectrum}
\end{table}
\begin{figure}
    \centering
    \includegraphics[scale=0.52]{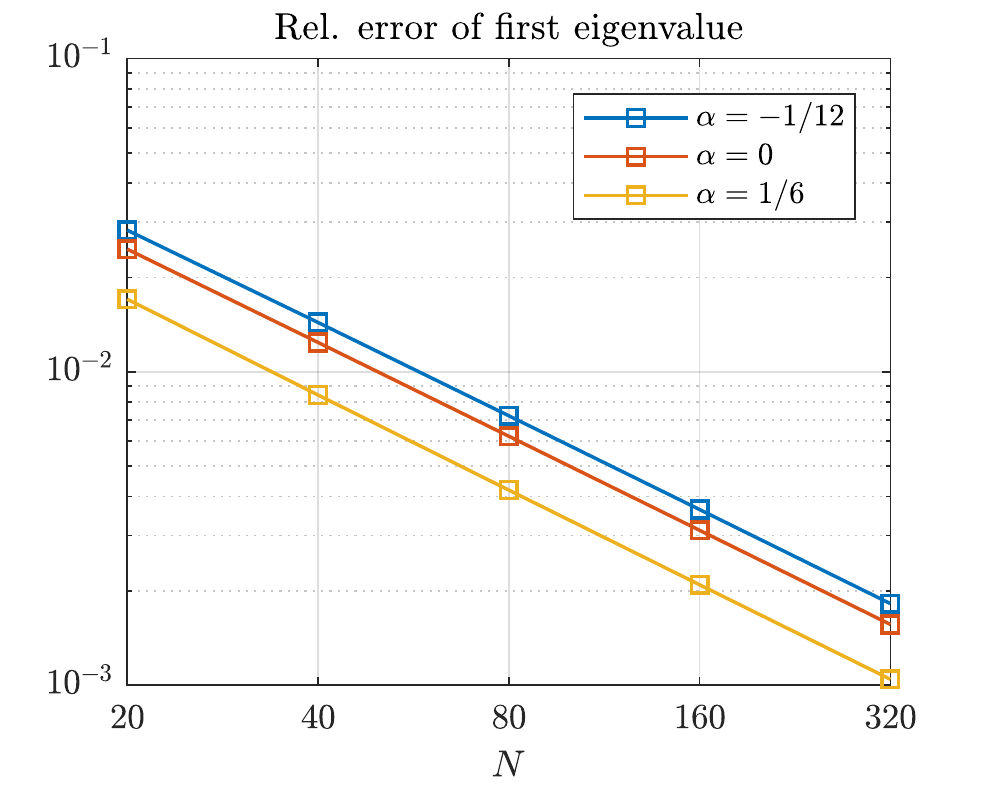}   
    \includegraphics[scale=0.52]{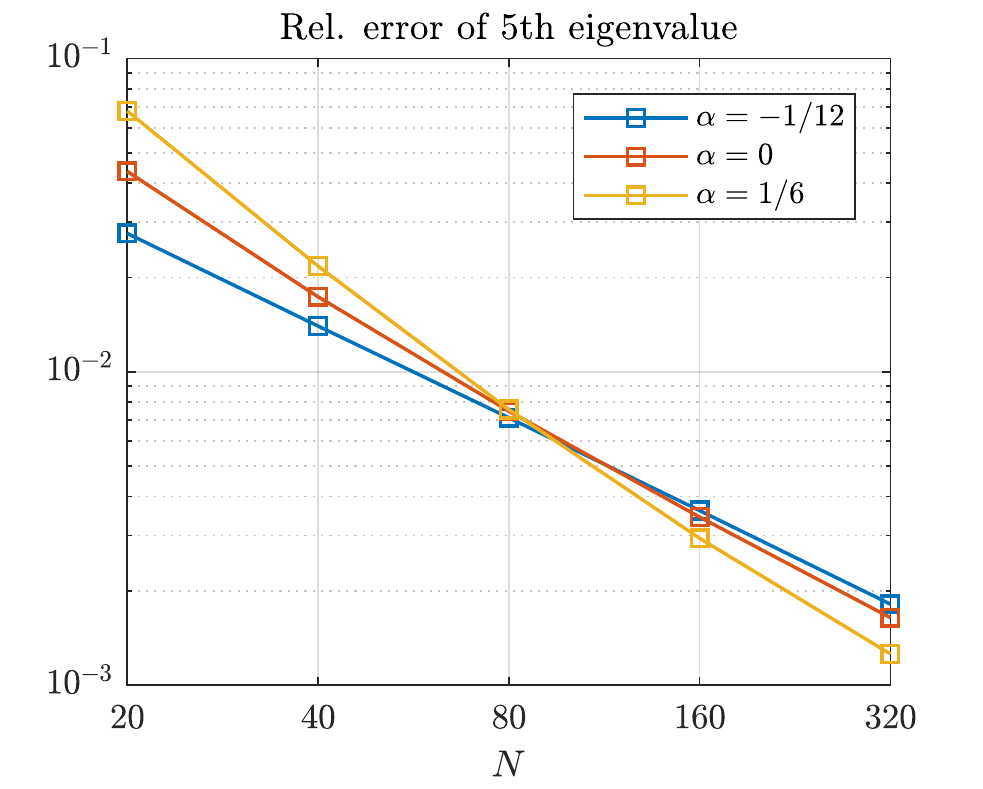}   
    \includegraphics[scale=0.52]{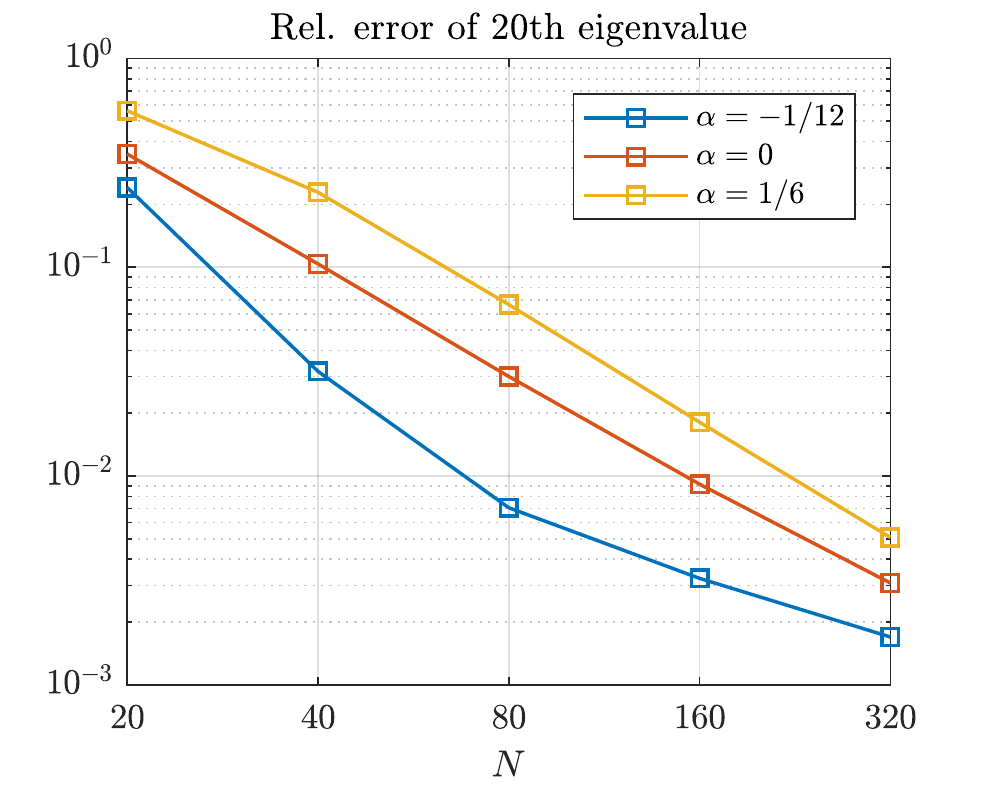}   
    \caption{Magnitude of the approximation error over different grid sizes for the 1st, 5th and 20th eigenvalue of the canonical system operator. Upwinding improves the approximation of the higher eigenvalues.}
    \label{fig:error-plots-our-method}
\end{figure}

\begin{rmk}
    In the presented context, it makes sense to talk of \emph{upwinding} for values of $\alpha < \frac{1}{2}$. With the parametrization $\alpha = \frac{1}{2}$, the co-states are computed based on the equally weighted information of the states to the left and to the right of the considered node. This \emph{centered} evaluation of the discrete constitutive equations leads to order 2 of consistency and the approximation of the eigenvalues. For $\frac{1}{2} < \alpha <1$, the numerical approximation is still in PH form, without numerical dissipation\footnote{For $\alpha=1$, the first or last element respectively of the Hodge matrices according \eqref{eq:Hodge-matrices} becomes singular.}. In the present 1D example, this \emph{upwinding in wrong direction} manifests itself by control input vectors $\vec B_p \in \IR^{2N}$ and $\vec B_q \in \IR^{2N}$ whose second (second last) element have bigger magnitude than the first (the last) element. At the same time, a pair of purely imaginary eigenvalues, which tends to $0$ under grid refinement, is introduced.
\end{rmk}

\section{Conclusions}
\label{sec:conclusions}
We introduced the weak form of the Stokes-Dirac structure with a segmented boundary, on which the causality of the port variables (the assignment as system in- or output) alternates. This Stokes-Dirac structure is the underlying geometric structure to represent power continuity in a port-Hamiltonian distributed parameter system. On the example of a system of two conservation laws with canonical interdomain coupling, we described the mixed Galerkin discretization of the Stokes-Dirac structure in a general way. To obtain finite-dimensional approximate models in PH form with the prescribed boundary inputs -- as basis for the interconnection of multi-physics models, control design and simulation -- we proposed power-preserving mappings on the space of discrete effort and flow variables. These maps allow to define non-degenerate duality pairings, leading to finite-dimensional approximate Dirac structures on the minimal discrete bond space. The Dirac structures admit several representations, one of them being an explicit input-output-representation. Port-Hamiltonian state space models are obtained, if dynamics is added and  the constitutive equations are approximated consistently. On the example of Whitney finite elements we demonstrated the discretization procedure and gave interpretations of the resulting discretization schemes.

\begin{table}
\centering
\caption{A comparable approximation of the eigenvalues can be obtained using the method presented in \cite{golo2004hamiltonian} in a \emph{weak} formulation, which permits to choose parameters other than $\alpha'= \frac{1}{2}$. Note, however, that the resulting state space models, in contrast to our approach, feature a \emph{structural feedthrough}.}
{\small
\begin{tabular}{c|c|ccc|ccc|ccc}
    \hline
    $k$    &  Exact   & \multicolumn{3}{c|}{$\alpha'=1/12$} & \multicolumn{3}{c|}{$\alpha'=0$} & \multicolumn{3}{c}{$\alpha'=-1/6$} \\ 
    & & $N=20$ & $N=40$ & $N=80$  & $N=20$ & $N=40$ & $N=80$  & $N=20$ & $N=40$ & $N=80$  \\ \hline
  1 & 1.5708 & 1.5387 & 1.5546 & 1.5627 & 1.5321 & 1.5513 & 1.5610 & 1.5189 & 1.5447 & 1.5577 \\ 
  2 & 4.7124 & 4.6152 & 4.6636 & 4.6879 & 4.5873 & 4.6516 & 4.6825 & 4.5283 & 4.6266 & 4.6712 \\ 
  3 & 7.8540 & 7.6888 & 7.7719 & 7.8130 & 7.6156 & 7.7449 & 7.8021 & 7.4544 & 7.6858 & 7.7789 \\ 
  4 & 10.996 & 10.757 & 10.879 & 10.938 & 10.599 & 10.827 & 10.919 & 10.250 & 10.708 & 10.877 \\ 
  5 & 14.137 & 13.816 & 13.985 & 14.062 & 13.521 & 13.892 & 14.031 & 12.875 & 13.679 & 13.961 \\ 
 10 & 29.845 & 28.700 & 29.459 & 29.675 & 26.613 & 28.814 & 29.490 & 22.886 & 27.377 & 29.052 \\ 
 20 & 61.261 & 47.800 & 59.416 & 60.773 & 39.883 & 54.899 & 59.422 & 29.950 & 46.903 & 56.384 \\ 
 40 & 124.09 & --     & 95.897 & 120.87 & --     & 79.940 & 111.47 & --     & 59.974 & 94.912 \\ 
 80 & 249.76 & --     & --     & 191.95 & --     & --     & 159.97 & --     & --     & 119.99 \\ 
 \hline
\end{tabular}
}
\label{tab:spectrum-golo}
\end{table}

\begin{figure}
    \centering
    \includegraphics[scale=0.52]{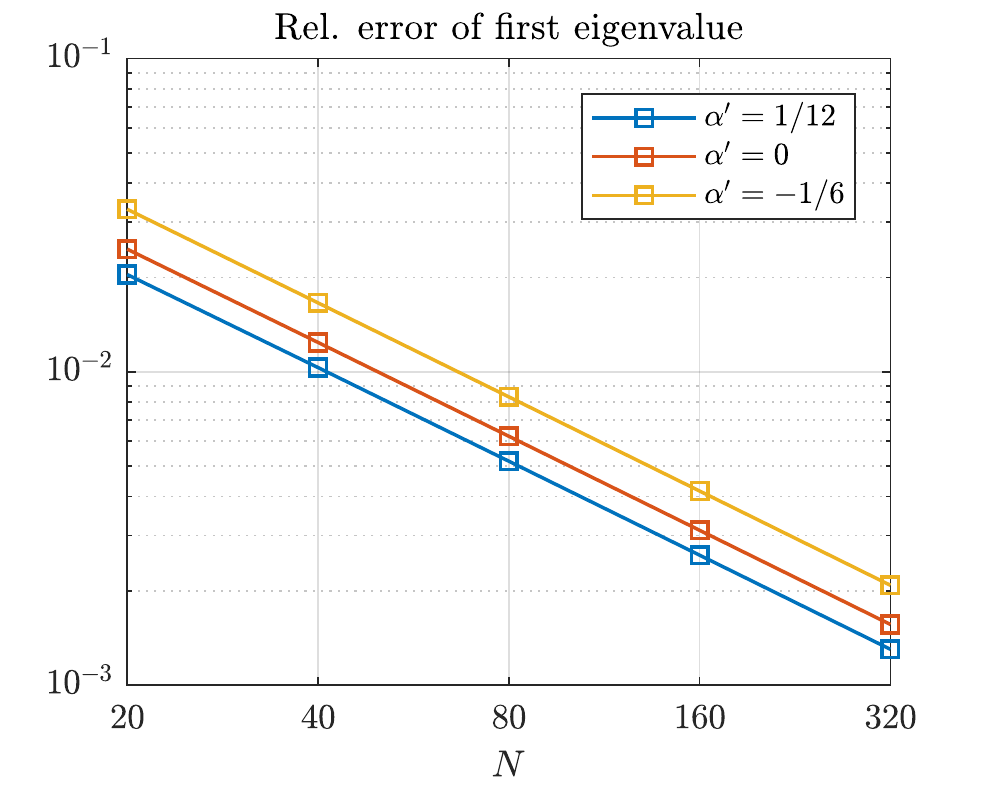}   
    \includegraphics[scale=0.52]{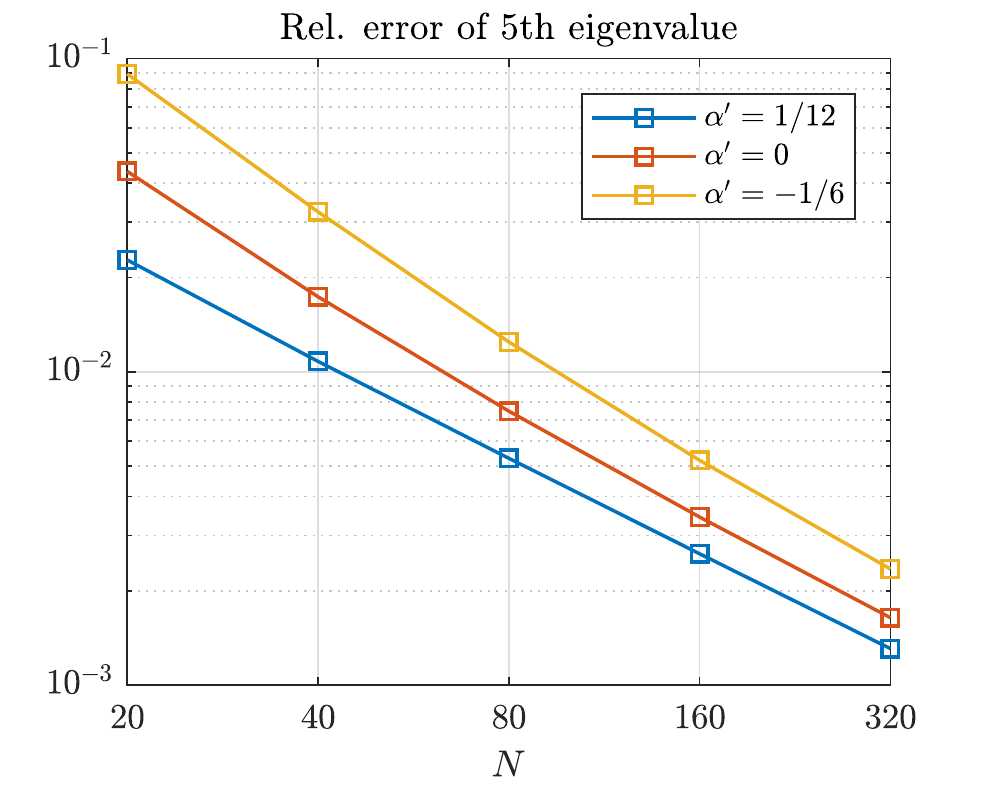}   
    \includegraphics[scale=0.52]{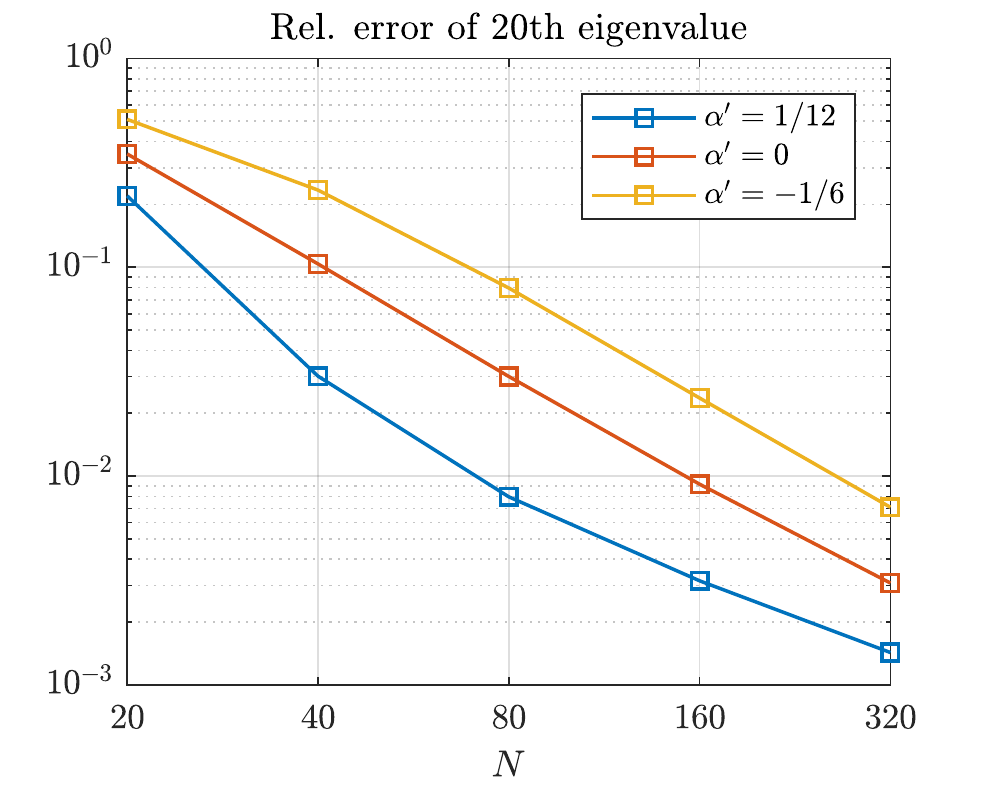}   
    \caption{Magnitude of the approximation error over different grid sizes for the 1st, 5th and 20th eigenvalue of the canonical system operator with the method according to \cite{golo2004hamiltonian}.}
    \label{fig:error-plots-golo}
\end{figure}

The proposed method is, to the best of our knowledge, the first method which allows for a structure-preserving discretization of PH distributed parameter systems in more than one spatial dimension with a systematic treatment of different boundary inputs and the possibility to tune the discretized models between centered schemes and upwinding.  The proposed family of approximation Dirac structures avoids a direct feedthrough in the state space model and the over-estimation of higher frequencies in the approximate spectrum, which is the case for the method presented in \cite{golo2004hamiltonian}, where the efforts instead of the flow degrees of freedom are mapped. The weak form of the Stokes-Dirac structure is the key feature that allows to include additional effects such as dissipation or diffusion or, more generally, to tackle the discretization of PH systems with general and higher order interconnection operators and distributed inputs.

An important difference of our work to related works like \cite{seslija2014explicit}, \cite{farle2013porta}, \cite{hiemstra2014high}, where either dual grids are used \emph{a priori} or at least one conservation law contains the Hodge star or the co-differential, is that our initial discretization is based on a \emph{metric-independent} formulation of the conservation laws.  We approximate all differential forms in the same conforming subspaces depending on their degree (i.\,e. on the same mesh in FE), which has the advantage that boundary variables are defined directly \emph{on} $\partial \Omega$, without having to cope with an eventual grid shift. To obtain an explicit state space model, however, we need -- \emph{no free lunch} -- the power-preserving mappings. These, in turn, give us degrees of freedom to tune the resulting numerical method.

Current and future work concerns the application of the method to the PH representations of systems including heat and mass diffusion phenomena, which share similar Stokes-Dirac structures, as well as coupled heat and mass transport phenomena in non-homogeneous media such as catalytic foams.  Moreover, we want to analyse the approach when applied to PH systems with non-canonical system operators (containing e.\,g. higher order derivatives). In this context, we are interested in the reasonable choice of design parameters in order to adapt the discretization scheme to the physical nature of the system (e.\,g. to account for the ratio between convection and diffusion). This aspect is closely related to the analysis of system-theoretic properties of the discretized models in view of control design. Further important issues are the implementation of the approach in existing finite element tools like FEniCS \cite{alnaes2015fenics} and the use of approximation spaces with higher degree \cite{rapetti2009whitney}, \cite{arnold2013spaces}. We intend to include the discretization of the nonlinear constitutive relations for the 2D shallow water equations in our \emph{open} models and clarify the links with recent work on geometric mixed finite elements like \cite{cotter2012mixed}, \cite{cotter2014finite}, where in- and outputs are not explicitly taken into account, and upwinding in differential forms as presented in \cite{christiansen2013upwinding}.

\section*{Acknowledgement}
This paper was written during the temporary leave (September 2015 -- August 2017) of the first author from Technical University of Munich, Chair of Automatic Control. The work was supported by the European Union's Horizon 2020 research and innovation programme (Marie Sk{\l}odowska-Curie Individual Fellowship) under grant agreement No 655204, EasyEBC, and Agence Nationale de la Recherche (ANR)/Deutsche Forschungsgemeinschaft (DFG) project INFIDHEM, ID ANR-16-CE92-0028.

\small
\bibliographystyle{ieeetr}
\bibliography{references-jcp-2018}

\end{document}